\documentclass[letterpaper,reqno,11pt,twoside]{amsart}
\usepackage{amsmath,amsthm,amsfonts,amssymb,euscript,mathrsfs,graphics,color,latexsym,marginnote}
\usepackage{stmaryrd}
\usepackage[dvips]{graphicx}
\usepackage[left=1 in, right=1 in,top=1 in, bottom=1 in]{geometry}
\usepackage{hyperref}
\usepackage[toc,page]{appendix}
\usepackage{relsize}
\usepackage[shortlabels]{enumitem}
\usepackage[all]{xy}
\usepackage{float}

\usepackage{xargs}
\usepackage[dvipsnames]{xcolor}

\usepackage{hyperref}
\hypersetup{colorlinks=true, pdfstartview=FitV,linkcolor=blue!70!black,citecolor=red!70!black, urlcolor=green!60!black}
\definecolor{labelkey}{rgb}{0.6,0,0}

\usepackage[colorinlistoftodos,prependcaption,textsize=small]{todonotes}
\newcommandx{\change}[2][1=]{\todo[#1]{#2}}
\newcommandx{\unsure}[2][1=]{\todo[linecolor=red,backgroundcolor=red!25,bordercolor=red,#1]{#2}}
\newcommandx{\rmk}[2][1=]{\todo[linecolor=blue,backgroundcolor=blue!25,bordercolor=blue,#1]{#2}}
\newcommandx{\info}[2][1=]{\todo[linecolor=OliveGreen,backgroundcolor=OliveGreen!25,bordercolor=OliveGreen,#1]{#2}}
\newcommandx{\improvement}[2][1=]{\todo[linecolor=Plum,backgroundcolor=Plum!25,bordercolor=Plum,#1]{#2}}
\newcommandx{\thiswillnotshow}[2][1=]{\todo[disable,#1]{#2}}

\makeatletter
\renewcommand \theequation {%
\ifnum \c@section>\z@ \@arabic\c@section.%
\fi\@arabic\c@equation} \@addtoreset{equation}{section}
\@namedef{subjclassname@2020}{2020 Mathematics Subject Classification}
\makeatother

\newtheorem{theorem}{Theorem}[section]

\newtheorem{lemma}[theorem]{Lemma}
\newtheorem{proposition}[theorem]{Proposition}

\theoremstyle{definition}
\newtheorem{definition}{Definition}[section]
\theoremstyle{remark}
\newtheorem{remark}{Remark}[section]

\def\XXint#1#2#3{{\setbox0=\hbox{$#1{#2#3}{\int}$ }
\vcenter{\hbox{$#2#3$ }}\kern-.6\wd0}}

\let\f=\frac
\let\p=\psi
\let\om=\omega

\let\Om=\Omega

\let\th=T
\let\pa=\partial

\def\dive{\mathop{\rm div}\nolimits}

\begin{document}

\title{Dynamics of contact points in 2D Boussinesq flow}

\author[Y. Zheng]{Yunrui Zheng}
\address[Y. Zheng]{
\newline\indent School of Mathematics, Shandong University}
\email{yunrui\_zheng@sdu.edu.cn}
\thanks{Y. Zheng was supported by National Natural Science Foundation of China with Grant No.  11901350 and No. 12371172.}

\date{}

\subjclass[2020]{Primary 35Q30, 35R35, 76D45; Secondary 35B40, 76E17, 76E99}

\keywords{contact point, Boussinesq equations, surface tension}

\maketitle

\begin{abstract}
We consider the evolution of contact lines for thermal convection of viscous fluids in a 2D open-top vessel. The domain is bounded above by a free moving boundary and otherwise by the solid wall of a vessel. The dynamics of the fluid are governed by the incompressible Boussinesq approximation under the influence of gravity, and the interface between fluid and air is under the effect of capillary forces. Motivated by energy-dissipation structure in [Guo-Tice, J. Eur. Math. Soc, 2024], we develop global well posedness theory in the framework of nonlinear energy methods for the initial data sufficiently close to equilibrium. Moreover, the solutions decay to equilibrium at an exponential rate. Our methods are mainly based on the construction of solutions to convected heat equation and a priori estimates of a geometric formulation of the Boussinesq equations.
\end{abstract}

\pagestyle{myheadings} \thispagestyle{plain} \markboth{ZHENG}{WELL-POSEDNESS OF CONTACT POINTS PROBLEM}

\setcounter{tocdepth}{1}
\tableofcontents

\section{Introduction}

\subsection{Formulation of the Problem}

We now consider the dynamical cooling process of hot water or coffee contained in 2D cup with adiabatic boundary. A $2$D open top vessel is treated as a bounded, connected open set $\mathcal{V}\subseteq\mathbb{R}^2$ which consists of two disjoint sections, i.e., $\mathcal{V}=\mathcal{V}_{top}\cup\mathcal{V}_{bot}$.
We refer to Figure 1 for an example.


\begin{figure}[H]
\begin{minipage}[t]{0.5\linewidth}
\centering
\scalebox{1.0}{\begin{tikzpicture}
 \filldraw [white!15] (-1, 5)rectangle (1, 1);
\draw  (-1, 1)--(-1, 5);
 \draw  (1, 1)--(1, 5);
  \node at(0, 3){$\mathcal{V}_{top}$};
  \filldraw [white!10, draw=black] (-1, 1) arc(90: 270: 1)-- (-1, -1)--(1, -1)-- (1, -1) arc(270: 450: 1);
  \node at(0,0.5) {$\mathcal{V}_{bot}$};
  \draw[dashed]  (-1,1)--(1, 1);

\end{tikzpicture}}
\caption{A vessel $\mathcal{V}$.}
\end{minipage}%
\begin{minipage}[t]{0.5\linewidth}
\centering
\scalebox{1.0}{\begin{tikzpicture}
  \draw  (6, 1)--(6, 5);
 \draw  (8, 1)--(8, 5);
 \filldraw [blue!10, draw=black] (6, 4)--(6, 1)--(6, 1) arc(90: 270: 1)--(6, -1)--(8, -1)arc(270: 450: 1)--(8, 1)--(8, 4)--(8, 3.8)..controls(7.4, 4.2)and (6.7, 3.9)..(6,3.6);
 \node at(7, 2){$\Om(t)$};
 \node at(7,4){$\Sigma(t)$};
 \node at(8.5, 2){$\Sigma_s(t)$};
\end{tikzpicture}}
\caption{The domain $\Om(t)$.}
\end{minipage}
\end{figure}

Assume that the ``top" part $\mathcal{V}_{top}$ is a rectangle channel
$
\mathcal{V}_{top}:=\mathcal{V}\cap\mathbb{R}^2_{+}=\{y\in\mathbb{R}^2: -\ell<y_1<\ell, 0\le y_2<L\}
$
for some $\ell$, $L>0$, where $\mathbb{R}^2_{+}$ is the upper half plane $\mathbb{R}^2_{+}=\{y\in\mathbb{R}^2: y_2\ge0\}$. Similarly, we write the ``bottom" part as
$
\mathcal{V}_{bot}:=\mathcal{V}\cap\mathbb{R}^2_{-}=\mathcal{V}\cap\{y\in\mathbb{R}^2: y_2\le0\}.
$
Clearly, $I=(-\ell, \ell)\times\{0\}$. In addition, we also assume that the boundary $\partial\mathcal{V}$ of $\mathcal{V}$ is $C^2$ away from the points $(\pm\ell, L)$.

Now we consider a viscous incompressible fluid filling $\mathcal{V}_{bot}$ entirely and $\mathcal{V}_{top}$ partially. More precisely, we assume that the fluid occupies the domain $\Om(t)$,
$
\Om(t):=\mathcal{V}_{bot}\cup\big\{y\in\mathbb{R}^2: -\ell<y_1<\ell,\ 0<y_2<\zeta(y_1,t)\big\},
$
where the free surface $\zeta(y_1,t)$ is assumed to be a graph of the function $\zeta: [-\ell, \ell]\times \mathbb{R}_{+}\rightarrow\mathbb{R}$ satisfying $0<\zeta(\pm\ell,t)\le L$ for all $t\in\mathbb{R}_+$, which means the fluid does not spill out of the top domain. For simplicity, we write the free surface as $\Sigma(t)=\{y_2=\zeta(y_1,t)\}$ and the interface between fluid and solid as $\Sigma_s(t)=\partial\Om(t)\backslash\Sigma(t)$.
We refer to Figure 2 for the description of domain.

For each $t>0$, the fluid is described by its velocity, pressure and temperature $(U,P,\Theta):\Om(t)\rightarrow\mathbb{R}^2\times\mathbb{R}$ which is governed by the incompressible Boussinesq approximation \cite{Chan}:
\begin{equation}\label{eq:Bouss}
  \left\{
  \begin{aligned}
    &\pa_t U+U\cdot\nabla U+\nabla P-\mu\Delta U=-g\Theta e_2 \quad & \text{in} &\ \Om(t),\\
    &\dive U=0 \quad & \text{in} &\ \Om(t),\\
    &\pa_t \Theta+U\cdot\nabla \Theta-k\Delta \Theta=0 \quad & \text{in} &\ \Om(t),
  \end{aligned}
  \right.
\end{equation}
where $\mu>0$ is the coefficient of viscosity, $k>0$ is the coefficient of heat conduction, $g>0$ is the acceleration of gravity and $P=\bar{P}+gy_2-P_{atm}$ by adjusting from the gravitational force and the constant atmospheric pressure $P_{atm}$ to the actual pressure $\bar{P}$.


From the physical viewpoint, the fluid are supposed to satisfy some boundary conditions. On the free surface, the fluid should be kinematic \cite{WL}:
\begin{equation}\label{bc:kie}
  \left\{
  \begin{aligned}
&S(P,U)\nu=g\zeta\nu-\sigma \mathcal{H}(\zeta)\nu \quad & \text{on} &\ \Sigma(t),\\
&\pa_t\zeta=U\cdot\nu \quad & \text{on} &\ \Sigma(t),
  \end{aligned}
  \right.
\end{equation}
where $S(P,U)$ is the viscous stress tensor
$S(P,U):=PI-\mu\mathbb{D}U$,
with $I$ the $2\times2$ identity matrix, and $\mathbb{D}U=\nabla U+\nabla^T U$ the symmetric gradient of $U$. Here, $
\mathcal{H}(\zeta)=\pa_1\left(\frac{\pa_1\zeta}{(1+|\pa_1\zeta|^2)^{1/2}}\right)$
is twice the mean curvature of the free surface. $\sigma = \sigma_1 - \sigma_2 \Theta$ is the coefficient of surface tension, with $\sigma_1>0$ and $\sigma_2 \in \mathbb{R}$.
On the fixed boundary, the fluid satisfies the Navier-slip condition:
\begin{equation}\label{bc:navier}
  \left\{
  \begin{aligned}
 &\big(S(P,U)\nu-\beta U\big)\cdot\tau=0 \quad & \text{on} &\ \Sigma_s(t),\\
 &U\cdot\nu=0 \quad & \text{on} &\ \Sigma_s(t).
  \end{aligned}
  \right.
\end{equation}
Without loss of generality, we assume that the temperature of atmosphere around the fluid is $0$. Due to the adiabatic solid boundary, we assume that $\Theta = 0$ on $\Sigma_s(t)$. The heat on the free surface $\Sigma(t)$ satisfies the Newton's cooling law due to the heat conduction between fluid and air:
\begin{equation}\label{bc:cool}
  k\nu\cdot\nabla \Theta + \Theta =0.
\end{equation}
The fluid on the contact points satisfies the condition:
\begin{equation}\label{bc:contact}
 \pa_t\zeta(\pm\ell,t)=\mathscr{V}\left([\![\gamma]\!]\mp\sigma\frac{\partial_1\zeta}{(1+|\partial_1\zeta|^2)^{1/2}}(\pm\ell,t)\right),
\end{equation}
where $\mathscr{V}:\mathbb{R}\rightarrow\mathbb{R}$ is the contact point velocity response function which is a $C^2$ increasing diffeomorphism satisfying $\mathscr{V}(0)=0$. $[\![\gamma]\!]:=\gamma_{sv}-\gamma_{sf}$ for $\gamma_{sv}$, $\gamma_{sf}\in\mathbb{R}$, where $\gamma_{sv}$, $\gamma_{sf}$ are measures of the free-energy per unit length with respect to the solid-vapor and solid-fluid intersection. We also assume that the Young relation
$|[\![\gamma]\!]|<\sigma$ (\cite{Yo}),
which is necessary for the equilibrium state we will consider in the following. For convenience, we introduce the inverse function $\mathscr{W}$ of $\mathscr{V}$ and rewrite \eqref{bc:contact} as
\begin{equation}\label{eq:contact law}
\mathscr{W}\big(\pa_t\zeta(\pm\ell,t)\big)=[\![\gamma]\!]\mp\sigma\frac{\pa_1\zeta}{(1+|\pa_1\zeta|^2)^{1/2}}(\pm\ell,t),
\end{equation}
which is the dynamic law at contact points.

The contact points in 2D or contact lines in 3D fluids, with temperature absent, has been studied for a long time, and we may refer to \cite[de Gennes]{dG} for an overview. There are so many papers dealing with  free boundary problems of contact points or contact lines, that we have to mention some works closely related to our background settings. For the contact lines in inviscid flow, we refer to \cite[T. de Poyferr\'e]{dG19} and \cite[Ming-Wang]{MW} for the breakthrough of water wave problem.

The analysis for Navier-slip conditions with contact angle can be referred to \cite[Ren-E]{RE} and \cite[Cox]{C1986}.
 When the contact angle is fixed by $\frac\pi2$, the well-posedness of 2D contact points problem in the Navier-Stokes flow is studied in \cite[Schweizer]{S01} and the 3D contact line problem with an additional periodic direction is studied in \cite[Bodea]{Bo06}.

For the non-stationary problem, the well-posedness and stability for 2D Navier-stokes flow are studied by \cite[Guo-Tice]{GT2020} and \cite[Guo-Tice-Wu-Zheng]{GTWZ2023}. There are also many results for the stationary problem of contact points or lines. The solvability of stationary Navier-Stokes problem for dynamic contact angles and moving contact lines is studied by \cite[Kr\"oner]{Kr87} and \cite[Socolosky]{Soco93}, respectively. The existence results of the Navier-Stokes equations for both static and dynamic contact points and lines in weighted H\"older spaces are proved by \cite[Solonnikov]{Sol95, Sol98}. The 3D contact lines for stationary Navier-Stokes equations with fixed angle $\frac{\pi}{2}$ is studied by \cite[Jin]{Jin05}. The local well-posedness and stability for the 2D contact points in the Stokes flow are studied by \cite[Guo-Tice]{GT18} and \cite[Tice-Zheng]{ZhT17}
The contact angles are also closely related to the droplet problem. The stability problem for droplets governed by the Stokes equations is studied in \cite[Tice-Wu]{TW21}.  The well-posedness for droplets governed by Darcy's law is studied in \cite[Kn\"upfer-Masmoudi]{KM15}. Recently, the well-posedness of the Stokes equations on a wedge in weighted Sobolve spaces with Navier-slip boundary conditions is studied by \cite[ Bravin-Gnann-Kn\"upfer-Masmoudi-Roodenburg-Sauer]{BGKMRS}.

To the authors' best knowledge, there is still no well-posedness and stability for the dynamic contact points problem with temperature at present. Our aim is to establish the first result for the global well-posedness for the 2D Boussinesq contact line problem. By referring to \cite[Guo-Tice]{GT18} or \cite[Finn]{Finn}, we known that \eqref{eq:Bouss}, under the condition of Young relation, admits an equilibrium for $U=0$, $P=P_0 $ a constant, $\Theta =0$ and $\zeta(y_1,t)=\zeta_0(y_1)$ satisfies
\begin{equation}\label{eq:equibrium}
P_0=g\zeta_0-\sigma_1\mathcal{H}(\zeta_0)\quad  \text{on} \ (-\ell,\ell),\quad
\sigma_1\frac{\pa_1\zeta_0}{\sqrt{1+|\pa_1\zeta_0|^2}}(\pm\ell)=\pm[\![\gamma]\!].
\end{equation}
This equilibrium would allow us to switch the free surface problem into a fixed domain one.
Let $\zeta_0\in C^\infty[-\ell,\ell]$ be the equilibrium surface given by \eqref{eq:equibrium}. We then define the equilibrium domain
$
\Om:=\mathcal{V}_{b}\cup\big\{x\in\mathbb{R}^2|-\ell<x_1<\ell, 0<x_2<\zeta_0(x_1)\big\} \subset\mathbb{R}^2
$.
The boundary $\pa\Om$ of the equilibrium $\Om$ is defined by
$
\pa\Om:=\Sigma\cup\Sigma_s$,
where
$
\Sigma:=\{x\in\mathbb{R}^2|-\ell<x_1<\ell, x_2=\zeta_0(x_1)\}$ and $ \Sigma_s=\pa\Om\backslash \Sigma$.
Here $\Sigma$ is the equilibrium free surface. Denote the corner angle $\om\in(0,\pi)$ of $\Om$ formed by the fluid and solid.
We assume that the function $\zeta(y_1,t)$ of free surface is the perturbation of $\zeta_0(y_1)$ as
$
\zeta(y_1,t)=\zeta_0(y_1)+\eta(y_1,t)$.
Let $\phi\in C^\infty(\mathbb{R})$ be such that $\phi(z)=0$ for $z\le\frac14\min\zeta_0$ and $\phi(z)=z$ for $z\ge\frac12\min\zeta_0$.
Now we define the mapping $\Phi: \Om\mapsto\Om(t)$ by
\begin{equation}\label{def:map}
\Phi(x_1,x_2,t):=\left(x_1, x_2+\frac{\phi(x_2)}{\zeta_0(x_1)}\bar{\eta}(x_1,x_2,t)\right)=(y_1,y_2)\in\Om(t),
\end{equation}
with $
\bar{\eta}(x_1,x_2,t):=\mathcal{P}E\eta(x_1,x_2-\zeta_0(x_1),t)$,
where $E: H^s(-\ell,\ell)\mapsto H^{s+\frac{1}{2}}(\mathbb{R})$ is a bounded extension operator for all $0\le s\le 3$ and $\mathcal{P}$ is the lower Poisson extension given by
$
\mathcal{P}f(x_1,x_2):=\int_{\mathbb{R}}\hat{f}(\xi)e^{2\pi|\xi| x_2}e^{2\pi ix_1\xi}\,\mathrm{d}\xi.
$
If $\eta$ is sufficiently small (in proper Sobolev spaces), the mapping $\Phi$ is a $C^1$ diffeomorphism of $\Om$ onto $\Om(t)$ that maps the components of $\pa\Om$ to the corresponding components of $\pa\Om(t)$. The Jacobian matrix $\nabla\Phi$ and the transform matrix $\mathcal{A}$ of $\Phi$ are:
\begin{equation}\label{eq:Jaccobian}
\nabla\Phi=\left(
\begin{array}{cc}
  1&0\\
  A&J
\end{array}
\right), \quad
\mathcal{A}=(\nabla\Phi)^{-T}=\left(
\begin{array}{cc}
  1&-AK\\
  0&K
\end{array}
\right),
\end{equation}
with entries
\begin{equation}\label{eq:components}
A:=W\pa_1\bar{\eta}-\frac{\phi}{\zeta_0^2}\pa_1\zeta_0\bar{\eta},\quad J:=1+\frac{\phi^\prime}{\zeta_0}\bar{\eta}+W\pa_2\bar{\eta},\quad K:=\frac{1}{J},\quad W:=\frac{\phi}{\zeta_0}.
\end{equation}
We define the transform operators as follows.
$
(\nabla_{\mathcal{A}}f)_i:=\mathcal{A}_{ij}\pa_jf,\quad\dive_{\mathcal{A}}X:=\mathcal{A}_{ij}\pa_j X_i,\quad \Delta_{\mathcal{A}}f:=\dive_{\mathcal{A}}\nabla_{\mathcal{A}}f
$
for appropriate $f$ and $X$. We write the stress tensor
$
S_{\mathcal{A}}(P,u):=PI-\mu\mathbb{D}_{\mathcal{A}}u
$
where $I$ is the $2\times2$ identity matrix and $(\mathbb{D}_{\mathcal{A}}u)_{ij}:=\mathcal{A}_{ik}\pa_ku_j+\mathcal{A}_{jk}\pa_ku_i$ is the symmetric $\mathcal{A}$--gradient. Note that if we extend $\dive_{\mathcal{A}}$ to act on symmetric tensors in the natural way, then $\dive_{\mathcal{A}}S_{\mathcal{A}}(P,u)=-\mu\Delta_{\mathcal{A}}u+\nabla_{\mathcal{A}}P$ for vector fields satisfying $\dive_{\mathcal{A}}u=0$.

Notice that $\Phi$ is a diffeomorphism (see \cite{ZhT17}), so that we rewrite the solutions to \eqref{eq:Bouss} as a perturbation around the equilibrium state $(U, P, \Theta, \zeta) = (0, P_0, 0, \zeta_0)$. By  defining new perturbed unknowns $(u, p, \theta, \eta)$ via $U=0+u$, $P=P_0+p$, $\Theta= 0+\theta$ and $\zeta=\zeta_0+\eta$, we can  use the analysis of \cite{GT2020} to transform the problem \eqref{eq:Bouss} into the perturbation form:
\begin{equation}\label{eq:geometric}
  \left\{
  \begin{aligned}
&\pa_tu-\pa_t\bar{\eta}WK\pa_2u+u\cdot\nabla_{\mathcal{A}}u+\dive_{\mathcal{A}}S_{\mathcal{A}}(p,u)= -g \theta e_2 \quad &\text{in}&\ \Om,\\
&\dive_{\mathcal{A}}u=0 \quad &\text{in}&\ \Om,\\
&\pa_t\theta -\pa_t\bar{\eta}WK\pa_2\theta + u\cdot\nabla_{\mathcal{A}}\theta - k \Delta_{\mathcal{A}} \theta =0\quad &\text{in}&\ \Om,
 \end{aligned}
  \right.
\end{equation}
under the boundary conditions
\begin{equation}\label{eq:geo_bc}
  \left\{
  \begin{aligned}
&S_{\mathcal{A}}(p,u)\mathcal{N}=\left[g\eta- (\sigma_1-\sigma_2\theta)\pa_1\left(\frac{\pa_1\eta }{(1+|\pa_1\zeta_0|)^{3/2}}\right)-\sigma\pa_1\Big(\mathcal{R}(\pa_1\zeta_0,\pa_1\eta)\Big)\right]\mathcal{N} \ &\text{on}&\ \Sigma,\\
&(S_{\mathcal{A}}(p,u)\nu-\beta u)\cdot\tau=0 \ &\text{on}&\ \Sigma_s,\\
&u\cdot\nu=0 \ &\text{on}&\ \Sigma_s,\\
&\theta =0 \ &\text{on}&\ \Sigma_s,\\
&\mathcal{N}\cdot\nabla \theta + \theta |\mathcal{N}|=0 \ &\text{on}&\ \Sigma\\
&\pa_t\eta=u\cdot\mathcal{N} \ &\text{on}&\ \Sigma,\\
&\kappa\pa_t\eta+\kappa\hat{\mathscr{W}}(\pa_t\eta)=\mp\sigma\left(\frac{\pa_1\eta  }{(1+|\zeta_0|^2)^{3/2}}+\mathcal{R}(\pa_1\zeta_0,\pa_1\eta)\right)\ &\text{at}&\ \pm\ell,
  \end{aligned}
  \right.
\end{equation}
where $\kappa:=\mathscr{W}'(0)>0$, $\mathcal{A}$, $\mathcal{N}:=-\pa_1\zeta e_1+e_2$ are still determined in terms of $\zeta=\zeta_0+\eta$ and $\mathcal{R}(\pa_1\zeta_0,\pa_1\eta)$ is the Taylor expansion remainder of $\frac1{\sqrt{1+|\pa_1\zeta|^2}}$ with respect to $\zeta_0$ multiplied by $\pa_1\eta$. In the following, let $\mathcal{N}_0$ be the non-unit normal vector for the equilibrium surface $\Sigma$, and $\mathcal{N}=\mathcal{N}_0-\pa_1\eta e_1$. The importance of $\varepsilon$ is discussed in Section \ref{section:discussion}.
Since all differential operators in \eqref{eq:geometric} are in terms of $\eta$, \eqref{eq:geometric} is connected to the geometry of the free surface. This geometric structure is essential to control higher-order derivatives.

\subsection{Main Theorem}

We denote $\om_{eq}$ as the equilibrium contact angle (refer to \cite{GT2020}) and the parameter $\varepsilon_{max} : = \min\{ 1, -1 + \pi / \om\}$ such that for $0<\varepsilon\le \varepsilon_{max}<1$,
there exist three parameters $\alpha$, $\varepsilon_-$ and $\varepsilon_+$ that satisfy
\begin{equation}\label{parameters}
\begin{aligned}
  &0<\alpha<\varepsilon_-<\varepsilon_+<\varepsilon_{max},\quad \alpha<\min\left\{\frac{\varepsilon_-}2, \frac{\varepsilon_+-\varepsilon_-}2\right\}, \quad \varepsilon_+\leq \frac{\varepsilon_-+1}2,&\\
  &q_-:=\frac2{2-\varepsilon_-}<q_+:=\frac2{2-\varepsilon_+}<q_{max} : = \frac2{2-\varepsilon_{max}}.&
  \end{aligned}
\end{equation}

Define the energy
\begin{equation}\label{energy}
\begin{aligned}
\mathcal{E}(u, p, \eta)&: =\|u\|_{W^{2,q_+}}^2+\|\pa_tu\|_{1+\varepsilon_-/2}^2+\sum_{k=0}^2\|\pa_t^ku\|_0^2+\|p\|_{W^{1,q_+}}^2+\|\pa_tp\|_0^2+\|\eta\|_{W^{3-1/q_+, q_+}}^2\\
&\quad+\|\pa_t\eta\|_{3/2+(\varepsilon_--\alpha)/2}^2+\sum_{j=0}^2\|\pa_t^j\eta\|_1^2,\\
 \mathcal{E}(\theta)&: = \|\theta\|_{W^{2,q_+}}^2+\|\pa_t\theta\|_{1+\varepsilon_-/2}^2+\sum_{k=0}^2\|\pa_t^k\theta\|_0^2,\\
\mathcal{E}(t)& : = \mathcal{E}(u, p, \eta, \theta) = \mathcal{E}(u, p, \eta) + \mathcal{E}(\theta),
\end{aligned}
\end{equation}
and the dissipation
\begin{equation}\label{dissipation}
\begin{aligned}
\mathcal{D}(u, p, \eta)&:=\|u\|_{W^{2,q_+}}^2+\|\pa_t u\|_{W^{2,q_-}}^2+\sum_{j=0}^2\Big(\|\pa_t^ju\|_1^2+\|\pa_t^ju\|_{L^2(\Sigma_s)}^2\Big)+\|p\|_{W^{1,q_+}}^2+\|\pa_tp\|_{W^{1,q_-}}^2\\
&\quad+\|\eta\|_{W^{3-1/q_+, q_+}}^2+\|\pa_t\eta\|_{W^{3-1/q_-, q_-}}^2+\sum_{j=0}^2\Big(\|\pa_t^j\eta\|_{3/2-\alpha}^2+[\pa_t^{j+1}\eta]_\ell^2\Big)+\|\pa_t^3\eta\|_{1/2-\alpha}^2,\\
\mathcal{D}(\theta)&:= \|\theta\|_{W^{2,q_+}}^2+\|\pa_t \theta\|_{W^{2,q_-}}^2+\sum_{j=0}^2\Big(\|\pa_t^j\theta\|_1^2+\|\pa_t^j\theta\|_{L^2(\Sigma)}^2\Big),\\
\mathcal{D}(t)&:=\mathcal{D}(u, p, \eta, \theta) = \mathcal{D}(u, p, \eta) + \mathcal{D}(\theta).
\end{aligned}
\end{equation}


\begin{theorem}\label{thm:main}
Assume that the contact angle $\om_{eq} \in (0, \pi)$ and $\varepsilon_{max} = \min\{1, -1+ \pi/\om_{eq}\} \in (0, 1)$.
Suppose that the initial data $\Big(u_0, p_0, \theta_0, \eta_0, \pa_tu(0), \pa_tp(0), \pa_t\theta(0), \pa_t\eta(0), \pa_t^2u(0), \pa_t^2\theta(0), \pa_t^2\eta(0)\Big)$ satisfy the compatibility conditions for $j=0,1,2$
\begin{equation}\label{compat_C2}
\left\{
\begin{aligned}
&\dive_{\mathcal{A}_0}D_t^ju(0)=0\quad &\text{in}&\ \Om,\\
&D_t^ju(0)\cdot\nu=0\quad &\text{on}&\ \Sigma_s,\\
&D_t^ju(0)\cdot\mathcal{N}(0)=\pa_t^{j+1}\eta(0)\quad &\text{on}&\ \Sigma,
\end{aligned}
\right.
\end{equation}
and zero-average condition for $k=0, 1, 2$
\begin{equation}\label{cond:zero}
\int_{-\ell}^\ell\pa_t^k\eta(0)=0.
\end{equation}
Then there exists sufficiently small $\delta_0>0$, such that if $\sqrt{\mathcal{E}(0)} + \|\pa_t^2 \eta(0)\|_{W^{2-1/q_+, q_+}} \le \delta_0$, then there exists a unique solution $(u,p, \theta, \eta)$ to \eqref{eq:geometric} for $t\in[0,\infty)$ that achieves the initial data and satisfies
\begin{equation}
  \begin{aligned}
\sup_{t\ge \infty}e^{\lambda t}\mathcal{E}(t)+\int_0^\infty\mathcal{D}(t)\,\mathrm{d}t\le C\mathcal{E}(0)
  \end{aligned}
  \end{equation}
for some universal constants $C>0$ and $\lambda>0$.
\end{theorem}

\begin{remark}
Note that the term $\|\pa_t^2\eta(0)\|_{W^{2-1/q_+, q_+}}^2$ is stronger than the initial energy functional. This term is sufficient to guarantee that our initial data set is non-empty (see the discussion in Appendix \ref{sec:initial}). Since $\Phi$ is a $C^1$ diffeomorphism due to $J>0$ for $\eta$ as well as its derivatives sufficiently small, we might change coordinates from $\Om$ to $\Om(t)$ to obtain solutions to \eqref{eq:Bouss}.
\end{remark}

\begin{remark}
  The differential operator $D_t$ in \eqref{compat_C2} is defined as $D_tu = \pa_tu - Ru$, where $R$ is the transport of $-\pa_tJKI_{2\times 2}-\pa_t\mathcal{A}\mathcal{A}^{-1}$. $D_t$ is introduced to preserve the divergence condition, i.e., $\dive_{\mathcal{A}}(D_t^j u) =0$. ($D_t^0$ should be considered as the identity.)
\end{remark}

\subsection{Strategy of the proof}\label{section:discussion}

Our proof is mainly focused on the linear heat system
\begin{equation}\label{eq:linear_heat}
\left\{
\begin{aligned}
  &\pa_t\theta - k \Delta_{\mathcal{A}} \theta=F^8\quad&\text{in}&\quad\Om,\\
  &k\mathcal{N} \cdot\nabla_{\mathcal{A}} \theta + \theta |\mathcal{N}|=F^9\quad&\text{on}&\quad\Sigma,\\
  &\theta = 0\quad&\text{on}&\quad\Sigma_s,
\end{aligned}
\right.
\end{equation}
coupled with the linear $\varepsilon-$ regularized Navier-Stokes system
\begin{equation}\label{eq:modified_linear}
\left\{
\begin{aligned}
  &\pa_tu+\dive_{\mathcal{A}}S_{\mathcal{A}}(p,u) + g \theta e_2=F^1\quad&\text{in}&\quad\Om,\\
  &\dive_{\mathcal{A}}u=0\quad&\text{in}&\quad\Om,\\
  &S_{\mathcal{A}}(p,u)\mathcal{N}=\left(\mathcal{K}(\eta+\varepsilon\eta_t)-\sigma_1\pa_1 F^3\right)\mathcal{N} + F^4 \quad &\text{on}& \quad \Sigma,\\
  &u \cdot \mathcal{N} = \pa_t \eta \quad &\text{on}& \quad \Sigma,\\
  &(S_{\mathcal{A}}(p,u) \nu -\beta u) \cdot \tau= F^5\quad &\text{on} & \quad \Sigma_s,\\
  &u \cdot \nu = 0\quad&\text{on}&\quad\Sigma_s,\\
  &\mp \sigma_1 \frac{\pa_1(\eta+\varepsilon\eta_t)}{(1+|\pa_1\zeta_0|^2)^{3/2}}(\pm\ell)=\kappa(u\cdot\mathcal{N})(\pm\ell) \pm \sigma F^3(\pm\ell) -F^7(\pm\ell).
\end{aligned}
\right.
\end{equation}
The role of $\varepsilon$ is played to lift the regularity of $\eta$, borrowed the idea of \cite{HG}. By analyze the elliptic theory near the corner points of Poisson equation, we can establish the existence and uniqueness of \eqref{eq:linear_heat}. Then by plugging the results into the well-posedness of \eqref{eq:modified_linear}, we can establish the results of linear systems that are used to do the iteration to establish the well-posedness of nonlinear systems \eqref{eq:geometric} and \eqref{eq:geo_bc}.

The elliptic theory for Poisson equation is determined via the spectral analysis of operator pencils. Then the symmtric structure of linear $\mathcal{A}$--heat equations allows us to use the Galerkin method to construct solutions via the two--tie energy methods. The first step is  to construct the $(\theta, \pa_t\theta)$; the second step is to construct $\pa_t^2\theta$ by $(\theta, \pa_t\theta)$ constructed in the first step. In the second step, due to the loss of high regularity of $\pa_t\theta$, we have to introduce another iteration to control the interaction between $\nabla^2\pa_t^2\bar{\eta}$ and $\pa_t\theta$.

The linear theory of \eqref{eq:modified_linear} is similar to \cite{GTWZ2023} after the minor modification of $\theta$ and coefficient of surface tension. Because of the smallness of $\theta \pa_1\left( \frac{\pa_1\eta}{\sqrt{1 + |\pa_1\eta|^1}} \right)$, the analysis in \cite{GTWZ2023} is still correct besides these modifications. Based on the linear theory, we use the fixed point theory to establish the global well-posedness of the nonlinear $\varepsilon$--regularized system
\begin{equation}\label{eq:epsilon}
  \left\{
  \begin{aligned}
&\pa_tu-\pa_t\bar{\eta}WK\pa_2u+u\cdot\nabla_{\mathcal{A}}u+\dive_{\mathcal{A}}S_{\mathcal{A}}(p,u)= -g \theta e_2 \quad &\text{in}&\ \Om,\\
&\dive_{\mathcal{A}}u=0 \quad &\text{in}&\ \Om,\\
&\pa_t\theta -\pa_t\bar{\eta}WK\pa_2\theta + u\cdot\nabla_{\mathcal{A}}\theta - k \Delta_{\mathcal{A}} \theta =0\quad &\text{in}&\ \Om,\\
&S_{\mathcal{A}}(p,u)\mathcal{N}=\left[g\eta- (\sigma_1-\sigma_2\theta)\pa_1\left(\frac{\pa_1\eta + \varepsilon \pa_1\pa_t\eta }{(1+|\pa_1\zeta_0|)^{3/2}}\right)-\sigma\pa_1\Big(\mathcal{R}(\pa_1\zeta_0,\pa_1\eta)\Big)\right]\mathcal{N} \ &\text{on}&\ \Sigma,\\
&(S_{\mathcal{A}}(p,u)\nu-\beta u)\cdot\tau=0 \ &\text{on}&\ \Sigma_s,\\
&u\cdot\nu=0 \ &\text{on}&\ \Sigma_s,\\
&\theta =0 \ &\text{on}&\ \Sigma_s,\\
&k\mathcal{N}\cdot\nabla \theta + \theta |\mathcal{N}|=0 \ &\text{on}&\ \Sigma\\
&\pa_t\eta=u\cdot\mathcal{N} \ &\text{on}&\ \Sigma,\\
&\kappa\pa_t\eta+\kappa\hat{\mathscr{W}}(\pa_t\eta)=\mp\sigma_1\left(\frac{\pa_1\eta + \varepsilon \pa_1\pa_t\eta  }{(1+|\zeta_0|^2)^{3/2}}+\mathcal{R}(\pa_1\zeta_0,\pa_1\eta)\right)\ &\text{at}&\ \pm\ell,
  \end{aligned}
  \right.
\end{equation}

We define
\begin{equation}\label{def:ep_energy}
  \begin{aligned}
\mathcal{E}^\varepsilon:=\mathcal{E}(u^\varepsilon, p^\varepsilon, \eta^\varepsilon, \theta^\varepsilon )+\varepsilon^2\|\pa_t\eta^\varepsilon\|_{W^{3-1/q_+,q_+}}^2+\varepsilon\sum_{j=0}^2\|\pa_t^j\eta^\varepsilon\|_{H^{3/2-\alpha}}^2
  \end{aligned}
\end{equation}
\begin{equation}\label{def:ep_dissipation}
  \begin{aligned}
\mathcal{D}^\varepsilon:=\mathcal{D}(u^\varepsilon,p^\varepsilon,\eta^\varepsilon, \theta^\varepsilon)+\varepsilon^2\|\pa_t\eta^\varepsilon\|_{W^{3-1/q_+,q_+}}^2+\varepsilon^2\|\pa_t^2\eta^\varepsilon\|_{W^{3-1/q_-,q_-}}^2+\varepsilon\sum_{j=0}^2\|\pa_t^j\eta^\varepsilon\|_{H^1}^2,
  \end{aligned}
\end{equation}
where $\mathcal{E}$ is the same as \eqref{energy} and $\mathcal{D}$ is the same as \eqref{dissipation}. Now the energy and dissipation do not propagate the regularity, we cannot use the usual way (for instance, set $t =0$ for the first equation in \eqref{eq:epsilon}) to construct the energy of initial data. Now we apply the backward way to construct $\mathcal{E}^\varepsilon (0)$. Then we can establish the global well-posedness and stability for \eqref{eq:epsilon}:
\begin{theorem}\label{thm:uniform1}
Suppose that there exists a universal constant $\delta\in (0,1)$ such that if a solution to \eqref{eq:geometric1} exists on the time interval $[0,T)$ for $0<T\le\infty$  and obeys the estimate
$\sup_{0\le t\le T}\mathcal{E}^\varepsilon(t)\le \delta_0$,
Then there exists universal constants $C>0$ and $\lambda>0$ such that
\begin{equation}
 \begin{aligned}
\sup_{0\le t\le T}e^{\lambda t}\mathcal{E}^\varepsilon(t)+\int_0^T\mathcal{D}^\varepsilon(t)\,\mathrm{d}t\le C\mathcal{E}^\varepsilon(0).
 \end{aligned}
\end{equation}
 \end{theorem}
Theorem \ref{thm:uniform1} allows us to pass the limit $\varepsilon \to 0^+$ to show that $\mathcal{E}^\varepsilon \to \mathcal{E}$ in $L^\infty$ and $\int_0^T\mathcal{D}^\varepsilon \to \int_0^T\mathcal{D}$. The convergence of initial energy $\mathcal{E}^\varepsilon(0) \to \mathcal{E}(0)$ is given in Appendix \ref{sec:initial}. In the construction of $\mathcal{E}(0)$, we still use the backward procedure, because the energy is not propagated in time.

\subsection{Constants and Sobolev Norms}\label{sec:notation}


Let $C>0$ denote a universal constant that only depends on the parameters of the problem, $N$ and $\Om$, but does not depend on the data, etc. They are allowed to change from line to line. We will write $C=C(z)$ to indicate that the constant $C$ depends on $z$. And we will write $a\lesssim b$ if $a\le C b$ for a universal constant $C>0$.

We will write $H^k$ instead of  $H^k(\Om)$, $W^{k, q}$ instead of $W^{k, q}(\Om)$ for $k\ge0$, and $H^s(\Sigma)$ or $W^{s, q}(\Sigma)$ with $s\in\mathbb{R}$ for usual Sobolev spaces. Typically, we will write $H^0=L^2$, with the exception to this is we will use $L^2([0,T];H^k)$ (or $L^2([0,T];H^s(\Sigma))$) to denote the space of temporal square--integrable functions with values in $H^k$ (or $H^s(\Sigma)$). Similarly, we write $L^2([0, T]; W^{k, q})$ or $L^2([0, T]; W^{s, q})$.

Sometimes we will write $\|\cdot\|_k$ instead of $\|\cdot\|_{H^k(\Om)}$ or $\|\cdot\|_{H^k(\Sigma)}$. We assume that functions have natural spaces. For example, the functions $u$, $p$ and $\bar{\eta}$ live on $\Om$, while $\eta$ lives on $\Sigma$. So we may write $\|\cdot\|_{H^k}$ for the norms of $u$, $p$ and $\bar{\eta}$ in $\Om$, and $\|\cdot\|_{H^s}$ for norms of $\eta$ on $\Sigma$.

Define time-independent spaces
\begin{equation}
{}_0H^1(\Om):=\Big\{\theta\in H^1(\Om)\Big|u=0\ \text{on}\ \Sigma_s\Big\},
\end{equation}
endowed with the usual $H^1$ norm,
\begin{equation}
  \mathring{H}^k(U):=\left\{f\in H^k(U)\Big| \int_Uf=0\right\},
\end{equation}
where $U=\Om$ or $(-\ell, \ell)$ and $k\in\mathbb{N}$.


Suppose that $\eta$ is given and that $\mathcal{A}$, $J$ and $\mathcal{N}$, etc are determined in terms of $\eta$. Let us define
\begin{equation}\label{sum_point}
[a,b]_\ell:=\kappa\Big(a(\ell)b(\ell)+a(-\ell)b(-\ell)\Big).
\end{equation}
We denote $\|\xi\|_{1,\Sigma}:=\sqrt{(\xi,\xi)_{1,\Sigma}}$ and $[\phi]_\ell :=\sqrt{[\phi,\phi]_\ell}$.

For convenience, let us define some time-dependent spaces
\begin{equation}
\mathcal{H}^0:=\Big\{\theta:\Om\rightarrow\mathbb{R}\Big|\sqrt{J}\theta\in H^0(\Om)\Big\}
\end{equation}
endowed with the norm $\|\theta\|_{\mathcal{H}^0} := \|\sqrt{J}\theta\|_{H^0}$,
\begin{equation}
\mathcal{H}^1:=\Big\{\theta:\Om\rightarrow\mathbb{R}\Big|\|\theta\|_{\mathcal{H}^1}<\infty,\ \theta=0\ \text{on}\ \Sigma_s\Big\}
\end{equation}
endowed with the norm $\|\theta\|_{\mathcal{H}^1} := \|\nabla_{\mathcal{A}}\theta\|_{\mathcal{H}^0} + \|\theta\sqrt{|\mathcal{N}|}\|_{L^2(\Sigma)}$, where the $\mathcal{H}^1$-inner product is defined via $(\theta, \psi)_{\mathcal{H}^1} = (\nabla_{\mathcal{A}}\theta, \nabla_{\mathcal{A}} \psi)_{\mathcal{H}^0} + (\theta |\mathcal{N}|, \psi)_{L^2(\Sigma)}$.
Finally, we define the inner product on $L^2([0, T]; \mathcal{H}^1)$ as
$
(\theta,\psi)_{\mathcal{H}^1_T}=\int_0^T\Big(\theta(t),\psi \Big)_{\mathcal{H}^1}\,\mathrm{d}t$,
endowed with norm $\|\theta\|_{\mathcal{H}^1_T}: = \sqrt{(\theta,\theta)_{\mathcal{H}^1_T}}$.


The following lemma implies that $\mathcal{H}^1$ is equivalent to ${}_0H^1(\Om)$, which is a modification of \cite{GT1}.
\begin{lemma}\label{lem:equivalence_norm}
There exists a small universal $\delta_0>0$ such that if $\sup_{0\le t\le T}\|\eta(t)\|_{W^{3-1/q_+, q_+}}<\delta_0$,
then
$
\frac{1}{\sqrt{2}} \|\theta\|_k \le \|\theta\|_{\mathcal{H}^k} \le \sqrt{2} \|\theta\|_k$
for $k=0, 1$ and for all $t\in[0, T]$. As a consequence, for $k=0, 1$,
  $
  \|\theta\|_{L^2H^k} \le \|\theta\|_{\mathcal{H}^k_T} \le \|\theta\|_{L^2H^k}$.
\end{lemma}





\section{Elliptic Theory for The Corner Domains}

We now consider some elliptic estimates for Poisson equations with corner points.
\subsection{Analysis for Poisson equations near the corner points}
 Let $(r,\rho)$ be the polar coordinates for $\mathbb{R}^2$.
\[
K_\om=\{x\in\mathbb{R}^2: r>0\ \text{and}\ \rho\in(-\pi/2, -\pi/2+\om)\}
\]
denotes the cone with open angle $\om\in(0,\pi)$. The lower and upper boundaries of $K_\om$ are
\[
\Gamma_-=\{x\in\mathbb{R}^2: r>0\ \text{and}\ \rho=-\pi/2\}\ \text{and}\ \Gamma_+=\{x\in\mathbb{R}^2: r>0\ \text{and}\ \rho= -\pi/2+\om\}
\]
respectively.

 Now, we consider the $\mathfrak{A}$-equation
 We first give the proof of $\vartheta$ through the Poisson equations
  \begin{align}\label{eq:piosson_1}
  \left\{
  \begin{aligned}
    &-k\Delta_{\mathfrak{A}}\vartheta=G^8\quad &\text{in}&\quad K_\om,\\
    &k\nabla_{\mathfrak{A}}\vartheta\cdot(\mathfrak{A}\nu) =G^9\quad &\text{on}&\quad\Gamma_+,\\
    & \vartheta = 0\quad &\text{on}&\quad\Gamma_-,
  \end{aligned}
  \right.
  \end{align}
where the differential operators $\nabla_{\mathfrak{A}}$ and $\Delta_{\mathfrak{A}}$ are defined in the same way as $\nabla_{\mathcal{A}}$ and $\Delta_{\mathcal{A}}$. Clearly, when $\mathfrak{A}=I_{2\times2}$, \eqref{eq:piosson_1} becomes
\begin{align}\label{eq:piosson_2}
  \left\{
  \begin{aligned}
    &-k\Delta \vartheta=G^8\quad &\text{in}&\quad K_\om,\\
    &k\nabla\vartheta\cdot\nu=G^9\quad &\text{on}&\quad\Gamma_+,\\
    & \vartheta = 0\quad &\text{on}&\quad\Gamma_-.
  \end{aligned}
  \right.
  \end{align}

\begin{theorem}\label{thm:corner}
Let $\delta = \pm$.    Assume that the forcing terms in \eqref{eq:piosson_1} satisfy $G^3\in L^{q_\delta}(K_\om)$ and $G^5\in W^{1-1/q_\delta, q_\delta}(\Gamma_+)$.
  Suppose that $\vartheta\in  H^1(K_\om)$ satisfies
  \[
  \int_{K_\om}k\nabla_{\mathfrak{A}}\vartheta\cdot\nabla_{\mathfrak{A}}\phi + \int_{\Gamma_+} \vartheta \phi=\int_{K_\om}G^8\phi+\int_{\Gamma_+}G^9\phi,
  \]
  for all $\phi\in H^1(K_\om)$. Furthermore, assume that $\vartheta$ and all the forcing terms $G^i$ are supported in $\bar{K}_\om\cap B_1(0)$, where $B_1(0)$ is a unit disk centered at $0$ with radius $1$. Then $\nabla^2\vartheta\in L^{q_\delta}(K_\om)$. Moreover,
  \begin{equation}\label{est:cone}
     \|\nabla^2\vartheta\|_{L^{q_\delta}(K_\om)}^2\lesssim \|G^8\|_{L^{q_\delta}(K_\om)}^2+\|G^9\|_{W^{1-1/q_{\delta}, q_\delta}(\Gamma)}^2.
  \end{equation}
\end{theorem}
\begin{proof}
  The key point for proof is the utilization of operator pencil (we refer to \cite{KMR2} for the definition). The eigenvalues of operator pencil for \eqref{eq:piosson_2} could be easily derived (for instance, refer to \cite{KMR1}) as $\{\f{n\pi}{\om}: n\in\mathbb{Z}\setminus\{0\}\}$, which are not contained in
  \[
  \{\lambda\in\mathbb{C}: 0\le\Re{\lambda}<1\},
  \]
  which allows us to use the $W^{2, q_\delta}$ estimate. $q_\delta < 2$ is determined by the argument of Stokes equations in the cone, indicated by \cite{GT2020}.
   So we can use \cite[Theorem 8.2.1]{KMR1}, which gives the condition in terms of eigenvalue of associated operator pencil for solving the elliptic system \eqref{eq:piosson_2},
  and then argue as \cite[Theorem 6.4.6]{MR} to derive that $\nabla^2\vartheta\in L^{q_\delta}(K_\om)$ satisfying
  \begin{align}\label{est:cone_1}
  \|\nabla^2\vartheta\|_{L^{q_\delta}(K_\om)}^2\lesssim \|G^8\|_{L^{q_\delta}(K_\om)}^2+\|G^9\|_{W^{1-1/q_\delta, 1/q_\delta}(\Gamma_+)}^2 + \|\vartheta\|_{H^1(K_\om)}^2.
  \end{align}
  The assumptions on $\mathfrak{A}$ guarantee that the two Poisson equations in \eqref{eq:piosson_1} and  \eqref{eq:piosson_2} generate the same operator pencil, which implies that the solution $\vartheta$ to \eqref{eq:piosson_1} also satisfies \eqref{est:cone_1}.
  From the assumption on $\vartheta$ and the Poinc\'are inequality together with Sobolev embedding theory 
   we also have
  \begin{align}\label{est:cone_2}
  \|\vartheta\|_{H^1(K_\om)}^2\lesssim \|G^8\|_{L^{q_\delta}(K_\om)}^2+\|G^9\|_{W^{1-1/q_\delta, 1/q_\delta}(\Gamma_+)}^2.
  \end{align}
  The result is completed by combining \eqref{est:cone_1} and \eqref{est:cone_2}.
\end{proof}
\subsection{$\mathcal{A}$-Poisson equation}
We first consider the  strong solutions to Poisson equations
\begin{equation}\label{eq:elliptic_1}
  \left\{
  \begin{aligned}
    &-k\Delta \vartheta=G^8\quad&\text{in}&\quad \Om,\\ &k\nabla\vartheta\cdot\nu=G^9&\text{on}&\quad\Sigma,\\
    &\vartheta=0&\text{on}&\quad\Sigma_s,
  \end{aligned}
  \right.
\end{equation}
where $\nu$ is the unit outward normal and unit tangential of $\p\Om$.

\begin{definition}\label{def:weak_e}
  Assume that $G^8\in L^{q_\delta}(\Om)$ and $G^9\in W^{1-1/q_\delta, q_\delta}(\Sigma)$ for  $\delta= \pm$. We say $\vartheta\in {}_0H^1(\Om)$ is a weak solution of \eqref{eq:elliptic_1} if
  \begin{align}\label{eq:weak_theta}
\int_\Om k\nabla\vartheta\cdot\nabla\phi=\int_\Om G^8\phi+\int_\Sigma G^9\phi
  \end{align}
  holds for any $\phi\in{}_0H^1(\Om)$. Clearly, from Sobolev embedding theory, the integrations on the righthand side of \eqref{eq:weak_theta} is well-defined.
\end{definition}

Now we give the sketch proof of existence and uniqueness of weak solutions of elliptic equation \eqref{eq:elliptic_1}. First, \eqref{eq:weak_theta} allows us to use Riesz representation theorem to obtain the existence of $\vartheta$ such that $\|\nabla\vartheta\|_0^2\lesssim\|G^8\|_{L^{q_\delta}}^2+\|G^9\|_{W^{1-1/q_\delta, q_\delta}}^2$. Then the Poinc\'are inequality implies $\|\vartheta\|_1^2 \lesssim \|G^8\|_{L^{q_\delta}}^2+\|G^9\|_{W^{1-1/q_\delta, q_\delta}}^2$.

The next theorem shows that, regularity of weak solution of the equation \eqref{eq:elliptic_1} could be improved to second-order.
\begin{theorem}\label{thm:second order}
  Assume that $G^8\in L^{q_\delta}(\Om)$ and $G^9\in W^{1-1/q_\delta, q_\delta}(\Sigma)$ for $\delta= \pm$. Then there exists a unique $\vartheta$ solving \eqref{eq:elliptic_1} such that $\vartheta\in W^{2, q_\delta}(\Om)$. Moreover,
  \begin{equation}
    \begin{aligned}
      \|\vartheta\|_{W^{2, q_\delta}(\Om)}^2\lesssim \|G^8\|_{L^{q_\delta}(\Om)}^2+\|G^9\|_{W^{1-1/q_\delta, q_\delta}(\Sigma)}^2.
    \end{aligned}
  \end{equation}
\end{theorem}
\begin{proof}
  The idea of this proof is very standard. We divided $\Om$ into three parts. One is away from the corners. Thus, we might use the standard elliptic theory (for instance, \cite[Theorem 10.5]{ADN}) to obtain the results. The other two parts are near corners. Then we might use Theorem \ref{thm:corner} to derive the conclusion. The proof is in a similar way as \cite[Theorem 5.6]{GT18}. So we omit details.
\end{proof}

We then consider the Poisson equations with coefficients:
\begin{align}\label{eq:elliptic_2}
\left\{
\begin{aligned}
  &-k\Delta_{\mathcal{A}}\vartheta=G^3\ &\text{in}&\ \Om,\\
  &k\nabla_{\mathcal{A}}\vartheta\cdot\mathcal{N}=G^6\ &\text{on}&\ \Sigma,\\
  &\vartheta=0\ &\text{on}&\ \Sigma_s,
\end{aligned}
\right.
\end{align}
where we assume that $\eta$ and $\mathcal{A}$, $\mathcal{N}$, etc. are given.
\begin{theorem}\label{thm:elliptic}
  Let $\delta= \pm$. Suppose that $\|\eta\|_{W^{3-1/q_\delta, q_\delta}}<\gamma_0$ where $\gamma_0\ll 1$ . Assume that $G^8\in L^{q_\delta}(\Om)$ and $G^9\in W^{1-1/q_\delta, q_\delta}(\Sigma)$. Then there exists a unique $\vartheta\in W^{2, q_\delta}(\Om)$ solving \eqref{eq:elliptic_2}. Moreover,
  \begin{align}\label{est:elliptic}
  \begin{aligned}
    \|\vartheta\|_{W^{2, q_\delta}(\Om)}^2
    \lesssim \|G^8\|_{L^{q_\delta}(\Om)}^2+\|G^9\|_{W^{1-1/q_\delta, q_\delta}(\Sigma)}^2.
  \end{aligned}
  \end{align}
\end{theorem}
\begin{proof}
We rewrite \eqref{eq:elliptic_2} as the perturbation form:
\begin{align}\label{eq:elliptic_3}
\left\{
\begin{aligned}
  &-k\Delta \vartheta=G^3-k\dive_{I-\mathcal{A}}\nabla_{\mathcal{A}}\vartheta-k\dive\nabla_{I-\mathcal{A}}\vartheta\ &\text{in}&\ \Om,\\
  &k\nabla\vartheta\cdot \mathcal{N}_0=G^6+k\nabla_{I-\mathcal{A}}\vartheta\cdot\mathcal{N}+k\nabla\vartheta\cdot(\mathcal{N}_0-\mathcal{N})\ &\text{on}&\ \Sigma,\\
  &\vartheta=0\ &\text{on}&\ \Sigma_s.
\end{aligned}
\right.
\end{align}
We now employ fixed point theory to solve \eqref{eq:elliptic_3}. Suppose that $\theta\in W^{2, q_\delta}(\Om)$. Then we define the operator $T_\eta: W^{2, q_\delta}(\Om)\to W^{2, q_\delta}(\Om)$ via $\theta\mapsto\vartheta=T_\eta\theta$, where $\vartheta$ and $\theta$ satisfy
\begin{align}\label{eq:elliptic_4}
\left\{
\begin{aligned}
  &-k\Delta \vartheta=G^3-k\dive_{I-\mathcal{A}}\nabla_{\mathcal{A}}\theta-k\dive\nabla_{I-\mathcal{A}}\theta\ &\text{in}&\ \Om,\\
  &k\nabla\vartheta\cdot \mathcal{N}_0=G^6+k\nabla_{I-\mathcal{A}}\theta\cdot\mathcal{N}+k\nabla\theta\cdot(\mathcal{N}_0-\mathcal{N})\ &\text{on}&\ \Sigma,\\
  &\vartheta=0\ &\text{on}&\ \Sigma_s.
\end{aligned}
\right.
\end{align}
In order to use Theorem \ref{thm:second order}, we need to estimate the right side of \eqref{eq:elliptic_4}. First note that $\nabla \mathcal{A} \sim \nabla^2 \bar{\eta}$.
We use H\"older inequality, Sobolev embedding and trace theory to estimate
\begin{align}\label{est:ellip_3}
\begin{aligned}
  &\|\dive_{I-\mathcal{A}}\nabla_{\mathcal{A}}\theta\|_{L^{q_\delta}(\Om)}^2+\|\dive\nabla_{I-\mathcal{A}}\theta\|_{L^{q_\delta}(\Om)}^2\\
  &\lesssim \|I-\mathcal{A}\|_{L^\infty}^2(1+\|\nabla\mathcal{A}\|_{L^{2/(1+\varepsilon_\delta)}}^2)\|\nabla\theta\|_{L^{2/(1-\varepsilon_\delta)}}^2
  +\|I-\mathcal{A}\|_{L^\infty}^2(1+\|\mathcal{A}\|_{L^\infty}^2)\|\theta\|_{W^{2, q_\delta}}^2\\
  &\lesssim \|\eta\|_{W^{3-1/q_\delta, q_\delta}}^2\|\theta\|_{W^{2, q_\delta}}^2.
\end{aligned}
\end{align}
Similarly, we use the trace theory to see that
\begin{align}\label{est:ellip_4}
\begin{aligned}
  &\|\nabla_{I-\mathcal{A}}\theta\cdot\mathcal{N}\|_{W^{1-1/q_\delta, q_\delta}(\Sigma)}^2+\|\nabla\theta\cdot(\mathcal{N}_0-\mathcal{N})\|_{W^{1-1/q_\delta, q_\delta}(\Sigma)}^2\\
  &\lesssim\|\nabla_{I-\mathcal{A}}\theta\cdot(\mathcal{N}_0-\pa_1\bar{\eta}e_1)\|_{W^{1, q_\delta}(\Om)}^2+\|\nabla\theta\cdot\pa_1\bar{\eta}e_1\|_{W^{1, q_\delta}(\Om)}^2\\
  &\lesssim \left(\|\nabla\mathcal{A}\|_{L^{2/(1+\varepsilon_\delta)}}^2\|\mathcal{N}_0-\pa_1\bar{\eta}e_1\|_{L^\infty}^2+\|I-\mathcal{A}\|_{L^\infty}^2\|\nabla(\mathcal{N}_0-\pa_1\bar{\eta}e_1)\|_{L^{2/(1+\varepsilon_\delta)}}^2\right)\|\nabla\theta\|_{L^{2/(1-\varepsilon_\delta)}}\\
  &\quad+ \left(\|I-\mathcal{A}\|_{L^\infty}^2\|\mathcal{N}_0-\pa_1\bar{\eta}e_1\|_{L^\infty}^2+\|\pa_1\bar{\eta}\|_{L^\infty}^2\right)\|\theta\|_{W^{2, q_\delta}}^2+\|\pa_1\bar{\eta}\|_{L^{2/(1+\varepsilon_\delta)}}^2\|\nabla\theta\|_{L^{2/(1-\varepsilon_\delta)}}\\
  &\lesssim \|\eta\|_{W^{3-1/q_\delta, q_\delta}}^2\|\theta\|_{W^{2, q_\delta}}^2.
\end{aligned}
\end{align}
With estimates \eqref{est:ellip_3} and \eqref{est:ellip_4}, we now use Theorem \ref{thm:second order} to \eqref{eq:elliptic_4} to obtain that
\[
\|\vartheta_1-\vartheta_2\|_{W^{2, q_\delta}}^2\lesssim \|\eta\|_{W^{3-1/q_\delta, q_\delta}}^2\|\theta_1-\theta_2\|_{W^{2, q_\delta}}^2,
\]
for $T_\eta\theta_j=\vartheta_j$, $j=1, 2$.
Then we choose $\gamma_0$ sufficiently small such that
\[
\|\vartheta_1-\vartheta_2\|_{W^{2, q_\delta}}\le\f12\|\theta_1-\theta_2\|_{W^{2, q_\delta}},
\]
which yields $T_\eta$ is strictly contractive.
Thus, we use Banach's fixed point theory to deduce that \eqref{eq:elliptic_3} has a unique solution $\vartheta\in W^{2, q_\delta}(\Om)$.
\end{proof}

Finally, we consider the equation
\begin{align}\label{eq:elliptic_5}
\left\{
\begin{aligned}
  &-k\Delta_{\mathcal{A}}\vartheta=G^8\ &\text{in}&\ \Om,\\
  &k\nabla_{\mathcal{A}}\vartheta\cdot\mathcal{N}+\vartheta|\mathcal{N}|=G^9\ &\text{on}&\ \Sigma,\\
  &\vartheta=0\ &\text{on}&\ \Sigma_s.
\end{aligned}
\right.
\end{align}
\begin{theorem}\label{thm:elliptic_1}
  Let $\delta = \pm$. Suppose that $\|\eta\|_{W^{3-1/q_\delta, q_\delta}}<\gamma_0$ where $\gamma_0$ is the same as Theorem \ref{thm:elliptic} . Assume that $G^8\in L^{q_\delta}(\Om)$ and $G^9\in W^{1-1/q_\delta, q_\delta}(\Sigma)$. Then there exists a unique $\vartheta\in W^{2, q_\delta}(\Om)$ solving \eqref{eq:elliptic_5}. Moreover,
  \begin{align}\label{est:elliptic_1}
  \begin{aligned}
    \|\vartheta\|_{W^{2, q_\delta}(\Om)}^2
    \lesssim \|G^8\|_{L^{q_\delta}(\Om)}^2+\|G^9\|_{W^{1-1/q_\delta, q_\delta}(\Sigma)}^2.
  \end{aligned}
  \end{align}
\end{theorem}
\begin{proof}
  For simplicity, we define the trace operator $R: W^{2, q_\delta}(\Om) \to L^{q_\delta}(\Om)\times W^{1-1/q_\delta, q_\delta}(\Sigma)\times W^{2-1/q_\delta, q_\delta}(\Sigma_s)$ via
  \[
  R\vartheta=(0, (\vartheta|\mathcal{N}|)|_{\Sigma}, 0).
   \]
  Note that $R$ is compact, since the embedding $W^{2-1/q_\delta, q_\delta}(\Sigma_s)\hookrightarrow W^{1-1/q_\delta, q_\delta}(\Sigma_s)$ is compact. So the operator $T_\eta+R$ is Fredholm, which means the dimensions of kernel and cokernel of $T_\eta+R$ are both finite. However, for $\vartheta\in \ \text{Ker} (T_\eta+R)$, we multiply the first equation in \eqref{eq:elliptic_5} by $\vartheta J$, and integrate by parts over $\Om$ to see that
  \begin{align}
  k\int_\Om |\nabla_\mathcal{A}\vartheta|^2J+\int_{\Sigma}|\vartheta|^2|\mathcal{N}|=0,
  \end{align}
  which implies $\vartheta=0$. So $T_\eta+R$ is injective. Then Fredholm alternative tells us that $T_\eta+R$ is also surjective. Hence $T_\eta+R$ is an isomorphism. Thus \eqref{eq:elliptic_5} is uniquely solvable, i.e. $\vartheta = (T_\eta+R)^{-1}(G^8, G^9, 0)$ and the estimate \eqref{est:elliptic_1} holds.
\end{proof}

\section{Solution to $\mathcal{A}$-linear Heat Equations}\label{sec:linear}

In this section, we will establish the well-posedness theory to the $\mathcal{A}$-linear heat equations. We first introduce the definition of weak solutions, then use the Galerkin method to construct weak solutions and strong solutions via elliptic theory. Finally, we use the two--tier energy method to construct solutions in higher regular spaces.

\subsection{Weak Solutions}

\begin{definition} \label{def:weak}  Suppose that $\eta$ is given as well as $\mathcal{A}$, $J$, $\mathcal{N}$, etc. that are determined in terms of $\eta$ as in \eqref{eq:Jaccobian}, \eqref{eq:components}. We say that $\theta \in L^\infty([0,T]; {}_0H^0)\cap L^2([0,T]; H^1)$ is a weak solution to \eqref{eq:linear_heat}, provided that
  \begin{equation}\label{eq:weak_limit_1}
  \begin{aligned}
  (\pa_t \theta, \phi)_{\mathcal{H}^0} + k(\nabla_{\mathcal{A}} \theta, \nabla_{\mathcal{A}}\phi)_{\mathcal{H}^0} + \int_{\Sigma}\theta \phi |\mathcal{N}| = \int_\Om F^8 \phi J + \int_{\Sigma}F^9 \phi
  \end{aligned}
  \end{equation}
  for each $\phi \in {}_0H^1(\Om)$, provided that $F^8 \in L^{q_+}(\Om)$ and $F^9 \in W^{1-1/q_+, q_+}(\Sigma)$.
\end{definition}

We now ignore the existence of weak solutions and record the uniqueness on some integral equalities and bounds satisfied by weak solutions, since the existence of weak solutions to \eqref{eq:linear_heat} will arise  as a byproduct of the construction of strong solutions to \eqref{eq:linear_heat}.

\begin{lemma}
  Suppose that $\theta$ is a weak solution to \eqref{eq:weak_limit_1}. Then for almost every $t\in [0,T]$,
  \begin{equation}\label{equality}
  \begin{aligned}
  \frac12\|\theta(t)\|_{\mathcal{H}^0}^2+\frac k2\int_0^t\|\nabla_{\mathcal{A}} \theta\|_{\mathcal{H}^0}^2\,\mathrm{d}s+\int_0^t\|\sqrt{|\mathcal{N}|}\theta\|_{0,\Sigma}^2
  =\frac12\|\theta(0)\|_{\mathcal{H}^0}^2 + \int_\Om F^8\theta J + \int_{\Sigma}F^9\theta.
    \end{aligned}
  \end{equation}
  Moreover,
  \begin{equation}\label{bound}
  \begin{aligned}
  &\sup_{0\le t\le T}\|\theta(t)\|_{\mathcal{H}^0}^2+\|\theta\|_{L^2H^0(\Sigma)}^2+\|\nabla_{\mathcal{A}} \theta \|_{\mathcal{H}^0_T}^2 \\
  &\lesssim \exp(C_0(\eta)T)\Big(\|\theta(0)\|_{\mathcal{H}^0}^2+\|F^8\|_{L^{q_+}(\Om)}^2 + \|F^9\|_{W^{1-1/q_+, q_+}(\Sigma)}^2\Big),
  \end{aligned}
  \end{equation}
  where $C_0(\eta):=\sup_{0\le t\le T}\|\pa_tJK\|_{L^\infty}$.
\end{lemma}
\begin{proof}
  The identity \eqref{equality} follows directly from \eqref{eq:weak_limit_1} by using the test function $\phi =\theta\chi_{[0,t]}$, where $\chi$ is the characteristic function on $[0,t]$.
  From \eqref{equality}, we can directly derive the inequality by H\"older inequality:
  \begin{align}
  \begin{aligned}
&\frac12\|\theta(t)\|_{\mathcal{H}^0}^2+\frac k2\int_0^t\|\nabla_{\mathcal{A}}\theta(s)\|_{\mathcal{H}^0 }^2+\|\sqrt{|\mathcal{N}|}\theta(s)\|_{0,\Sigma}^2\,\mathrm{d}s\\
  &\le\frac12\|\theta(0)\|_{\mathcal{H}^0}^2+\|J\|_{L^\infty(\Om)}\|F^8\|_{L^{q_+}(\Om)}\|\theta\|_{L^{q_+/(q_+-1)}(\Om)}\\ &\quad+ \|F^9\|_{L^{1/(1-\varepsilon_+)}(\Sigma)}\|\theta\|_{L^{1/\varepsilon_+}(\Sigma)} + \frac12\|\pa_tJK\|_{L^\infty}\|\theta(t)\|_{\mathcal{H}^0}^2.
  \end{aligned}
  \end{align}
  The Cauchy inequality and Sobolev embedding inequality then imply that
  \[
  \begin{aligned}
&\frac12\|\theta(t)\|_{\mathcal{H}^0}^2+\frac k2\|\nabla_{\mathcal{A}}\theta\|_{\mathcal{H}^0_t}^2 + \|\sqrt{|\mathcal{N}|}\theta\|_{L^2H^0(\Sigma)}^2\\
&\le \frac12\|\theta(0)\|_{\mathcal{H}^0}^2+ 2\|J\|_{L^\infty(\Om)}^2\|F^8\|_{L^{q_+}(\Om)}^2+ C \|F^9\|_{L^{1/(1-\varepsilon_+)}(\Sigma)}^2 + \frac12\|\pa_tJK\|_{L^\infty}\|\theta(t)\|_{\mathcal{H}^0_t}^2.
  \end{aligned}
  \]
  Then Gronwall's inequality implies \eqref{bound}.
\end{proof}
\begin{proposition}\label{prop:unique}
  Weak solutions to \eqref{eq:weak_limit_1}
  are unique.
\end{proposition}
\begin{proof}
  If $\theta^1$ and $\theta^2$ are both weak solutions to \eqref{eq:weak_limit_1}, then $\vartheta=\theta^1-\theta^2$ is a weak solution to \eqref{eq:weak_limit_1} with $F^8 = F^9 =0$ and the initial data $\vartheta(0)=0$. The bound \eqref{bound} implies that $\vartheta=0$. Hence, weak solutions to \eqref{eq:weak_limit_1} are unique.
\end{proof}

\subsection{Construction of Strong Solutions}

 Suppose that $\eta$, as well as $\mathcal{A}$, $J$, $\mathcal{N}$, etc. that are determined in terms of $\eta$, are given. We will use Galerkin method to construct strong solutions.

\subsubsection{Galerkin Approximation for Initial Data}

  In the Galerkin method for $t>0$, we use the tw--tier method to construct solutions: we construct the approximated solutions $\theta^m$ and $\pa_t\theta^m$ and then pass to the limit to find $\theta$ and $\pa_t\theta$ in the first step, where we need the convergence of initial data $(\theta_0^m, \pa_t\theta^m(0))$ to $(\theta_0, \pa_t\theta(0))$;  in the following second step, we use $(\pa_t\theta^m, \theta)$ to construct $\pa_t^2\theta^m$, where we need the convergence of initial data $(\pa_t\theta^m(0), \pa_t^2\theta^m(0))$ to $(\pa_t\theta(0), \pa_t^2\theta(0))$.  As a preliminary, we can use the following Lemma \ref{lem:basis_initial}, Propositions  \ref{prop:bound_dtu} and \ref{prop:converge_dt2u0} to prove that
 \begin{theorem}\label{thm:initial_convergence}
  There exists a basis $\{w^k\}_{k=1}^\infty$ for the Hilbert space $\mathcal{H}^1(t)$ for each $t\ge0$. When $t=0$, $\mathcal{H}^1_m(0):=span \{w^1(0), \ldots, w^m(0)\}$, the Galerkin approximated initial data $(\pa_t\theta^m(0), \pa_t^2\theta^m(0))\in \mathcal{H}^1_m(0)$ obeys the estimates
  \[
  \begin{aligned}
\|\pa_t\theta^m(0)\|_{\mathcal{H}^0} & \lesssim \|\pa_t\theta(0)\|_{H^0}+\|\theta_0\|_{L^2(\Sigma)}, \\
\|\pa_t^2\theta^m(0)\|_{\mathcal{H}^0} & \lesssim \|\pa_t^2\theta(0)\|_{H^0}+\|\pa_t\theta(0)\|_{H^1}+\|\pa_t\theta(0)\|_{L^2(\Sigma)}.
  \end{aligned}
  \]
  Moreover, $(\pa_t\theta^m(0), \pa_t^2\theta^m(0))$ converge to $(\pa_t\theta(0), \pa_t^2\theta(0))$ weakly in $\mathcal{H}^0$.
\end{theorem}


In order to construct the approximate solutions $\pa_t\theta^m(t)$ for the linear system \eqref{eq:linear_heat}, we need to construct the initial data $\pa_t\theta^m(0)$ and prove that they are uniformly bounded. Our construction is much subtle due to the orthogonality of basis of $\mathcal{H}^0(0)$ and $\mathcal{H}^1(0)$. For the orthogonality,  we define the map $\mathcal{S}:\mathcal{H}^1(0)\rightarrow (\mathcal{H}^1(0))^\ast$ via
  $\left<\mathcal{S}\theta, \phi\right>_{(\mathcal{H}^1(0))^\ast}=(\theta, \phi)_{\mathcal{H}^1(0)},\quad \forall \theta, \phi \in \mathcal{H}^1(0)$.
  By Riesz representation theorem, $\mathcal{S}$ is an isometric isomorphism. We then define $A=\mathcal{S}^{-1}: (\mathcal{H}^1(0))^\ast \rightarrow \mathcal{H}^1(0)$ so that for any $f\in (\mathcal{H}^1(0))^\ast$, $Af\in \mathcal{H}^1(0)$ is uniquely determined via
  $\left<f, \phi\right>_{(\mathcal{H}^1(0))^\ast}=(Af, \phi)_{\mathcal{W}(0)},\quad \forall \phi \in \mathcal{H}^1(0)$,
  that is equivalent to say that $Af=\theta$ is a weak solution to $\mathcal{S}\theta=f$.

 \begin{lemma}\label{lem:basis_initial}
The operator $A$ restricted on $\mathcal{H}^0(0)$ is a compact, positive and self-adjoint operator. So the eigenvalues of $A$ are $\{\lambda_k\}_{k=1}^\infty$ such that $
0<\lambda_1\le\lambda_2\le\cdots$
and
$\lambda_k\to\infty$, as $ k\to\infty$.
Finally, there exists an orthonormal basis $\{w^k\}_{k=1}^\infty$ of $\mathcal{H}^0(0)$, where $w^k\in \mathcal{H}^1(0)$ is an eigenfunction of $\mathcal{S}$  corresponding to $\lambda_k$:
$
\mathcal{S}w^k=\lambda_kw^k.
$
Moreover, $\left\{\frac{w^k}{\sqrt{\lambda_k}}\right\}_{k=1}^\infty$ is an orthonormal basis of $\mathcal{H}^1(0)$.
  \end{lemma}
  \begin{proof}
  Note that $\mathcal{H}^1(0)\subset\subset \mathcal{H}^0 (0)\hookrightarrow(\mathcal{H}^1(0))^\ast$.
  We restrict $A|_{\mathcal{H}^0(0)}: \mathcal{H}^0(0) \rightarrow \mathcal{H}^1(0)$ and $A|_{\mathcal{H}^0(0)}: \mathcal{H}^0(0)\rightarrow \mathcal{H}^0(0)$ is compact.

  We claim that $A$ is a positive, self-adjoint operator.  For $f, g\in \mathcal{H}^0(0)$, let $u=Af$ and $v=Ag$. Then
  $(u, v)_{\mathcal{H}^1(0)}=(Af , v)_{\mathcal{H}^1(0)}=(f, v)_{\mathcal{H}^0(0)}=(f,Ag)_{\mathcal{H}^0(0)}$,
  and
  $(v, u)_{\mathcal{H}^1(0)}=(Ag , u)_{\mathcal{H}^1(0)}=(g, u)_{\mathcal{H}^0(0)}=(g,Af)_{\mathcal{H}^0(0)}$.
  By symmetry, $(u, v)_{\mathcal{H}^1(0)}=(v, u)_{\mathcal{H}^1(0)}$. So
  $(f,Ag)_{\mathcal{H}^0(0)}=(g,Af)_{\mathcal{H}^0(0)}=(Af, g)_{\mathcal{H}^0(0)}$.
  Also, $(Af, f)_{\mathcal{H}^0(0)}=(f, f)_{\mathcal{H}^1(0)}\ge0$. Therefore, our claim is valid.

  Consequently, by Riesz-Schauder theorem, there exists a complete orthonormal basis $\{w^k\}_{k=1}^\infty\subseteq \mathcal{H}^0(0)$ such that
  $Aw^k=\mu_kw^k$, and $\mu_k\to0$, as $k\to\infty$.
  But
  $Aw^k=\mu_kw^k\in \mathcal{H}^1(0)$, for all $w^k\in \mathcal{H}^0(0)$,
  so
  $\{w^k\}_{k=1}^\infty\subseteq \mathcal{H}^1(0)$,
  and
  $\mu_k(w^k, \psi)_{\mathcal{H}^1(0)}=(Aw^k, \psi)_{\mathcal{H}^1(0)}=(w^k, \psi)_{\mathcal{H}^0(0)}, \forall \psi\in \mathcal{H}^1(0)$.
  Then we have
  $(w^k, \psi)_{\mathcal{H}^1(0)}=\lambda_k(w^k, \psi)_{\mathcal{H}^0(0)}$ with  $\lambda_k=\frac1{\mu_k}\to\infty$, as $k\to\infty$.
  Thus $(w^k, w^j)_{\mathcal{H}^1(0)}=\lambda_k(w^k, w^j)_{\mathcal{H}^0(0)}=\lambda_k\delta_{jk}$ implies $\{w^k/\sqrt{\lambda_k}\}$ is a complete orthonormal basis of $\mathcal{H}^1(0)$.
  \end{proof}

For the finite dimensional space $\mathcal{H}_m^1(t) = \text{span}\{w^1, \ldots, w^m\}$ with $t\ge0$, consider the approximate formulation
  \begin{equation}\label{app_1}
\begin{aligned}
  (\pa_t\theta^m,\psi)_{\mathcal{H}^0}+k(\nabla_{\mathcal{A}}\theta^m, \nabla_{\mathcal{A}}\psi)_{\mathcal{H}^0}+(|\mathcal{N}|\theta^m,\psi)_{0,\Sigma}=\int_\Om F^8\psi J + \int_\Sigma F^9\psi |\mathcal{N}|
  \end{aligned}
  \end{equation}
  for any $\psi\in \mathcal{H}_m^1(t)$.  For our purpose to construct $\pa_t\theta^m(t)$, we define
  $\theta^m(t)=\sum_{j=1}^m d_j(t)w^j$, so that $\pa_t\theta^m(t)=\sum_{j=1}^m \dot{d}_j(t)w^j$.

  By choose $t=0$ in \eqref{app_1}, we have
  \begin{equation}\label{app_2}
\begin{aligned}
 &(\pa_t\theta^m(0),\psi)_{\mathcal{H}^0(0)}+k(\nabla_{\mathcal{A}(0)}\theta^m(0), \nabla_{\mathcal{A}(0)}\psi)_{\mathcal{H}^0(0)}+(|\mathcal{N}(0)|\theta^m(0),\psi)_{0,\Sigma} \\
 &=\int_\Om F^8(0)\psi J(0) + \int_\Sigma F^9(0)\psi |\mathcal{N}(0)|
  \end{aligned}
  \end{equation}
  for each $\psi\in \mathcal{H}_m^1(0)$. Here $\theta^m_0$ is the projection of $\theta_0$ onto $\mathcal{H}_m^1(0)$. Then we take $\psi=w^j$ so that
  \[
  \begin{aligned}
  A_{kj}\dot{d}_j(0)+B_{k,j}d_j(0) = \int_\Om F^8(0)w^k\psi J(0) + \int_\Sigma F^9(0) w^k |\mathcal{N}(0)|,
  \end{aligned}
  \]
  where $A_{kj} = (w^k,w^j)_{\mathcal{H}^0(0)}$, $B_{k,j} = k(\nabla_{\mathcal{A}(0)}w^k, \nabla_{\mathcal{A}(0)}w^j)_{\mathcal{H}^0(0)}+(|\mathcal{N}(0)|w^k,w^j)_{0,\Sigma}$. Due to the Lemma \ref{lem:basis_initial}, the matrix $A = (A_{kj})$ is invertible. So we get the formula of $\dot{d}_j(0)$ and $\pa_t\theta^m(0)$.

\begin{proposition}\label{prop:bound_dtu}
  The sequences $\{\pa_t\theta^m(0)\}$ constructed in \eqref{app_2} are uniformly bounded such that
  \begin{equation}\label{bound_dtu}
  \|\pa_t\theta^m(0)\|_{\mathcal{H}^0}\lesssim \|\pa_t\theta(0)\|_{H^0}+\|\theta_0\|_{L^2(\Sigma)}
  \end{equation}
  uniformly for $\varepsilon\le 1$.
\end{proposition}
\begin{proof}
 From \eqref{equality},
  \begin{equation}\label{app_3}
\begin{aligned}
(\pa_t\theta(0),\psi)_{\mathcal{H}^0}+k(\nabla_{\mathcal{A}(0)}\theta_0, \nabla_{\mathcal{A}(0)}\psi)_{\mathcal{H}^0}+(\theta_0|\mathcal{N}(0)|,\psi)_{0,\Sigma}\\
=\int_\Om F^8(0)\psi J(0) + \int_\Sigma F^9(0)\psi |\mathcal{N}(0)|,
\end{aligned}
  \end{equation}
  for each $\psi\in \mathcal{H}_m^1(0)$. Since $\left\{\frac{w^k}{\sqrt{\lambda_k}}\right\}_{k=1}^\infty$ is a complete orthonormal basis of $\mathcal{H}^1(0)$, we can represent $\theta_0$ and also $\theta_0^m$ by
  $
  \theta_0=\sum_{j=0}^\infty a_j \frac{w^j}{\sqrt{\lambda_j}}$, and $\theta_0^m=\sum_{j=0}^m a_j \frac{w^j}{\sqrt{\lambda_j}}$
  for some $a_j\in\mathbb{R}$.
  For each $\psi=w^j/\sqrt{\lambda_j}$, $j=1, \ldots, m$, we use the orthogonality of basis $\{w^k/\sqrt{\lambda_k}\}$ to arrive that
  $\left(\theta_0, \frac{w^j}{\sqrt{\lambda_j}}\right)_{\mathcal{H}^1}= a_j=\left(\theta_0^m, \frac{w^j}{\sqrt{\lambda_j}}\right)_{\mathcal{H}^1}$
  and hence $(\theta_0, \psi)_{\mathcal{H}^1}=(\theta_0^m, \psi)_{\mathcal{H}^1} $
  for any $\psi\in \mathcal{H}^1_m(0)$ by linear superposition.
  So the subtraction between \eqref{app_2} and \eqref{app_3} gives us that
  \begin{equation}\label{eq:g_1}
  (\pa_t\theta^m(0)-\pa_t\theta(0),\psi)_{\mathcal{H}^0}=0.
  \end{equation}

  By taking $\psi=\pa_t\theta^m(0)$ in \eqref{eq:g_1}, we use the Cauchy's inequality to deduce that
  \begin{equation}
   \|\pa_t\theta^m(0)\|_{\mathcal{H}^0}^2
  \lesssim \|\pa_t\theta(0)\|_{H^0}\|\pa_t\theta^m(0)\|_{\mathcal{H}^0}
  \end{equation}
  implying \eqref{bound_dtu}.
  \end{proof}

  So we can show that
  \begin{proposition}\label{prop:converge_dtu0}
It holds that $\pa_t\theta^m(0)$ converges to $\pa_t\theta(0)$ weakly in $\mathcal{H}^0$ as $m\to \infty$.
   \end{proposition}
  \begin{proof}
  The uniform boundedness for $\pa_t\theta^m(0)$ in \eqref{bound_dtu} allows us to take a weak limit in $\mathcal{H}^0$, up to an extraction of a subsequence,
  $\pa_t\theta^m(0)\rightharpoonup w\quad\text{in}\ \mathcal{H}^0$,
  for some $w\in {}_0H^0$.
  By passing to the limit $m\to \infty$ for \eqref{eq:g_1},
  we have
  \begin{equation}\label{eq:in_dtu}
  (w-\pa_t\theta(0), \psi)_{\mathcal{H}^0}=0
  \end{equation}
  for any $\psi\in \mathcal{H}^1(0)$.  We now switch the domain $\Om$ to the initial state of original deformable domain $\Om(0)$ by the inverse mapping $\Phi(0)^{-1}$ and let $
  \tilde{w}\circ \Phi(0)=w,\ \widetilde{\pa_t\theta}(0)\circ\Phi(0)=\pa_t\theta(0)$,
  so that $\tilde{w}\in L^2(\Om(0))$ and
  $
  \widetilde{\pa_t\theta}(0)\in H^1(\Om(0))\simeq\{\theta\in L^2(\Om(0)): \|\nabla \theta\|_{L^2(\Om(0))}^2+\|\theta |\mathcal{N}(0)|\|_{L^2(\Sigma)}^2<\infty\}$.
  We choose $\psi$ in the closure of $C_{c}^\infty(\Om)$ under $H^1$ norm, so that \eqref{eq:in_dtu} is reduce to $
  (\tilde{w}-\widetilde{\pa_t\theta}(0), \psi)_{L^2(\Om(0))}=0$.
  Since $C_{c}^\infty(\Om)$ is dense in $L^2(\Om)$, we have that
  $\tilde{w}-\widetilde{\pa_t\theta}(0)=0$ implying $w=\pa_t\theta(0)$ by $\Phi(0)^{-1}$.
  \end{proof}


Based on the weak formula of
\begin{align}\label{eq:initial_p2}
  (\pa_t^2\theta, \psi)_{\mathcal{H}^0} + k (\nabla_{\mathcal{A}} \pa_t\theta, \nabla_{\mathcal{A}}\psi)_{\mathcal{H}^0} + (\pa_t\theta, \psi |\mathcal{N}|)_{L^2(\Sigma)} = \left<\mathscr{F}_2, \psi\right>
\end{align}
for any $\psi \in \mathcal{H}^1$, where
\[
\begin{aligned}
\left<\mathscr{F}_2, \psi\right> &= \int_{\Om} \pa_tF^8 \psi J + (F^8-\pa_t\theta \psi \pa_tJ) -k \nabla_{\pa_t\mathcal{A}} \theta \cdot \nabla_{\mathcal{A}}\psi J + \nabla_{\mathcal{A}} \theta \cdot \nabla_{\pa_t\mathcal{A}}\psi J + \nabla_{\mathcal{A}} \theta \cdot \nabla_{\mathcal{A}}\psi \pa_t J\\
&\quad+ \int_{\Sigma}\pa_tF^9 \psi -k \pa_t\mathcal{N} \cdot \nabla_{\mathcal{A}} \nabla \theta + k \mathcal{N} \cdot\nabla_{\pa_t\mathcal{A}} \nabla \theta - \theta \pa_t|\mathcal{N}|,
\end{aligned}
\]
we apply the two--tier method of Galerkin approximation to construct $\pa_t\theta(0)$. We assume that $(\pa_t\theta)^m : = \sum_{j=1}^m e_jw^j$, where $e_j$ is the coefficient. Then we construct $\pa_t(\pa_t\theta)^m=\sum_{j=1}^m \dot{e}_jw^j$, $t\ge0$ and the initial data $\pa_t(\pa_t\theta)^m(0)=\sum_{j=1}^m \dot{e}_j(0)w^j\in \mathcal{H}^1_m(0)$. When $t=0$, $\pa_t(\pa_t\theta)^m(0)$ is the projection of $\pa_t^2\theta(0)$ onto $\mathcal{H}^1_m(0)$ and $\pa_t(\pa_t\theta)^m(0)$ is supposed to satisfying
  \begin{equation}\label{app_5}
  (\pa_t(\pa_t\theta)^m(0),\psi)_{\mathcal{H}^0}+k(\nabla_{\mathcal{A}(0)}(\pa_t\theta)^m(0),\nabla_{\mathcal{A}(0)}\psi)_{\mathcal{H}^0} + ((\pa_t\theta)^m(0),\psi |\mathcal{N}(0)|)_{L^2(\Sigma)}
=\mathscr{F}_2(\psi),
  \end{equation}
  then we take $\psi=w^j(0)$ so that \eqref{app_5} is reduced to
  \[
  \begin{aligned}
  \dot{e}_j(0)+\lambda_je_j(0)
=\mathscr{F}_2(w^j).
  \end{aligned}
  \]
 Consequently, we have the formula of $\pa_t(\pa_t\theta)^m(0)$.

We now show the uniform bound and convergence for $\pa_t(\pa_t\theta)^m(0)$.
  \begin{proposition}\label{prop:uniform_bound_dt2u}
\begin{align}\label{est:uniform_bound_dt2u}
\|\pa_t(\pa_t\theta)^m(0)\|_{\mathcal{H}^0}\lesssim \|\pa_t(\pa_t\theta)(0)\|_{H^0}+\|\pa_t\theta(0)\|_{H^1}+\|\pa_t\theta(0)\|_{L^2(\Sigma)}.
\end{align}
   \end{proposition}
  \begin{proof}
  By orthogonality of the basis $\{w^k/\sqrt{\lambda_k}\}$,  we have
  \[
  \begin{aligned}
((\pa_t\theta)^m(0),\psi)_{\mathcal{H}^1}&=k(\nabla_{\mathcal{A}(0)}(\pa_t\theta)^m(0),\psi))+( (\pa_t\theta)^m(0),\psi |\mathcal{N}(0)|)_{L^2(\Sigma)}\\
&=k(\nabla_{\mathcal{A}(0)}(\pa_t\theta)(0),\psi))+( (\pa_t\theta)(0),\psi |\mathcal{N}(0)|)_{L^2(\Sigma)}\\
&=((\pa_t\theta)(0),\psi)_{\mathcal{H}^1}
  \end{aligned}
  \]
  for each $\psi\in\mathcal{H}^1_m(0)$. So the difference between \eqref{app_5} and \eqref{eq:initial_p2} at $t=0$ gives us that
  \begin{equation}\label{eq:g_2}
  \begin{aligned}
(\pa_t(\pa_t\theta)^m(0)-\pa_t^2\theta(0),\psi)_{\mathcal{H}^0}=0.
  \end{aligned}
  \end{equation}
  After taking $\psi=(\pa_t\theta)^m(0)$ in \eqref{eq:g_2}, we use the Cauchy's inequality to deduce that
  \begin{equation}
  \begin{aligned}
\|\pa_t(\pa_t\theta)^m(0)\|_{\mathcal{H}^0}\lesssim \|\pa_t^2\theta(0)\|_{H^0}
  \end{aligned}
  \end{equation}
  implying \eqref{est:uniform_bound_dt2u}
  uniformly for $m$.
  \end{proof}
  \begin{proposition}\label{prop:converge_dt2u0}
  $\pa_t(\pa_t\theta)^m(0)$ converges to $\pa_t^2\theta(0)$ weakly in $\mathcal{H}^0_\sigma$ as $m\to \infty$.
\end{proposition}
\begin{proof}
  The uniform boundedness in Proposition \ref{prop:uniform_bound_dt2u} for $(\pa_t\theta)^m$ allows us to take a weak limit in $H^0$. Using the argument as in the proof of Proposition \ref{prop:converge_dtu0}, we have that up to an extraction of a subsequence, $\pa(\pa_t\theta)^m(0)\rightharpoonup \pa_t^2\theta)(0)$ in $H^0$.

  By passing the limit $m\to\infty$ for \eqref{app_5}, we can  obtain
  \[
\begin{aligned}
(\pa_t^2\theta(0),\psi)_{\mathcal{H}^0}+k(\nabla_{\mathcal{A}(0)}(\pa_t\theta)(0),\nabla_{\mathcal{A}(0)}\psi)_{\mathcal{H}^0} + ((\pa_t\theta)(0),\psi |\mathcal{N}(0)|)_{L^2(\Sigma)}=\mathscr{F}_2(\psi),
\end{aligned}
  \]
  for each $\psi\in \mathcal{H}^1(0)$.
\end{proof}

\begin{remark}
  The convergence of $\pa_t^2\theta^m(0)$ to $\pa_t^2\theta(0)$ in $H^0(\Om)$ is actually used to prove Theorem \ref{thm:higher order} in the step of applying Theorem \ref{thm:linear_low}, since the convergence $\pa_t(\pa_t\theta)^m(0)\to \pa_t(\pa_t\theta)(0)$ in $H^0$ is necessary. More precisely, following the procedure of proving Theorem \ref{thm:linear_low}, we need to construct the approximate solution $(\pa_t\theta)^m$, and then construct $\pa_t(\pa_t\theta)^m$ based on $\pa_t(\pa_t\theta)^m(0)$.
\end{remark}

\subsection{Construction of Strong Solutions to \eqref{eq:linear_heat}}

For some $T>0$, we define
\[
\begin{aligned}
  &\mathfrak{D}(\eta):=\|\eta\|_{L^2W^{3-1/q_+,q_+}}^2+\|\pa_t\eta\|_{L^2W^{3-1/q_-,q_-}}^2+\sum_{j=0}^2\|\pa_t^j\eta\|_{L^2H^{3/2-\alpha}}^2+\|\pa_t^3\eta\|_{L^2H^{1/2-\alpha}}^2,\\
  &\mathfrak{E}(\eta):=\|\eta\|_{L^\infty W^{3-1/q_+,q_+}}^2+\|\pa_t\eta\|_{L^\infty H^{3/2+(\varepsilon_--\alpha)/2}}^2+\sum_{j=0}^2\|\pa_t^j\eta\|_{L^\infty H^1}^2, \quad \mathfrak{K}(\eta):=\mathfrak{D}(\eta)+\mathfrak{E}(\eta).
\end{aligned}
\]

We say $\theta \in L^\infty W^{2,q_+} $ is a strong solution to \eqref{eq:modified_linear}, provided that $\pa_t^k \theta \in L^\infty H^0 \cap L^2 H^1$,  $k=0, 1$, such that $\theta$ satisfies \eqref{eq:linear_heat} when $j=0$.

In the following theorem, we give the construction of strong solutions to \eqref{eq:linear_heat} by Galerkin method. The strategy consists four main steps. The first is to use the basis for $\mathcal{H}^1(t)$ and ODE theory to construct $\theta^m$ in a finite dimensional space with dimension $m$. Secondly, we use energy estimate to find a uniform bound for $\theta^m$; we then take time derivative to find the formulation of $\pa_t\theta^m$ and use the energy estimates to find the uniform bound for $\pa_t\theta^m$. Thirdly, we can pass to the limit $m \to \infty$ to get the weak formulation, and use Theorem \ref{thm:elliptic_1} to obtain the strong solution. Lastly, we construct the weak solution $\pa_t\theta$ to \eqref{eq:linear_heat} when $j=1$.

  \begin{theorem}\label{thm:linear_low}
Assume the initial data are constructed in Section \ref{sec:initial_linear}.
Suppose that $\mathfrak{K}(\eta)\le\delta$ is smaller than $\delta_0$ in Lemma \ref{lem:equivalence_norm} and Theorem 4.7 in \cite{GT2020}.
Then there exists a unique strong solution $\theta$ solving \eqref{eq:linear_heat}. Moreover, the solution obeys the estimate
\begin{equation}\label{est:bound_linear}
  \begin{aligned}
  &\sum_{j=0}^1\|\pa_t^j\theta\|_{L^\infty H^0}^2+\|\theta\|_{L^2W^{2,q_+}}^2+\sum_{j=0}^1\left(\|\pa_t^j\theta\|_{L^2 H^1}^2+\|\pa_t^j\theta\|_{L^2 H^0(\Sigma)}^2\right)\\
  &\lesssim \exp\{\mathfrak{E}(\eta)T\}\bigg(\mathcal{E}(0)+\|(F^8+F^9)(0)\|_{(H^1)^\ast}^2+(1+\mathfrak{E}(\eta))(\|F^8\|_{L^2L^{q_+}}^2\\
  &\quad+\|F^9\|_{L^2W^{1-1/q_+,q_+}}^2)+\|\pa_t(F^8+F^9)\|_{(\mathcal{H}^1_T)^{\ast}}^2\bigg),
  \end{aligned}
  \end{equation}
  where "$\lesssim$" omits a constant depending on $(g, \sigma, \ell, |\Om|)$.
Moreover, $\pa_t\theta$ satisfies
\begin{equation}\label{eq:weak_dt_u}
\left\{
\begin{aligned}
  &\pa_t^2\theta- k\Delta_{\mathcal{A}}\pa_t\theta=\pa_tF^8+G^8 \quad&\text{in}&\quad\Om,\\
  &k\mathcal{N}\cdot\nabla_{\mathcal{A}}\pa_t\theta + \pa_t\theta |\mathcal{N}| =\pa_tF^9+G^9 \quad&\text{on}&\quad\Sigma,\\
  &\pa_t\theta  =0 \quad&\text{on}&\quad\Sigma_s,
\end{aligned}
\right.
\end{equation}
in the pressureless weak sense of \eqref{eq:weak_limit_1}, where $G^8$ and $G^9$ are defined by
\[
\begin{aligned}
G^8&=k\dive_{\mathcal{A}}\nabla_{\pa_t\mathcal{A}}\theta + k\dive_{\pa_t\mathcal{A}}\nabla_{\mathcal{A}} \theta,\\
G^9&= - k\nabla_{\pa_t\mathcal{A}}\theta \cdot\mathcal{N}- k\nabla_{\mathcal{A}}\theta \cdot\pa_t\mathcal{N}- \theta\pa_t|\mathcal{N}|.
\end{aligned}
\]
More precisely, \eqref{eq:weak_dt_u} holds in the weak sense of
\begin{equation}\label{eq:weak_dt_u_1}
\begin{aligned}
  &\left<\pa_t^2\theta,\psi \right>_\ast+k\int_0^T(\pa_t\theta,\psi )_{\mathcal{H}^1}
&=\int_0^T\int_\Om (\pa_tF^8 + G^8) \cdot \psi J +\int_0^T \int_{-\ell}^\ell  (\pa_tF^9+G^9) \psi
\end{aligned}
\end{equation}
for each $\psi\in \mathcal{H}^1_T$.
  \end{theorem}

\begin{proof}
Step 1 -- Galerkin Setup.
In order to utilize the Galerkin method, we choose the same basis $\{w^j(t)\}_{j=1}^\infty$ of $\mathcal{H}^1$ constructed in Lemma \ref{lem:basis_initial}.
For any integer $m\ge1$, we define the finite dimensional space
$
\mathcal{H}^1_m:=\text{span}\{w^1, \cdots, w^m\}\subseteq \mathcal{H}^1$,
 and we write
$
\mathcal{P}^m: \mathcal{H}^1\rightarrow \mathcal{H}^1_m
$
for the $\mathcal{H}^1$ orthogonal projection onto $\mathcal{H}^1_m$. Clearly, for each $w\in \mathcal{H}^1$, we have that $\mathcal{P}^m w\rightarrow w$ as $m\rightarrow\infty$.

Step 2 -- Solving the Approximate Problem.
For each $m\ge1$, we define an approximate solution $
\theta^m(t):=d^m_j(t)w^j$, with $d^m_j: [0, T]\rightarrow\mathbb{R}\ \text{for}\ j=1, \dots, m$,
where we used the Einstein convention of summation of the repeated index $j$.

We want to choose the coefficients $d^m_j(t)$ so that
\begin{equation}\label{eq:galerkin}
\begin{aligned}
  (\pa_t\theta^m,\varphi)_{\mathcal{H}^0}+ k(\nabla_{\mathcal{A}}\theta^m,\nabla_{\mathcal{A}}\varphi)_{\mathcal{H}^0}+(\theta^m,\varphi|\mathcal{N}|)_{0, \Sigma}=\int_{\Om}F^8\cdot \varphi J+\int_{\Sigma}F^9\varphi ,
\end{aligned}
  \end{equation}
  for each $\varphi\in \mathcal{H}^1_m(t)$. We supplement the initial data
  $\theta^m(0)=\mathcal{P}^m_0\theta_0\in\mathcal{H}^1_m(0)$.

  Then we see that \eqref{eq:galerkin} is equivalent to the equation of $d^m_j$ given by
  \begin{equation}\label{eq:galerkin_2}
\begin{aligned}
 \dot{d}^m_i(w^i,w^j)_{\mathcal{H}^0}+ kd^m_i\left((\nabla_{\mathcal{A}}w^i,\nabla_{\mathcal{A}}w^j)_{\mathcal{H}^0}+(w^i,w^j|\mathcal{N}|)_{0,\Sigma}\right)
  =\int_{\Om}F^8\cdot w^jJ+\int_{\Sigma}F^9w^j ,
\end{aligned}
  \end{equation}
  for $i,j=1, \dots, m$. Since $\{w^j\}_{j=1}^\infty$ is a basis of $\mathcal{H}^1(t)$, the $m\times m$ matrix with $j, k$ entry $(w^i, w^j)_{\mathcal{H}^0}$ is invertible. Thus the initial data $d^m_i(0)$ is determined uniquely via
  \[
  d^m_i(0)(w^i, w^j)_{\mathcal{H}^1}=(\mathcal{P}_0^m\theta_0, w^j)_{\mathcal{H}^1}.
  \]

  Then we view \eqref{eq:galerkin_2} as an differential system of the form
  \begin{equation}\label{eq:integral}
  \dot{d}^m(t)+\mathfrak{B}(t)d^m(t)=\mathfrak{F}(t).
  \end{equation}
  The $m\times m$ matrix $\mathfrak{B}$ belongs to $C^1([0,T])$   and the forcing term $\mathfrak{F}\in C^1([0,T])$. By the standard linear theory of ODE, \eqref{eq:galerkin_2} has a unique solution $d^m\in C^{1, 1}([0,T])$ achieving the initial data $d^m(0)$.

Step 3 -- Energy Estimates for $\theta^m$.
  By construction, $\theta^m(t)\in\mathcal{H}^1_m(t)$, so we may choose $w=\theta^m$ as a test function in \eqref{eq:galerkin} such that we have
\[
\begin{aligned}
 \frac{d}{dt}\frac12\|\theta^m\|_{\mathcal{H}^0}^2+\|\nabla_{\mathcal{A}}\theta^m\|_{\mathcal{H}^0}^2+ \int_{\Sigma}|\theta^m|^2|\mathcal{N}|
  =\frac12\int_\Om\pa_tJK|\theta^m|^2J+\int_{\Om}F^8\cdot \theta^mJ+\int_{\Sigma}F^9\theta^m.
\end{aligned}
\]

Then using the H\"older inequality with $\frac1{q_+}+\frac2{\varepsilon_+}=1$, and the Cauchy inequality,  we have that
\[
\begin{aligned}
  &\frac{d}{dt}\frac12\|\theta^m\|_{\mathcal{H}^0}^2 + k\|\nabla_{\mathcal{A}}\theta^m\|_{\mathcal{H}^0}^2+\|\theta^m\|_{H^0(\Sigma)}^2\\
  &\le \frac12\|\pa_tJK\|_{L^\infty}\|\theta^m\|_{\mathcal{H}^0}^2+ \|J\|_{L^\infty(\Sigma)}\|F^8\|_{L^{q_+}}\|\theta^m\|_{L^{2/\varepsilon_+}} +\|F^9\|_{L^{q_+}(\Sigma)}\|\theta^m\|_{L^{2/\varepsilon_+}(\Sigma)}.
\end{aligned}
  \]
  Then we employ the Cauchy's inequality, Sobolev embedding theory $H^1(U)\hookrightarrow L^{2/\varepsilon_+}(U)$ for $U=\Om, \Sigma$, and the usual trace theory to derive that
  \begin{equation}
\begin{aligned}
  &\frac{d}{dt}\frac12\|\theta^m\|_{\mathcal{H}^0}^2+\frac12(\|\theta^m\|_{\mathcal{H}^1}^2+\|\theta^m\|_{H^0(\Sigma)})\\
  &\le \frac12\|\pa_tJK\|_{L^\infty}\|\theta^m\|_{\mathcal{H}^0}^2+C(1+\|\eta\|_{W^{3-1/q_+,q_+}}^2)(\|F^8\|_{L^{q_+}}^2+\|F^9\|_{L^{q_+}}^2).
\end{aligned}
  \end{equation}

  Since $\mathcal{P}_0$ is projected orthogonal, then
  \[
  \|\theta^m(0)\|_{\mathcal{H}^0}\lesssim \|\theta^m(0)\|_{H^0}\lesssim \|\theta_0\|_1\lesssim\|\theta_0\|_{W^{2,q_+}}.
  \]

Then the Gronwall's lemma and Lemma \ref{lem:equivalence_norm} reveal that
  \begin{equation}\label{eq:xi_m}
\begin{aligned}
  &\sup_{0\le t\le T}\|\theta^m\|_{H^0}^2 + \|\theta^m\|_{L^2H^1(\Om)}^2 \\
  &\lesssim \exp(C_0(\eta))\bigg(\|\theta_0\|_{W^{2,q_+}}^2+(1+\|\eta\|_{L^\infty W^{3-1/q_+,q_+}}^2)\Big(\|F^8\|_{L^2L^{q_+}}^2+\|F^9\|_{L^2L^{q_+}}^2\Big)\bigg),
\end{aligned}
  \end{equation}
  where $C_0(\eta)=\int_0^T\|\pa_tJK\|_{L^\infty}\lesssim T\|\pa_t\eta\|_{L^\infty H^{3/2+(\varepsilon_--\alpha)/2}}$.

Step 4 -- Energy Estimate for $\pa_t\theta^m$.
  Suppose that $w=b_j^mw^j\in\mathcal{H}^1_m(t)$ for $b_j^m\in C^{1,1}([0, T])$. We now use this $w$ in \eqref{eq:galerkin}, temporally differentiate the resulting equation, and then subtract this from the equation \eqref{eq:galerkin} with test function $\pa_tw$. This eliminates the terms of $\pa_tw$ and leaves us with the equality
  \begin{equation}\label{eq:pa_tu_m_1}
  \begin{aligned}
&\left(\pa_t^2\theta^m,w\right)_{\mathcal{H}^0}+ k(\nabla_{\mathcal{A}}\pa_t\theta^m,\nabla_{\mathcal{A}}w)_{\mathcal{H}^0}+\int_{\Sigma}\pa_t\theta^m w|\mathcal{N}|\\
&=-(\pa_t\theta^m, w\pa_tJK)_{\mathcal{H}^0}-\int_{\Sigma}\theta^m w\pa_t|\mathcal{N}|\\
&\quad-k\int_\Om(\nabla_{\pa_t\mathcal{A}}\theta^m \cdot \nabla_{\mathcal{A}}w+\nabla_{\mathcal{A}}\theta^m \cdot \nabla_{\pa_t\mathcal{A}}w+\pa_tJK\nabla_{\mathcal{A}}\theta^m \cdot \nabla_{\mathcal{A}}w)J\\
&\quad+\int_{\Om}\pa_tF^8\cdot wJ+\int_{\Sigma}\pa_tF^9w +\int_{\Om}F^8\cdot w \pa_tJ.
  \end{aligned}
  \end{equation}

 By plugging the test function $w=\pa_t\theta^m$ into \eqref{eq:pa_tu_m_1},
 we have
  \begin{equation}\label{eq:e_2}
  \begin{aligned}
  \frac{d}{dt}\frac12\|\pa_t\theta^m\|_{\mathcal{H}^0}^2 + k \|\nabla_{\mathcal{A}}\pa_t\theta^m\|_{\mathcal{H}^0}^2 + \int_{\Sigma}|\pa_t\theta^m||\mathcal{N}|
  =I+II+III,
  \end{aligned}
  \end{equation}
  where
  \[
  I=-\frac12\int_\Om|\pa_t\theta^m|^2\pa_tJ - \int_{\Sigma}\theta^m \pa_t\theta^m\pa_t|\mathcal{N}|,
  \]
  \[
  \begin{aligned}
  II=-k\int_\Om(\nabla_{\pa_t\mathcal{A}}\theta^m \cdot \nabla_{\mathcal{A}}\pa_t\theta^m+\nabla_{\mathcal{A}}\theta^m \cdot \nabla_{\pa_t\mathcal{A}}\pa_t\theta^m+\pa_tJK\nabla_{\mathcal{A}}\theta^m \cdot \nabla_{\mathcal{A}}\pa_t\theta^m)J,
  \end{aligned}
  \]
  \[
  \begin{aligned}
  III=\int_{\Om}\pa_tF^8\cdot \pa_t\theta^mJ+\int_{\Sigma}\pa_tF^9\pa_t\theta^m |\mathcal{N}|+\int_{\Om}F^8\cdot \pa_t\theta^m \pa_tJ.
  \end{aligned}
  \]

  We now estimate each term of right-hand side in \eqref{eq:e_2}. We first employ the H\"older inequality and Sobolev embedding theory to derive that
  \begin{equation}\label{est:dt_u}
  \begin{aligned}
  I\le C(\|\pa_tJK\|_{L^\infty(\Om)})\|\pa_t\theta^m\|_{\mathcal{H}^0}^2 + C(\|\pa_tJ\|_{L^\infty(\pa\Om)})\|\theta^m\|_{H^0(\Sigma)}\|\pa_t\theta^m\|_{H^0(\Sigma)}\\
  \le C(\|\pa_t\eta\|_{H^{3/2+(\varepsilon_--\alpha)/2}}^2) (\|\pa_t\theta^m\|_{\mathcal{H}^0}^2 + \|\theta^m\|_{H^0(\Sigma)}^2) + \frac{1}{2}\int_{\Sigma}|\pa_t\theta^m|^2J.
  \end{aligned}
  \end{equation}

  Similarly, we estimate
  \begin{equation}
  \begin{aligned}
  II&\lesssim\|\pa_t\mathcal{A}\|_{L^\infty}\|J\|_{L^\infty}(\|\theta^m\|_{\dot{H}^1}\|\nabla_{\mathcal{A}}\pa_t\theta^m\|_{\mathcal{H}^0}+\|\pa_t\theta^m\|_{\dot{H}^1}\|\nabla_{\mathcal{A}}\theta^m\|_{\mathcal{H}^0})\\
  &\quad + \|\pa_tJK\|_{L^\infty}\|\nabla_{\mathcal{A}}\pa_t\theta^m\|_{\mathcal{H}^0}\|\nabla_{\mathcal{A}}\theta^m\|_{\mathcal{H}^0}\\
  &\le C(\|\pa_t\eta\|_{H^{3/2+(\varepsilon_--\alpha)/2}}^2)\|\nabla_{\mathcal{A}}\theta^m\|_{\mathcal{H}^0}^2 + \frac{1}{2}\|\nabla_{\mathcal{A}}\pa_t\theta^m\|_{\mathcal{H}^0}^2.
  \end{aligned}
  \end{equation}

  We now turn to $III$. We first use the dual space estimate to bound terms related to one time derivative of forces:
\begin{equation}
\begin{aligned}
  \int_{\Om}\pa_tF^8\cdot(\pa_t\theta^m)J+\int_{\Sigma}\pa_tF^9\cdot(\pa_t\theta^m)
  \lesssim \|\pa_t(F^8+F^9)\|_{(\mathcal{H}^1)^{\ast}}\|\pa_t\theta^m\|_{\mathcal{H}^1}.
\end{aligned}
\end{equation}
For terms with $F^8$, we estimate it by H\"older inequality with $\f1{q_+}+\f{\varepsilon_+}2=1$ and Sobolev embedding with $H^1\hookrightarrow L^{\varepsilon_+/2}$:
\begin{equation}\label{est:iii}
\begin{aligned}
  \int_{\Om}\left[\pa_tJF^8(\pa_t\theta^m)\right]
  &\lesssim \|\pa_tJ\|_{L^\infty}\|F^8\|_{L^{q_+}(\Om)}\|\pa_t\theta^m\|_{L^{\varepsilon_+/2}(\Om)}\\
  &\lesssim\|\pa_t\eta\|_{H^{3/2+(\varepsilon_--\alpha)/2}}\|F^8\|_{L^{q_+}(\Om)}\|\pa_t\theta^m\|_{\mathcal{H}^1}.
\end{aligned}
\end{equation}

  Thus, by the Schwarz inequality and combining \eqref{est:dt_u}--\eqref{est:iii}, we have the energy structure
  \begin{equation}
  \begin{aligned}
&\frac{d}{dt}\frac12\|\pa_t\theta^m\|_{\mathcal{H}^0}^2+\frac14\|\pa_t\theta^m\|_{\mathcal{H}^1}^2\\
&\le C\|\pa_t\eta\|_{H^{3/2+(\varepsilon_--\alpha)/2}}\left(\frac12\|\pa_t\theta^m\|_{\mathcal{H}^0}^2\right)+C\|\pa_t\eta\|_{H^{3/2+(\varepsilon_--\alpha)/2}}^2(\|\theta^m\|_{\mathcal{H}^1}^2+\|\theta^m\|_{H^0(\Sigma)}^2)\\
&\quad +C\|\pa_t(F^8+F^9)\|_{(\mathcal{H}^1)^{\ast}}^2+C\|\pa_t\eta\|_{H^{3/2+(\varepsilon_--\alpha)/2}}^2\|F^8\|_{L^{q_+}}^2.
  \end{aligned}
  \end{equation}

  Note that the derivative $\{\|\pa_t\theta^m(0)\|_{H^0}\}$ are uniformly bounded , as proved in Theorem \ref{thm:initial_convergence}. We then employ the Gronwall's inequality, H\"older inequality, the smallness of $\mathfrak{K}(\eta)$ and the initial data $\pa_t\theta(0)$ to derive that
  \begin{equation}\label{eq:pa_tu_m}
  \begin{aligned}
  &\sup_{0\le t\le T}\|\pa_t\theta^m\|_{0}^2+\|\pa_t\theta^m\|_{L^2H^1}^2 \\
  &\lesssim \exp\{\|\pa_t\eta\|_{L^\infty H^{3/2+\varepsilon_-/2}}T\}\bigg(\|\pa_t\theta(0)\|_{L^2(\Om)}^2+\|(F^8+F^9)(0)\|_{(H^1)^\ast}^2\\
  &\quad+\|\pa_t\eta\|_{L^\infty H^{3/2+\varepsilon_-/2}}^2(\|F^8\|_{L^2L^{q_+}}^2+\|F^9\|_{L^2W^{1-1/q_+,q_+}}^2+\|\pa_t(F^8+F^9)\|_{(\mathcal{H}^1_T)^{\ast}}^2)\bigg).
  \end{aligned}
  \end{equation}

Step 5 -- Passing to the Limit.
We now utilize the energy estimates \eqref{eq:xi_m} and \eqref{eq:pa_tu_m} to pass to the limit $m\rightarrow\infty$. According to Lemma \ref{lem:equivalence_norm} and energy estimates, we have that the sequence $\{\theta^m\}$ and $\{\pa_t\theta^m\}$ are uniformly bounded both in $L^2{}_0H^1$, and $L^\infty H^0$.  Up to the extraction of a subsequence, we then know that
  \begin{align}
  \begin{aligned}
  \theta^m\rightharpoonup \theta\ \text{weakly-}\ \text{in}\ L^2 {}_0H^1,\quad \pa_t\theta^m\rightharpoonup\pa_t\theta\ \text{weakly in}\ L^2{}_0H^1,\\
  \theta^m\stackrel{\ast}\rightharpoonup \theta\ \text{weakly-}\ast\ \text{in}\ L^\infty H^0, \quad\pa_t \theta^m\stackrel{\ast}\rightharpoonup \pa_t\theta\ \text{weakly-}\ast\ \text{in}\ L^\infty H^0.
\end{aligned}
  \end{align}
  By lower semicontinuity, the energy estimates imply that
  \begin{align}\label{est:b_low}
  \|\theta\|_{L^\infty H^0}^2+\|\pa_t\theta\|_{L^\infty H^0}^2+\|\theta\|_{L^2 H^1}^2+\|\pa_t\theta\|_{L^2H^1}^2
  \end{align}
  is bounded by \eqref{eq:xi_m} and \eqref{eq:pa_tu_m}.

Step 6 -- Strong Solutions.
  Due to the convergence, we may pass to the limit in \eqref{eq:galerkin} for almost every $t\in [0, T]$ so that
  \begin{equation}\label{eq:weak_limit}
  \begin{aligned}
  (\pa_t\theta, w)_{\mathcal{H}^0}+ k(\nabla_{\mathcal{A}}\theta,\nabla_{\mathcal{A}}w)_{\mathcal{H}^0}+ \int_{\Sigma}\theta w |\mathcal{N}|
  =\int_{\Om}F^8\cdot wJ+\int_{\Sigma}F^9w
  \end{aligned}
  \end{equation}
  for each $w\in \mathcal{H}^1(t)$.  By the elliptic theory,
\begin{equation}\label{est:diss_4}
  \begin{aligned}
  \|\theta\|_{W^{2,q_+}}^2
  \lesssim\|\pa_t\theta\|_{H^0}^2 +\|\theta\|_{H^1}^2+ \|F^8\|_{L^{q_+}}^2+\|F^9\|_{W^{1-1/q_+,q_+}}^2.
  \end{aligned}
  \end{equation}
By \eqref{eq:xi_m}, \eqref{eq:pa_tu_m}, \eqref{est:b_low} and \eqref{est:diss_4}, we have
\begin{equation}\label{est:ellip_1}
\begin{aligned}
  \|\theta\|_{L^2W^{2,q_+}}^2
  & \lesssim\exp\{\|\pa_t\eta\|_{L^\infty H^{3/2+\varepsilon_-/2}}T\}\bigg(\|(F^8+F^9)(0)\|_{(H^1)^\ast}^2+\|\theta_0\|_{W^{2,q_+}(\Om)}\\
  &\quad +\|\pa_t\theta(0)\|_{H^{1+\varepsilon_-/2}}+(1+\|\pa_t\eta\|_{L^\infty H^{3/2+\varepsilon_-/2}}^2)(\|F^8\|_{L^2L^{q_+}}^2+\|F^9\|_{L^2W^{1-1/q_+,q_+}}^2)\\
  &\quad+(1+\|\pa_t\eta\|_{L^\infty H^{3/2+\varepsilon_-/2}}^2)\|\pa_t(F^8 + F^9)\|_{(\mathcal{H}^1_T)^{\ast}}^2\bigg),
\end{aligned}
\end{equation}
as well as
\begin{equation}\label{est:ellip_2}
\begin{aligned}
  \|\theta\|_{L^\infty W^{2,q_+}}^2
  & \lesssim\exp\{\|\pa_t\eta\|_{L^\infty H^{3/2+\varepsilon_-/2}}T\}\bigg(\|\theta_0\|_{W^{2,q_+}(\Om)}+\|\pa_t\theta(0)\|_{H^{1+\varepsilon_-/2}}+\|(F^8+F^9)(0)\|_{(H^1)^\ast}^2\\
  &\quad+(1+\|\pa_t\eta\|_{L^\infty H^{3/2+\varepsilon_-/2}}^2)(\|F^8\|_{L^\infty L^{q_+}}^2+\|F^9\|_{L^\infty W^{1-1/q_+,q_+}}^2 +\|F^8\|_{L^2L^{q_+}}^2\\
  &\quad+\|F^9\|_{L^2W^{1-1/q_+,q_+}}^2)+(1+\|\pa_t\eta\|_{L^\infty H^{3/2+\varepsilon_-/2}}^2)\|\pa_t(F^8 + F^9)\|_{(\mathcal{H}^1_T)^{\ast}}^2\bigg).
\end{aligned}
\end{equation}

Due to the bounds \eqref{eq:xi_m}, \eqref{eq:pa_tu_m}, and \eqref{est:ellip_1}, $\theta$ have one more regularity, besides the weak solutions, so that $\theta$ is a strong solution to \eqref{eq:linear_heat}.

Step 7 -- Weak solution for $(\pa_t\theta)$.
Now we seek to use \eqref{eq:pa_tu_m_1} to determine the PDE satisfied by $\pa_t\theta$.
We integrate \eqref{eq:pa_tu_m_1} temporally over $[0,T]$ and pass to the limit $m\rightarrow \infty$. The result equation is combined with the equation \eqref{eq:weak_limit}, with the test function $w\in L^2\mathcal{H}^1$, to derive that
\begin{equation}\label{eq:d_tu_1}
\begin{aligned}
&\left<\pa_t^2\theta, w\right>_{\ast}+\int_0^T\left[(\nabla_{\mathcal{A}}\pa_t\theta,\nabla_{\mathcal{A}}w)_{\mathcal{H}^0}+(\pa_t\theta,w |\mathcal{N}|)_{0,\Sigma} \right]\\
&=\int_0^T\int_{\Om}\left[\pa_tF^8\cdot w+\pa_tJKF^8\cdot w\right]J +\int_0^T\int_{\Sigma}[\pa_tF^9 w |\mathcal{N}|\\
&\quad-\int_0^T\int_\Om\frac{k}{2}(\nabla_{\pa_t\mathcal{A}}\theta \cdot \nabla_{\mathcal{A}}w+\nabla_{\mathcal{A}}\theta \cdot \nabla_{\pa_t\mathcal{A}}w+\pa_tJK\nabla_{\mathcal{A}}\theta \cdot\nabla_{\mathcal{A}}w)J\\
&\quad-(\pa_tJK\pa_t\theta, w)_{\mathcal{H}^0_T} - \int_0^T(\theta,w \pa_t|\mathcal{N}|)_{0,\Sigma}.
\end{aligned}
\end{equation}

Then the definition of $R$ that $
R^T=-\pa_tM^T(M)^{-T}=-\pa_tJKI_{2\times 2}-\pa_t\mathcal{A}\mathcal{A}^{-1}$,
and integration by parts yield that
\begin{equation}\label{eq:d_tu_2}
\begin{aligned}
 &\int_\Om\frac{k}{2}(\nabla_{\pa_t\mathcal{A}}\theta \cdot \nabla_{\mathcal{A}}w+\nabla_{\mathcal{A}}\theta \cdot \nabla_{\pa_t\mathcal{A}}w+\pa_tJK\nabla_{\mathcal{A}}\theta \cdot\nabla_{\mathcal{A}}w)J\\
 &=\int_\Om k(\nabla_{\pa_t\mathcal{A}}\theta-R\nabla_{\mathcal{A}}\theta)\cdot \nabla_{\mathcal{A}}wJ\\
 &=-k\left(\dive_{\mathcal{A}}(\nabla_{\pa_t\mathcal{A}}\theta-R\nabla_{\mathcal{A}}\theta),w\right)_{\mathcal{H}^0}+k\left<\nabla_{\pa_t\mathcal{A}}\theta\cdot\mathcal{N}+\nabla_{\mathcal{A}}\theta \cdot\pa_t\mathcal{N},w\right>_{-1/2,\Sigma}.
\end{aligned}
\end{equation}
Due to $R^T=-\pa_tJKI_{2\times 2}-\pa_t\mathcal{A}\mathcal{A}^{-1}$ and $\pa_j(J\mathcal{A}_{ij})=0$, we also have
\begin{align}\label{eq:d_tu_3}
 (\dive_{\mathcal{A}}R\nabla_{\mathcal{A}}\theta,w)_{\mathcal{H}^0} = (-\pa_tJK \Delta_{\mathcal{A}}\theta, w)_{\mathcal{H}^0} + (-\dive_{\pa_t\mathcal{A}}\nabla_{\mathcal{A}}\theta,w)_{\mathcal{H}^0}.
\end{align}

  By the above inequalities \eqref{eq:d_tu_1}--\eqref{eq:d_tu_3}, we have that
\begin{equation}\label{eq:weak_dt_u_2}
\begin{aligned}
 &\left<\pa_t^2\theta, w\right>_{\ast}+k\int_0^T(\nabla_{\mathcal{A}}\pa_t\theta,\nabla_{\mathcal{A}}w)_{\mathcal{H}^0} + \int_0^T(\pa_t\theta, w|\mathcal{N}|)_{0, \Sigma}\\
 &=k\int_0^T\int_\Om\left[\dive_{\mathcal{A}}\nabla_{\pa_t\mathcal{A}}\theta+ \dive_{\pa_t\mathcal{A}}\nabla_{\mathcal{A}}\theta\right]\cdot w J+\int_0^T\int_\Om\pa_tF^8 wJ \\
 &\quad-\int_0^T\int_\Sigma(k\nabla_{\pa_t\mathcal{A}}\theta \cdot\mathcal{N}+k\nabla_{\mathcal{A}}\theta\cdot\pa_t\mathcal{N} + \theta\pa_t|\mathcal{N}| - \pa_t F^9|) w\\
 &\quad-(\pa_tJK(\pa_t\theta -k\Delta_{\mathcal{A}}\theta -F^8), w)_{\mathcal{H}^0_T}.
\end{aligned}
\end{equation}
 Since $\theta$ is a strong solution to \eqref{eq:linear_heat}, we then plug the equations of \eqref{eq:linear_heat} into the last line of \eqref{eq:weak_dt_u_2} to cancel the terms of $\pa_tJ$ to arrive that the weak solution to \eqref{eq:weak_dt_u} follows from \eqref{eq:weak_dt_u_1} in the sense of \eqref{eq:weak_dt_u_1}.

  \end{proof}

\subsection{Higher Regularity}

In order to state our higher regularity results for the problem \eqref{eq:modified_linear}, we must be able to bound the forcing term that results from temporally differentiating \eqref{eq:modified_linear} one time. We consider the system with $j=1, 2$:
\begin{equation}\label{eq:linear_heat2}
\left\{
\begin{aligned}
  &\pa_t^{j+1}\theta - k \Delta_{\mathcal{A}} \pa_t^j\theta=F^{8,j}\quad&\text{in}&\quad\Om,\\
  &k\mathcal{N} \cdot\nabla_{\mathcal{A}} \pa_t^j\theta + \pa_t^j\theta |\mathcal{N}|=F^{9, j}\quad&\text{on}&\quad\Sigma,\\
  &\pa_t^j\theta = 0\quad&\text{on}&\quad\Sigma_s,
\end{aligned}
\right.
\end{equation}

Now, with the forcing terms of $F^{i,1}$ defined in \eqref{def:F_11}, we have the following estimate .
\begin{lemma}\label{lemma:est_force}
  The following estimates hold.
  \begin{equation}\label{est:force_11}
  \begin{aligned}
  \|F^{8,1}\|_{L^2 L^{q_-}}^2&\lesssim \|\pa_tF^8\|_{L^2L^{q_-}}^2+\mathfrak{E}(\eta)(\|\theta\|_{L^2W^{2,q_-}}^2+ \|\theta\|_{L^2H^1}^2)\\
  &\quad+\mathfrak{D}(\eta)(\|\theta_0\|_{W^{2,q_+}}^2+\|\theta\|_{L^2W^{2,q_-}}\|\pa_t \theta\|_{L^2H^1}),
  \end{aligned}
  \end{equation}
  \begin{equation}\label{est:force_41}
  \begin{aligned}
  \|F^{9,1}\|_{L^2W^{1-1/q_-,q_-}}^2&\lesssim \|\pa_tF^9\|_{L^2W^{1-1/q_-,q_-}}^2+\mathfrak{E}(\eta)\|\theta\|_{L^2W^{2,q_-}}^2.
  \end{aligned}
  \end{equation}
\end{lemma}
\begin{proof}

  According to the definition of $F^{8,1}$ and $F^{9,1}$ in \eqref{def:F_11}, we use Leibniz rule to rewrite $F^{i,1}$ as a sum of products for two terms. One term is a product of various derivatives of $\bar{\eta}$, and the other is linear for derivatives of $\theta$. Then for a.e. $t\in[0,T]$, we estimate these resulting products using the product estimates in \cite{GT2020}, the Sobolev embedding and trace theory, the smallness of $\mathfrak{K}(\eta)$ and the finite volume of $|\Om|$. Then the resulting inequality after integrating over $[0,T]$ reveals
  \[
  \begin{aligned}
  \|F^{8,1}\|_{L^2L^{q_-}}&\lesssim \|\pa_tF^8\|_{L^2L^{q_-}}+\|\pa_t\eta\|_{L^\infty H^{3/2+(\varepsilon_--\alpha)/2}}(1+\|\eta\|_{L^\infty W^{1-1/q_+, q_+}})\|\theta\|_{L^2W^{2,q_-}}\\
  &\quad+\|\pa_t\eta\|_{L^2W^{3-1/q_-, q_-}}\|\theta\|_{L^\infty H^{1+\varepsilon_-/2}}.
  \end{aligned}
  \]
  Then we use the Sobolev extension and restriction theory to have that
  \[
  \begin{aligned}
   \f{d}{dt}\|\theta\|_{H^{1+\varepsilon_-/2}}^2\le\|\theta\|_{H^{1+\varepsilon_-}}\|\pa_t \theta\|_{H^1}\lesssim\|\theta\|_{W^{2,q_-}}\|\pa_t \theta\|_{H^1},
  \end{aligned}
  \]
  which, after an integration by parts, implies
  \[
  \|\theta\|_{L^\infty H^{1+\varepsilon_-/2}}^2\lesssim\|\theta_0\|_{W^{2,q_+}}^2+\|\theta\|_{L^2W^{2,q_-}}\|\pa_t \theta\|_{L^2H^1}.
  \]
Thus we have the bounds for \eqref{est:force_11}. Similarly, we have the bounds for \eqref{est:force_41}.
\end{proof}

We now present the dual estimates for forcing terms.
\begin{lemma}\label{lem:est_force_dual}
The following estimates hold that
\[
\begin{aligned}
  &\|F^{8,1} +F^{9,1}\|_{L^\infty(\mathcal{H}^1)^\ast}
  \lesssim \mathfrak{E}(\eta)(\|\theta\|_{L^2W^{2,q_-}}
  +\|\pa_t(F^8 +F^9)\|_{L^\infty(\mathcal{H}^1)^\ast},
\end{aligned}
\]
  \[
\begin{aligned}
  &\|F^{8,1} +F^{9,1}(0)\|_{(\mathcal{H}^1)^\ast}
  \lesssim \mathfrak{E}(\eta)(0)(\|\theta_0\|_{W^{2,q_+}}
  +\|\pa_t(F^8 +F^9)(0)\|_{(\mathcal{H}^1)^\ast},
\end{aligned}
  \]
  \[
  \begin{aligned}
\|\pa_t(F^{8,1} +F^{9,1})\|_{(\mathcal{H}^1_T)^\ast}^2
  &\lesssim \|\pa_t^2(F^8 +F^9)\|_{(\mathcal{H}^1_T)^\ast}
  +\mathfrak{E}(\eta)(\|\pa_t\theta\|_{L^2W^{2,q_-}}
  +\|\theta\|_{L^2W^{2,q_-}}\\
  &\quad +\|\pa_tF^8\|_{L^2L^{q_-}}+\|\pa_tF^9\|_{L^2 W^{2-1/q_-,q_-}}) + \mathfrak{D}(\eta)\|\theta\|_{L^\infty W^{2,q_+}}.
  \end{aligned}
  \]
\end{lemma}
\begin{proof}
Since the proof of the first inequality is similar to Lemma \ref{lemma:est_force}, so we only give the proof  of the third inequality, which is much harder.
  From the notation in Section \ref{sec:notation}, we have that
  \begin{equation}\label{est:dual_1}
\left<\pa_t(F^{8,1}+F^{9,1}),\psi\right>_{(\mathcal{H}^1_T)^\ast}=\int_0^T\int_{\Om}\pa_tF^{8,1} \psi J + \int_0^T\int_{-\ell}^{\ell}\pa_tF^{9,1} \psi,
\end{equation}
for each $\psi\in \mathcal{H}^1_T$.
By the definition of  $F^{i,j}$ in \eqref{def:F_11}, \eqref{est:dual_1} is reduced to
\begin{equation}\label{est:dual_2}
\begin{aligned}
\left<\pa_t(F^{8,1} +F^{9,1}),\psi\right>_{(\mathcal{H}^1_T)^\ast}=\left<\pa_t^2(F^8 +F^9),\psi\right>_{(\mathcal{H}^1_T)^\ast}\\
+\int_0^T\int_{\Om}\pa_tG^8 \psi J + \int_0^T\int_{-\ell}^{\ell}\pa_tG^9 \psi,
\end{aligned}
\end{equation}
for each $\psi\in \mathcal{H}^1$, where $G^8$ and $G^9$ are defined as in Theorem \ref{thm:linear_low}. It is trivial that
\begin{equation}\label{est:dual_3}
\left<\pa_t^2(F^8 +F^9),\psi\right>_{(\mathcal{H}^1_T)^\ast}\le\|\pa_t^2(F^8 +F^9)\|_{(\mathcal{H}^1_T)^\ast}\|\psi\|_{\mathcal{H}^1_T}.
\end{equation}

By the H\"older inequality, Sobolev embedding inequality and usual trace theory, we have
\begin{align}\label{est:dual_4}
  \begin{aligned}
    \int_0^T\int_{\Om}\pa_tG^8 \psi J & \lesssim \int_0^T \|\psi\|_{L^{2/\varepsilon_-}}\|\nabla^2 \theta\|_{L^{q_-}} \|\pa_t\bar{\eta}\|_{W^{1,\infty}}^2 \\
    & \quad +\int_0^T \|\psi\|_{L^{4/(3\varepsilon_--\alpha)}}\|\nabla \theta\|_{L^{2/(1-\varepsilon_-)}} \|\pa_t\bar{\eta}\|_{W^{2,4/(2+\alpha-\varepsilon_-)}} \| \pa_t\bar{\eta}\|_{W^{1,\infty}}\\
    & \quad +\int_0^T \|\psi\|_{L^{2/(\varepsilon_+-2\alpha)}} \| \pa_t^2\eta\|_{W^{2,2/(1+2\alpha)}} \|\theta\|_{W^{1, 2/(1-\varepsilon_+)}} \\
    & \quad +\int_0^T \|\psi\|_{L^{2/(\varepsilon_--\alpha)}} \| \pa_t^2\eta\|_{W^{1,2/\alpha}} \|\nabla^2 \theta\|_{L^{q_-}} \\
    &\lesssim (\|\pa_t\eta\|_{L^\infty H^{3/2+(\varepsilon_--\alpha)/2}} \|\theta\|_{L^2W^{2,q_-}} + \|\theta\|_{L^\infty W^{2,q_+}}\|\pa_t^2\eta\|_{L^2H^{3/2-\alpha}})\|\psi\|_{\mathcal{H}^1_T}.
  \end{aligned}
\end{align}
We now turn to estimate the integration on the boundary. By the H\"older's inequality and Sobolev inequalities, we have
\begin{align}\label{est:dual_5}
  \begin{aligned}
    \int_0^T\int_{-\ell}^{\ell}\pa_tG^9 \psi & \lesssim \int_0^T\int_{-\ell}^\ell |\pa_t^2\mathcal{A}||\nabla \theta| |\psi| + |\pa_t\mathcal{A}||\nabla \pa_t\theta| |\psi| + |\pa_t\mathcal{A}||\nabla\theta| |\pa_t\pa_1\eta| |\psi| + |\theta||\pa_1\pa_t^2\eta| |\psi|\\
    &\quad + \int_0^T\int_{-\ell}^\ell |\nabla \pa_t\theta||\pa_1\pa_t\eta| |\psi| + |\nabla \theta||\pa_1\pa_t^2\eta| |\psi| + |\pa_t\eta||\pa_1\pa_t\theta||\psi| + |\theta| |\pa_1\pa_t\eta||\psi| \\
    &\lesssim \int_0^T (\|\pa_t\eta\|_{W^{1, 1/\alpha}} + \|\pa_t^2\eta\|_{W^{1, 1/\alpha}}) \|\nabla \theta\|_{L^{1/(1-\varepsilon_+)}(\Sigma)} \|\psi\|_{L^{1/(\varepsilon_+-\alpha)}(\Sigma)}\\
    &\quad + \int_0^T \|\pa_t\eta\|_{W^{1,\infty}} \|\pa_t\theta\|_{W^{1, 1/(1-\varepsilon_-)}(\Sigma)} \|\psi\|_{L^{1/(\varepsilon_-)}(\Sigma)}\\
    &\quad + \int_0^T \|\pa_t\eta\|_{W^{1,\infty}}^2 \|\theta\|_{W^{1, 1/(1-\varepsilon_-)}(\Sigma)} \|\psi\|_{L^{1/(\varepsilon_-)}(\Sigma)}\\
    &\lesssim (\|\pa_t\eta\|_{L^2H^{3/2-\alpha}}+\|\pa_t^2\eta\|_{L^2H^{3/2-\alpha}})\|\theta\|_{L^\infty W^{2, q_+}}\|\psi\|_{\mathcal{H}^1_T}\\
    &\quad + \|\pa_t\eta\|_{L^\infty H^{3/2+(\varepsilon_--\alpha)/2}}(\|\pa_t\theta\|_{L^2 W^{2, q_-}} +\|\theta\|_{L^2 W^{2, q_-}})\|\psi\|_{\mathcal{H}^1_T}.
  \end{aligned}
\end{align}

Then combining \eqref{est:dual_3} -- \eqref{est:dual_5} yields the result.
\end{proof}

\begin{theorem}\label{thm:higher order}
  Assume that the forcing terms $F^{i,j}$ satisfy the conditions in Lemma \ref{lemma:est_force} and \ref{lem:est_force_dual}, that $\mathfrak{K}(\eta)\le\delta$ is smaller than $\delta_0$ in Lemma \ref{lem:equivalence_norm} and Theorem 4.7 in \cite{GT2020}.
  Then there exists a unique strong solution $\theta$ to \eqref{eq:linear_heat} on temporal interval $[0, T]$. Moreover,
  $\pa_t^j\theta$ satisfies \eqref{eq:linear_heat2}in the strong sense with initial data $\pa_t^j\theta(0)$ for $j=0,1$ and $\pa_t^2\theta$ solves \eqref{eq:linear_heat2} in the weak sense for $j=2$. Moreover, there exists a constant $C_0=C_0(\sigma, \ell, \Om)$, such that the solution satisfies the estimate

  \begin{equation}\label{est:higher}
  \begin{aligned}
  &\sup_{0\le t\le T} \Big(\|\theta(t)\|_{W^{2, q_+}}^2 + \|\pa_t\theta(t)\|_{H^{1+\varepsilon_-/2}}^2+ \sum_{j=0}^2\|\pa_t^j\theta(t)\|_{H^0}^2\Big)\\
  &\quad + \int_0^T\Big(\|\theta(t)\|_{W^{2, q_+}}^2 + \|\pa_t\theta(t)\|_{W^{2, q_-}}^2 + \sum_{j=0}^2\|\pa_t^j\theta(t)\|_{H^1}^2\Big)\\
  &\le C_0\exp\{\mathfrak{E}(\eta)T\}\bigg(\|\theta_0\|_{W^{2, q_+}}^2 + \|\pa_t\theta(0)\|_{H^{3/2+(\varepsilon_+-\alpha)/2}}^2 + \|\pa_t^2\theta(0)\|_{H^0}^2+\|\pa_t(F^8+F^9)(0)\|_{(H^1)^\ast}^2\\
  &\quad+\|F^8(0)\|_{L^{q_+}}^2 +\|F^9\|(0)\|_{W^{1-1/q_{+}, q_+}}^2+(1+\mathfrak{E}(\eta))\Big(\|F^8\|_{L^\infty L^{q_+}}^2+\|F^9\|_{L^\infty W^{1-1/q_+,q_+}}^2\\
 &\quad+\|\pa_tF^8\|_{L^2L^{q_-}}^2+\|\pa_tF^9\|_{L^2W^{1-1/q_-,q_-}}^2+\|\pa_t^2(F^8+F^9)\|_{(\mathcal{H}^1_T)^{\ast}}^2 \Big)\bigg).
  \end{aligned}
  \end{equation}
\end{theorem}
\begin{proof}
\
\paragraph{\underline{Step 1 -- Construction of iteration}}
 Theorem \ref{thm:linear_low} guarantees  that $\pa_t^j\theta$ is a solution to \eqref{eq:linear_heat2} in the strong sense when $j=0$ and in the weak sense when $j=1$. For $j=1$, the assumption of Theorem \ref{thm:linear_low} are satisfied by Lemma \ref{lemma:est_force}, Lemma \ref{lem:est_force_dual} and compatibility conditions for initial data in Section \ref{sec:linear}. So based on the initial data convergence for Galerkin setting in Proposition \ref{prop:converge_dt2u0}, we want to use the result in Theorem \ref{thm:linear_low} to deduce that $\pa_t^j\theta$ is a solution to \eqref{eq:linear_heat} in the strong sense when $j=1$ and in the weak sense when $j=2$.

 Note that in Proposition \ref{prop:converge_dt2u0}, the estimate for $\pa_t(F^{8, 1} + F^{9,1})$ involves $\|\pa_t\theta\|_{L^2W^{2, q_-}}$, which is our aim to be obtained in this two-tier energy method. So we cannot directly use Proposition \ref{prop:converge_dt2u0}. In order to overcome this difficulty, we introduce an iteration:
 \begin{align}\label{eq:iteration_1}
   \left\{
   \begin{aligned}
     &\pa_t\vartheta^{n+1} - k \Delta_{\mathcal{A}} \vartheta^{n+1}=F^{8,1,n}\quad&\text{in}&\quad\Om,\\
  &k\mathcal{N} \cdot\nabla_{\mathcal{A}} \vartheta^{n+1} + \vartheta^{n+1} |\mathcal{N}|=F^{9,1,n}\quad&\text{on}&\quad\Sigma,\\
  &\vartheta^{n+1} = 0\quad&\text{on}&\quad\Sigma_s,
   \end{aligned}
  \right.
 \end{align}
 and
 \begin{align}\label{eq:iteration_2}
   \theta^{n+1} = \theta_0 + \int_0^t \vartheta^{n+1}(s)\,\mathrm{d}s.
 \end{align}
 Here we set $F^{8,1,n} = \pa_t F^8 + G^8(\theta^n)$ and $F^{9,1,n} = \pa_t F^9 + G^9(\theta^n)$. When $n=0$, we can choose $\theta^0 = \theta_0$ as the start point. By the Theorem \ref{thm:linear_low}, we know that $\vartheta^{n+1}$ solves \eqref{eq:iteration_1} with initial data $\vartheta^{n+1}(0) = \pa_t\theta(0)$ and $\pa_t\vartheta^{n+1}$ is a weak solution with initial data $\pa_t\vartheta^{n+1}(0) = \pa_t^2\theta(0)$ to
\begin{align}\label{eq:iteration_3}
   \left\{
   \begin{aligned}
     &\pa_t(\pa_t\vartheta^{n+1}) - k \Delta_{\mathcal{A}} \pa_t\vartheta^{n+1}=F^{8,2,n}\quad&\text{in}&\quad\Om,\\
  &k\mathcal{N} \cdot\nabla_{\mathcal{A}} \pa_t\vartheta^{n+1} + \pa_t\vartheta^{n+1} |\mathcal{N}|=F^{9,2,n}\quad&\text{on}&\quad\Sigma,\\
  &\pa_t\vartheta^{n+1} = 0\quad&\text{on}&\quad\Sigma_s,
   \end{aligned}
  \right.
 \end{align}
 where $F^{8,2,n} = \pa_t F^{8,1,n} + k\dive_{\pa_t\mathcal{A}} \nabla_{\mathcal{A}}\vartheta^{n+1} + k\dive_{\mathcal{A}} \nabla_{\pa_t\mathcal{A}}\vartheta^{n+1}$, and $F^{9,2,n} = \pa_t F^{9,1,n} - k\pa_t\mathcal{N} \cdot\nabla_{\mathcal{A}} \vartheta^{n+1} -\vartheta^{n+1} \pa_t|\mathcal{N}|$.

 We now derive the energy structure of \eqref{eq:iteration_3} by choosing the test function $\pa_t\vartheta^{n+1}$ such that
 \begin{align}
   \begin{aligned}
     \f12 \f{d}{dt} \int_\Om |\pa_t\vartheta^{n+1}|^2 J + k \int_\Om |\nabla_{\mathcal{A}} \pa_t \vartheta^{n+1}|^2J +\int_{-\ell}^\ell |\pa_t\vartheta^{n+1}|^2\sqrt{|\mathcal{N}|}\\
     = \left<F^{8,2,n} + F^{9,2,n}, \pa_t\vartheta^{n+1}\right>_\ast -\f12\int_\Om |\pa_t\vartheta^{n+1}|^2 \pa_tJ.
   \end{aligned}
 \end{align}
 By Lemma \ref{lemma:est_force} and Lemma \ref{lem:est_force_dual}, we have
 \begin{align}\label{est:iteration_1}
   \begin{aligned}
  &\|\pa_t\vartheta^{n+1}\|_{L^\infty H^0}^2 + \|\pa_t\vartheta^{n+1}\|_{L^2 H^1}^2+\|\pa_t\vartheta^{n+1}\|_{L^2 H^0(\Sigma)}^2\\
  &\le C\bigg( \|\pa_t^2\theta(0)\|_{H^0}^2  +\mathfrak{E}(\eta)(\|\pa_tF^8\|_{L^2L^{q_-}}^2+\|\pa_tF^9\|_{L^2W^{1-1/q_-,q_-}}^2 + \|\pa_t\theta^n\|_{L^2W^{2,q_-}}^2)\\
  &\quad + \mathfrak{D}(\eta)\|\theta^n\|_{L^\infty W^{2,q_-}}^2 +\|\pa_t^2(F^8+F^9)\|_{(\mathcal{H}^1_T)^{\ast}}^2 \bigg).
  \end{aligned}
 \end{align}
Furthermore, by the energy estimates and elliptic estimates to \eqref{eq:iteration_1}, we have
  \begin{equation}\label{est:high_d}
\begin{aligned}
  &\|\vartheta^{n+1}\|_{L^\infty H^0}^2+\|\vartheta^{n+1}\|_{L^2W^{2,q_-}}^2+ \|\vartheta^{n+1}\|_{L^2 H^1}^2+\|\vartheta^{n+1}\|_{L^2 H^0(\Sigma)}^2 \\
  &\le C\bigg( \|\pa_t\theta(0)\|_{H^0}^2 + \|\pa_t\vartheta\|_{L^2 H^0}^2 + \|\pa_tF^8\|_{L^2L^{q_-}}^2+\|\pa_tF^9\|_{L^2W^{1-1/q_-,q_-}}^2 \bigg).
\end{aligned}
  \end{equation}

Now we have to show that the sequences $\{\vartheta^n\}$ and $\{\theta^n\}$ are Cauchy sequences. Set
\[
\mathfrak{M}(\vartheta^{n+1}):=\sum_{j=0}^1\|\pa_t^j\vartheta^{n+1}\|_{L^\infty H^0}^2+\|\vartheta^{n+1}\|_{L^2W^{2,q_-}}^2+\sum_{j=0}^1\big(\|\pa_t^j\vartheta^{n+1}\|_{L^2 H^1}^2 +\|\pa_t^j\vartheta^{n+1}\|_{L^2 H^0(\Sigma)}^2),
\]
\[
\mathfrak{N}(\theta^n) :=\|\theta^n\|_{L^\infty W^{2,q_-}}^2+ \|\theta^n\|_{L^2 W^{2,q_-}}^2+ \|\theta^n\|_{L^2H^1}^2+\|\pa_t\theta^n\|_{L^2W^{2,q_-}}^2,
\]
and
\[
\begin{aligned}
  \mathcal{Z}&:=\sum_{j=1}^2\|\pa_t^j\theta(0)\|_{H^0}^2+\|\pa_t^2(F^1-F^4-F^5)\|_{(\mathcal{H}^1_T)^\ast}^2+\mathfrak{D}(\eta)(\|F^8\|_{L^\infty L^{q_+}}^2+\|F^9\|_{L^\infty W^{2-1/q_+,q_+}}^2)\\
 &\quad+(1+\mathfrak{E}(\eta))(\|\pa_tF^8\|_{L^2L^{q_-}}^2+\|\pa_tF^9\|_{L^2 W^{2-1/q_-,q_-}}^2)+\|\pa_t(F^8+F^9)(0)\|_{(H^1)^\ast}^2.
\end{aligned}
\]
In the following estimates \eqref{est:cauchy_1}-- \eqref{est:cauchy_3}, $C$ is a constant independent of $n$, which is allowed to be changed from line to line.
By \eqref{est:iteration_1} and \eqref{est:high_d}, we have
\begin{align}\label{est:cauchy_1}
  \mathfrak{M}(\vartheta^{n+1}) \le C \mathcal{Z} + C\mathfrak{K}(\eta)\mathfrak{N}(\theta^n).
\end{align}
From \eqref{eq:iteration_2}, we get
\begin{align}\label{est:cauchy_2}
  \mathfrak{N}(\theta^{n+1}) \le CT (\|\theta_0\|_{W^{2,q_+}}^2 + \|\theta_0\|_{H^1}^2) + (1+T) \mathfrak{M}(\vartheta^{n+1}).
\end{align}
Then for a fixed $T>0$, due to the smallness of $\mathfrak{\eta}$, we have
\begin{align}\label{est:cauchy_3}
  \mathfrak{M}(\vartheta^{n+1}) + \mathfrak{N}(\theta^{n+1}) \le C \mathcal{Z}
\end{align}

Now we show that the systems \eqref{eq:iteration_1} coupled with \eqref{eq:iteration_2} is a contraction mapping : $\theta^n \to \vartheta^{n+1}$ under the topology $\mathfrak{M}$ and $\mathfrak{N}$. Suppose that the mapping $\theta^{n+1} \to \vartheta^{n+2}$  satisfies the systems \eqref{eq:iteration_1} and \eqref{eq:iteration_2}, then the mapping $\theta^{n+1}-\theta^n \to \vartheta^{n+2}-\vartheta^{n+1}$ also satisfies \eqref{eq:iteration_1} and \eqref{eq:iteration_2} with the forces $F^8= F^9= 0$ and zero initial data. Therefore the results in \eqref{est:cauchy_1} and \eqref{est:cauchy_2} imply
  \begin{align}\label{est:iterated_contract1}
\begin{aligned}
  \mathfrak{M}(\vartheta^{n+2}-\vartheta^{n+1}) \lesssim \mathfrak{K}(\eta) \mathfrak{N}(\theta^{n+1}-\theta^n).
\end{aligned}
  \end{align}
  Then we use \eqref{eq:iteration_2} to deduce that
  \begin{align}\label{est:contract_u}
\|\theta^{n+1}-\theta^n\|_{L^2W^{2,q_-}}^2 \lesssim T\|\vartheta^{n+1}-\vartheta^n\|_{L^2W^{2,q_-}}^2, \quad \|\theta^{n+1}-\theta^n\|_{L^\infty W^{2,q_-}}^2 \lesssim \|\vartheta^{n+1}-\vartheta^n\|_{L^2W^{2,q_-}}^2.
  \end{align}
  We still use \eqref{eq:iteration_2} to see that
  \begin{align}
\pa_t\theta^{n+1} = \vartheta^{n+1},
  \end{align}
  so that
  \begin{align}\label{est:contract_dtp}
  \begin{aligned}
\|\pa_t\theta^{n+1} - \pa_t\theta^n\|_{L^2W^{2,q_-}}^2 = \|\vartheta^{n+1}-\vartheta^n\|_{L^2W^{2,q_-}}^2.
  \end{aligned}
  \end{align}
  By combining all the estimates \eqref{est:contract_u}--\eqref{est:contract_dtp}, we have that
  \begin{align}\label{est:iterated_contract2}
\begin{aligned}
  \mathfrak{N}(\theta^{n+1}-\theta^n) \lesssim (T+1)\mathfrak{K}(\eta)\mathfrak{M}(\vartheta^{n+1}-\vartheta^n).
\end{aligned}
  \end{align}
  Hence for the restricted smallness of $\mathfrak{K}(\eta) < \gamma$, the estimates \eqref{est:iterated_contract1} and \eqref{est:iterated_contract2} show that the operator for \eqref{eq:iteration_1} and \eqref{eq:iteration_2} is strictly contracted. So that \eqref{eq:linear_heat2} has a unique fixed point when $j=1$. That means $\vartheta^n$ converges to $\vartheta$ in the topology of $\mathfrak{M}$, and $\theta^n$ converges to $\tilde{\theta}$ in the topology of $\mathfrak{N}$. Moreover, $\vartheta$ and $\tilde{\theta}$ are strong solutions to
  \begin{equation}\label{eq:limit1}
\left\{
\begin{aligned}
  &\pa_t\vartheta-k\Delta_{\mathcal{A}}\vartheta=F^{8,1} \quad&\text{in}&\quad\Om,\\
  &k\mathcal{N}\cdot\nabla_{\mathcal{A}}\vartheta + \vartheta |\mathcal{N}| =F^{9,1} \quad&\text{on}&\quad\Sigma,\\
  &\vartheta  =0 \quad&\text{on}&\quad\Sigma_s,
\end{aligned}
\right.
\end{equation}
  and
  \begin{align}\label{equ:limit2}
\tilde{\theta} = \theta_0 + \int_0^t \vartheta(s)\,\mathrm{d}s,
  \end{align}
  with the initial data $\vartheta(0) = \pa_t\theta(0)$ and $\tilde{\theta}(0) = \theta_0$.

  \paragraph{\underline{Step 2 -- Consistency of Solution}}
  In the following, we show that $\tilde{\theta} = \theta$, which guarantees that \eqref{eq:limit1} is exactly equal to the one time derivative of \eqref{eq:linear_heat2} with $j=1$.
  We first use \eqref{equ:limit2} and the facts that $\vartheta = \theta_0 =0$ on $\Sigma_s$ to obtain that $\tilde{\theta} =0$ on $\Sigma_s$.

  We use \eqref{equ:limit2} to get that $\vartheta = \pa_t\tilde{\theta}$, which might be plugged into \eqref{eq:limit1} so that  $\pa_t\tilde{\theta}$ is a strong solution to
  \begin{align}\label{eq:derivative_linear}
\left\{
  \begin{aligned}
&\pa_t(\pa_t\tilde{\theta}-k\Delta_{\mathcal{A}}\tilde{\theta})=\pa_tF^8 \quad &\text{in} \quad &\Om,\\
&\pa_t(k\nabla_{\mathcal{A}} \tilde{\theta}\cdot\mathcal{N} + \tilde{\theta} |\mathcal{N}|)=\pa_tF^9 \quad &\text{on} \quad &\Sigma,\\
&\pa_t\tilde{\theta}=0 \quad &\text{on} \quad &\Sigma_s.
  \end{aligned}
  \right.
  \end{align}
   Since the first equation in \eqref{eq:derivative_linear} holds in $L^{q_-}(\Om)$, we have
  \begin{align}
\frac{d}{dt}\|\pa_t\tilde{\theta}-k\Delta_{\mathcal{A}}\tilde{\theta}-F^8\|_{L^{q_-}}^{q_-} =0.
  \end{align}
  After integrating in time from $0$ to $t <T$, we have
  \begin{align}\label{eq:integral_1}
\|\pa_t\tilde{\theta}-k\Delta_{\mathcal{A}}\tilde{\theta}-F^8\|_{L^{q_-}} = \|-F^8(0)+\pa_t\tilde{\theta}(0) -k\Delta_{\mathcal{A}_0}\tilde{\theta}_0\|_{L^{q_-}}.
  \end{align}
  We can still use the similar computation to deal with the boundary terms:
  \begin{align}
\begin{aligned}
  &\| \nabla_{\mathcal{A}} \tilde{\theta}\cdot\mathcal{N} + \tilde{\theta} |\mathcal{N}|-F^9 \|_{L^{q_-}}= \|\nabla_{\mathcal{A}_0} \tilde{\theta}_0\cdot\mathcal{N} (0)+ \tilde{\theta}_0 |\mathcal{N}_0|)-F^9(0)\|_{L^{q_-}},
\end{aligned}
  \end{align}
  and
  \begin{align}\label{eq:integral_2}
\|\tilde{\theta}\|_{L^2(\Sigma_s)}=\|\tilde{\theta}_0\|_{L^2(\Sigma_s)}.
  \end{align}

  Thanks to the initial data of $(\theta_0, \pa_t\theta(0))$ satisfying \eqref{eq:linear_heat} at $t=0$, we directly see that the initial terms in \eqref{eq:integral_1} -- \eqref{eq:integral_2} all vanish. Then from the regularity of $\tilde{\theta}$, \eqref{eq:integral_1} -- \eqref{eq:integral_2} implies that $\tilde{\theta}$ solves \eqref{eq:linear_heat} at least in the sense of weak solutions, since we have shown that $\theta$ is also a solution to \eqref{eq:linear_heat}. Then the uniqueness of weak solutions to \eqref{eq:linear_heat} in Proposition \ref{prop:unique} guarantees that $\tilde{\theta} = \theta$. Then by the Theorem \ref{thm:linear_low} for strong solutions to \eqref{eq:linear_heat}, we can see that $\tilde{\theta} =\theta$ is the strong solution to \eqref{eq:linear_heat}. Moreover,
  $
\mathfrak{M}(\theta)\lesssim \mathcal{Z}.
  $

\paragraph{\underline{Step 3 -- Estimates for energy and dissipation.}}

 We combine the estimates following from Theorem \ref{thm:linear_low}, the boundedness of $\mathfrak{M}(\theta)$ and Lemma \ref{lemma:est_force}, along with $\mathfrak{K}(\eta)$ sufficiently small to obtain that
  \begin{equation}\label{est:high_d1}
  \begin{aligned}
  &\sum_{j=0}^2\|\pa_t^j\theta\|_{L^\infty H^0}^2+\|\theta\|_{L^2W^{2,q_+}}^2+\|\theta\|_{L^\infty W^{2,q_+}}^2+\|\pa_t\theta\|_{L^2W^{2,q_-}}^2+\sum_{j=0}^2\left(\|\pa_t^j\theta\|_{L^2 H^1}^2+\|\pa_t^j\theta\|_{L^2 H^0(\Sigma)}^2\right)\\
  &\lesssim \exp\{\mathfrak{E}(\eta)T\}\bigg(\mathcal{E}(0)+\|F^8(0)\|_{L^{q_-}}^2+\|F^9(0)\|_{W^{1-1/q_-, q_-}}^2 +\|\pa_t(F^8+F^9)(0)\|_{(H^1)^\ast}^2\\
  &\quad+(1+\mathfrak{E}(\eta))(\|\pa_tF^8\|_{L^2L^{q_-}}^2+\|\pa_tF^9\|_{L^2W^{1-1/q_-,q_-}}^2+\|F^8\|_{L^\infty L^{q_+}}^2+\|F^9\|_{L^\infty W^{1-1/q_+,q_+}}^2\\
  &\quad +\|\pa_t^2(F^8+F^9)\|_{(\mathcal{H}^1_T)^{\ast}}^2)\bigg).
  \end{aligned}
  \end{equation}
  We now use Sobolev extension and restriction theory to estimate $
  \f{d}{dt}\frac12\|\pa_t\theta\|_{H^{1+\varepsilon_-/2}}^2\le \frac12\|\pa_t\theta\|_{1+\varepsilon_-}^2+\frac12\|\pa_t^2\theta\|_1^2$
  so that
  \begin{equation}\label{est:continu_dtu}
\|\pa_t\theta\|_{L^\infty H^{1+\varepsilon_-/2}}^2\lesssim \|\pa_t\theta(0)\|_{ H^{1+\varepsilon_-/2}}^2+\|\pa_t\theta\|_{L^2H^{1+\varepsilon_-}}^2+\|\pa_t^2\theta\|_{L^2H^1}^2.
  \end{equation}
 Finally, we combine \eqref{est:high_d1} and  \eqref{est:continu_dtu} conclusion \eqref{est:higher}.
  \end{proof}

\section{Local Well-Posedness for Nonlinear $\varepsilon$--Regularized Boussinesq System}

\subsection{Solving the linear coupled system \eqref{eq:linear_heat} and \eqref{eq:modified_linear}}
We have established the well-posedness for \eqref{eq:linear_heat}. Then we can put the $\theta$ into \eqref{eq:modified_linear} to solve the coupled linear system \eqref{eq:linear_heat} and \eqref{eq:modified_linear}.
With $\mathcal{E}$ in \eqref{energy} and $\mathcal{D}$ in \eqref{dissipation}, we define the quantities as follows.
\[
\begin{aligned}
\mathscr{D}(u,p,\xi,\theta):&=\|\mathcal{D}(u, p, \xi,\theta)\|_{L^2_t([0,T])}^2+\varepsilon\|\pa_t^2\xi\|_{L^2H^1}^2
+\varepsilon^2\|\xi\|_{L^2W^{3-1/q_+,q_+}}^2+\varepsilon^2\|\pa_t\xi\|_{L^2W^{3-1/q_-,q_-}}^2,\\
\mathscr{E}(u,p,\xi,\theta):&=\sup_{t\in[0,T]}\mathcal{E}(u, p, \xi,\theta)+\varepsilon\|\pa_t^2\xi\|_{L^\infty H^{3/2}}^2+\varepsilon^2\|\pa_t\xi\|_{L^\infty W^{3-1/q_+,q_+}}^2,\\
\mathscr{K}(u,p,\xi,\theta):&=\mathscr{E}(u,p,\xi,\theta)+\mathscr{D}(u,p,\xi,\theta).
\end{aligned}
\]

Now for convenience, we introduce two new spaces
\begin{equation}\label{def:xy}
\begin{aligned}
\mathcal{X}=\Big\{(u,p,\eta,\theta)|\mathscr{E}(u,p,\eta,\theta)<\infty\Big\},\quad \|(u,p,\eta,\theta)\|_{\mathcal{X}}=\left[\mathscr{E}(u,p,\eta,\theta)\right]^{1/2},\\ \mathcal{Y}=\Big\{(u,p,\eta,\theta)|\mathscr{D}(u,p,\eta,\theta)<\infty\Big\}, \quad\|(u,p,\eta,\theta)\|_{\mathcal{Y}}=\left[\mathscr{D}(u,p,\eta,\theta)\right]^{1/2}.
\end{aligned}
\end{equation}

By some minor modifications to \cite{GTWZ2023}, we can directly establish the solutions to \eqref{eq:modified_linear}, which is stated in the following theorem.
\begin{theorem}\label{thm:linear}
  Assume that the forcing terms $F^{i,j}$ satisfy the conditions in Appendix \ref{sec:initial} and that $\mathfrak{K}(\eta)\le\delta$ is smaller than $\delta_0$, which is the same as in Theorem 4.7 in \cite{GTWZ2023}. Suppose that the initial data $u_0$, $p_0$, $\theta_0$ and $\eta_0$ satisfy $\mathcal{E}(0)<\infty$ and compatibility conditions in Appendix \ref{sec:initial}.
  Then for each $0<\varepsilon\le1$ there exist $T_\varepsilon>0$ such that for $0<T\le T_\varepsilon$, then there exists a unique strong solution $(u,p,\xi, \theta)$ to \eqref{eq:modified_linear} on $[0, T]$ such that
  $(u,p,\xi)\in\mathcal{X}\cap\mathcal{Y}$.
  The triple $(D_t^ju, \pa_t^jp, \pa_t^j\xi, \pa_t^j\theta)$ satisfies \eqref{eq:modified_linear}
  in the strong sense with initial data $(D_t^ju(0), \pa_t^jp(0), \pa_t^j\xi(0), \pa_t^j\theta(0))$ for $j=0,1$ and $(D_t^2u, \pa_t^2\xi, \pa_t^2\theta)$ solves \eqref{eq:modified_linear} in the pressureless weak sense for $j=2$. Moreover, $(u, p, \xi, \theta)$ satisfy the estimate
  \begin{equation}\label{est:coupled_linear}
  \begin{aligned}
  \mathscr{K}(u,p,\xi, \theta)&\le C_0\left[1+\frac T\varepsilon\right]\exp\{\mathfrak{E}(\eta)T\}\bigg(\mathcal{E}(0)+\|\pa_t(F^1-F^4-F^5)(0)\|_{(H^1)^\ast}+\|F^3(0)\|_{L^2}\\
  &\quad +\|\pa_t(F^8+F^9)(0)\|_{(H^1)^\ast}^2+ \|F^1(0)\|_{L^{q_+}}^2 + \|F^4(0)\|_{W^{1-1/q_+, q_+}}^2 + \|F^5(0)\|_{W^{1-1/q_+, q_+}}^2\\
  &\quad + \|F^8(0)\|_{L^{q_+}}^2 + \|F^9(0)\|_{W^{1-1/q_+, q_+}}^2+\frac T{\varepsilon^2}\mathfrak{E}(\eta_0)(\mathcal{E}(0)+\|\pa_1F^3(0)\|_{W^{1-1/q_-,q_-}}^2 \\
  &\quad +\|F^4(0)\|_{W^{1-1/q_-,q_-}}^2)+(1+\mathfrak{E}(\eta))(\|F^1\|_{L^\infty L^{q_+}}^2+\|F^4\|_{L^\infty W^{1-1/q_+,q_+}}^2\\
  &\quad+\|F^5\|_{L^\infty W^{1-1/q_+,q_+}}^2+\|\pa_1F^3\|_{L^\infty W^{1-1/q_+,q_+}}^2+\|\pa_tF^1\|_{L^2L^{q_-}}^2+\|\pa_tF^4\|_{L^2W^{1-1/q_-,q_-}}^2\\
 &\quad+\|\pa_tF^5\|_{L^2W^{1-1/q_-,q_-}}^2+\|\pa_1\pa_tF^3\|_{L^2W^{1-1/q_-,q_-}}^2+\|\pa_t^2(F^1-F^4-F^5)\|_{(\mathcal{H}^1_T)^{\ast}}^2\\
 &\quad +\|F^8\|_{L^\infty L^{q_+}}^2+\|F^9\|_{L^\infty W^{1-1/q_+,q_+}}^2+\|\pa_tF^8\|_{L^2L^{q_-}}^2+\|\pa_tF^9\|_{L^2W^{1-1/q_-,q_-}}^2\\
 &\quad +\|\pa_t^2(F^8+F^9)\|_{(\mathcal{H}^1_T)^{\ast}}^2)+\sum_{j=0}^2\|[F^{7,j}]_\ell\|_{L^2}^2+\frac T\varepsilon\sum_{j=0}^2\|F^{3,j}\|_{L^\infty H^0}^2\\
 &\quad+\sum_{j=0}^2\|F^{3,j}\|_{L^2 H^{1/2-\alpha}}^2+\sum_{j=0}^1(\|\pa_t^jF^1\|_{L^\infty H^0}^2+\|[F^{7,j}]_\ell\|_{L^\infty_t}^2)\bigg),
  \end{aligned}
  \end{equation}
  where $C_0=C_0(g,\sigma, \ell, \Om)$ is a positive constant independent of $\varepsilon$.
\end{theorem}

\subsection{Preliminaries for the nonlinear system}\label{sec:nonlinear}

We now give some preparations for the local well-posedness of the nonlinear $\varepsilon-$ regularized system \eqref{eq:epsilon}.

From the formulation of \eqref{eq:geometric}, we have the nonlinear interaction terms as the forcing given by Appendix \ref{sec:dive_forcing}.
In order to close the energy estimates, we need to control the forcing terms in the sense of Theorem \ref{thm:linear_low} and \ref{thm:higher order}. We first estimate the forcing terms in the bulk.
\begin{proposition}\label{prop:force_bulk}
  It holds that
  \begin{align}
  \begin{aligned}
  \|F^8\|_{L^\infty L^2}+\|\pa_tF^8\|_{L^\infty L^2}+\|\pa_t(F^8+F^9)\|_{L^\infty(\mathcal{H}^1)^\ast}\lesssim \mathfrak{E}+\mathfrak{E}^{3/2},\\
  \|\pa_tF^8\|_{L^2 L^{q_-}}+\|\pa_t^2(F^8+F^9)\|_{L^2(\mathcal{H}^1)^\ast}\lesssim \mathfrak{E}^{1/2}\mathfrak{D}^{1/2}.
  \end{aligned}
  \end{align}
\end{proposition}
\begin{proof}
  It is directly to combine the H\"older inequality, Sobolev embedding estimates and trace theory, and the bounds of $K$ and $\mathcal{A}$ to estimate
  \begin{align}
  \begin{aligned}
  \|F^8\|_{L^\infty L^2}&\lesssim (\|\pa_t\bar{\eta}\|_{L^\infty_{t,x}}+\|u\|_{L^\infty_{t,x}})\|\nabla \theta\|_{L^\infty L^2}\\
  &\lesssim (\|\pa_t\eta\|_{L^\infty H^{3/2+(\varepsilon_--\alpha)/2}}+\|u\|_{L^\infty W^{2,q_+}})\|\theta\|_{L^\infty W^{2,q_+}}.
  \end{aligned}
  \end{align}

  Then we consider $
\pa_tF^1=\pa_t^2\bar{\eta}W\p_2\theta+\pa_t\bar{\eta}W\pa_tK\p_2\theta +\pa_t\bar{\eta}W\p_2\pa_t\theta
-\pa_tu\cdot\nabla_{\mathcal{A}}\theta - u\cdot\nabla_{\pa_t\mathcal{A}}\theta -u\cdot\nabla_{\mathcal{A}}\pa_t \theta$.
 The similar argument as $F^8$ enables us to obtain
  \begin{align}
  \begin{aligned}
\|\pa_tF^8\|_{L^\infty H^0} & \lesssim (\|\pa_t^2\bar{\eta}\|_{L^\infty_{x,t}}+\|\pa_t\bar{\eta}\|_{L^\infty_{x,t}}\|\nabla\pa_t\bar{\eta}\|_{L^\infty_{x,t}}+\|\nabla\pa_t\bar{\eta}\|_{L^\infty_{x,t}}\|u\|_{L^\infty_{x,t}}+\|\pa_tu\|_{L^\infty_{x,t}})\|\nabla \theta\|_{L^\infty L^2}\\
& \quad +(\|\pa_t\bar{\eta}\|_{L^\infty_{x,t}}+\|u\|_{L^\infty_{x,t}})\|\nabla\pa_t\theta\|_{L^2}\\
& \lesssim (\|\pa_t^2\eta\|_{L^\infty H^1}+\|\pa_t\eta\|_{L^\infty H^{3/2+(\varepsilon_--\alpha)/2}}^2+\|\pa_tu\|_{L^\infty H^{1+\varepsilon_-/2}})\|\theta\|_{L^\infty W^{2,q_+}}\\
& \quad +\|\pa_t\eta\|_{L^\infty H^{3/2+(\varepsilon_--\alpha)/2}}(\|\theta\|_{L^\infty W^{2,q_+}}^2+\|\pa_t\theta\|_{L^\infty H^{1+\varepsilon_-/2}}).
  \end{aligned}
  \end{align}
  Then it is trivial that $
\|\pa_t(F^8+F^9)\|_{L^\infty(\mathcal{H}^1)^\ast}\le \|\pa_tF^8\|_{L^\infty L^2}$,
  since $F^9=0$.

  In addition, we bound the $\|\pa_tF^8\|_{L^2 L^{q_-}}$ by
  \begin{equation}
\begin{aligned}
  \|\pa_tF^8\|_{L^2 L^{q_-}}&\lesssim \|\pa_t^2\bar{\eta}\|_{L^2L^{2/\alpha}}\|\nabla \theta\|_{L^\infty L^{1/(1-(\varepsilon_-+\alpha)/2)}}+\|\pa_t\bar{\eta}\|_{L^\infty_{x,t}}(\|\pa_t\bar{\eta}\|_{L^\infty H^1}\|\theta\|_{L^2H^1}+\|\pa_t\theta\|_{L^2H^1})\\
  &\quad+(\|\pa_tu\|_{L^\infty_{x,t}} + \|\nabla\pa_t\bar{\eta}\|_{L^\infty_{x,t}}\|u\|_{L^\infty_{x,t}})\|\theta\|_{L^2H^1} + \|u\|_{L^\infty_{x,t}}\|\pa_t\theta\|_{L^2H^1}\\
  &\lesssim(\|\pa_t\theta\|_{L^2H^1}+\|\pa_t\eta\|_{L^\infty H^{3/2+(\varepsilon_--\alpha)/2}}\|\theta\|_{L^2H^1})\|u\|_{L^\infty W^{2,q_+}}\\
  &\quad+\|\pa_t\eta\|_{L^\infty H^{3/2+(\varepsilon_--\alpha)/2}}(\|\theta\|_{L^2H^1}+\|\pa_t\theta\|_{L^2H^1}) + \|\pa_tu\|_{L^\infty H^{1+\varepsilon_-/2}}\|\theta\|_{L^2H^1}\\
  &\quad+\|\pa_t^2\eta\|_{L^2H^{3/2-\alpha}}\|\theta\|_{L^\infty W^{2,q_+}}.
\end{aligned}
  \end{equation}

  For $\pa_t^2F^8$, we consider
  \begin{align}
  \begin{aligned}
\left<\pa_t^2F^8,\psi\right>_\ast&=\int_\Omega (\pa_t^3\bar{\eta}WK\p_2\theta+2\pa_t^2\bar{\eta}W\pa_tK\p_2\theta +2\pa_t^2\bar{\eta}WK\p_2\pa_t\theta)\cdot\psi\\
&\quad+\int_\Omega (\pa_t\bar{\eta}\pa_t^2K W\p_2\theta+2\pa_t\bar{\eta}W\pa_tK\p_2\pa_t\theta +\pa_t\bar{\eta}WK\p_2\pa_t^2\theta)\cdot\psi\\
&\quad-\int_\Omega(\pa_t^2u\cdot\nabla_{\mathcal{A}}\theta+2\pa_tu\cdot\nabla_{\pa_t\mathcal{A}}\theta+2\pa_tu\cdot\nabla_{\mathcal{A}}\pa_t \theta)\cdot\psi J\\
&\quad-\int_\Omega(u\cdot\nabla_{\pa_t^2\mathcal{A}}\theta+2u\cdot\nabla_{\pa_t\mathcal{A}}\pa_t\theta+u\cdot\nabla_{\mathcal{A}}\pa_t^2 \theta)\cdot\psi J
  \end{aligned}
  \end{align}
  for any $\psi\in H^1(\Om)$. We first use H\"older inequality and the finite volume of $\Omega$ to see that
  \begin{align}
  \begin{aligned}
\left<\pa_t^2F^8,\psi\right>_{(\mathcal{H}^1)^\ast}&\lesssim \|\nabla \theta\|_{L^{2/(1-\varepsilon_+)}}\Big(\|\pa_t^3\bar{\eta}\|_{L^{2/\varepsilon_+}}+\|\pa_t^2\bar{\eta}\|_{L^\infty}\|\pa_t\bar{\eta}\|_{W^{1,\infty}}+\|\pa_t^2\bar{\eta}\|_{W^{1,2/\alpha}}\|\pa_t\bar{\eta}\|_{L^{\infty}}\\
&\quad+\|\pa_t^2u\|_{L^{2/\varepsilon_+}}+\|\pa_tu\|_{L^{2/\varepsilon_+}}\|\pa_t\bar{\eta}\|_{W^{1,\infty}}+\|\pa_t^2\bar{\eta}\|_{W^{1,2/\alpha}}\|u\|_{L^\infty}\Big)\|\psi\|_{L^2}\\
&\quad+\|\nabla\pa_t\theta\|_{H^0}(\|\pa_t^2\bar{\eta}\|_{L^\infty}+\|\pa_t\bar{\eta}\|_{W^{1,\infty}}\|\pa_t\bar{\eta}\|_{L^\infty}+\|\pa_tu\|_{L^\infty}\\
&\quad+\|u\|_{L^\infty}\|\pa_t\bar{\eta}\|_{W^{1,\infty}})\|\psi\|_{L^2}+\|\nabla\pa_t^2\theta\|_{H^0}(\|u\|_{L^\infty}+\|\pa_t\bar{\eta}\|_{L^\infty})\|\psi\|_{L^2}.
  \end{aligned}
  \end{align}
  Then we employ Sobolev embedding estimates and trace theory with $\mathfrak{E}$ small to deduce
  \begin{align}
  \begin{aligned}
\left<\pa_t^2F^8,\psi\right>_\ast&\lesssim \|\theta\|_{L^\infty W^{2,q_+}}(\|\pa_t^3\eta\|_{L^2H^{1/2}}+\|\pa_t^2\eta\|_{L^2H^{3/2-\alpha}}+\|\pa_tu\|_{L^2H^1}+\|\pa_t^2u\|_{L^2H^1}\\
&\quad + \|u\|_{L^2W^{2,q_-}})\|\psi\|_{L^2H^1} + \|\pa_t\theta\|_{L^2H^1}(\|\pa_t^2\eta\|_{L^\infty H^1}+ \|\pa_tu\|_{L^\infty H^{1+\varepsilon_-/2}} \\
&\quad+ \|\pa_t\eta\|_{L^\infty H^{3/2+(\varepsilon_--\alpha)/2}}(\|\pa_t\eta\|_{L^\infty H^{3/2+(\varepsilon_--\alpha)/2}}+\|u\|_{L^\infty W^{2,q_+}}))\|\psi\|_{L^2H^1} \\
&\quad + \|\pa_t^2\theta\|_{L^2H^1}(\|u\|_{L^\infty W^{2,q_+}}+\|\pa_t\eta\|_{L^\infty H^{3/2+(\varepsilon_--\alpha)/2}})\|\psi\|_{L^2H^1}.
  \end{aligned}
  \end{align}
  Hence the completion of proof by combining all the above inequalities and the definition of $\mathfrak{E}$ and $\mathfrak{D}$.
\end{proof}

\subsection{Initial Data for the Nonlinear $\varepsilon$-Regularized System}\label{sec:initial_nonlinear1}

We assume that the initial data
\begin{align*}
\begin{aligned}
(u_0, p_0, \theta_0, \eta_0, \pa_tu(0), \pa_tp(0), \pa_t\theta(0), \pa_t\eta(0), \pa_t^2u(0), \pa_t^2\theta(0), \pa_t^2\eta(0), \varepsilon \pa_t\eta(0), \sqrt{\varepsilon}\pa_t^2\eta(0), \varepsilon\pa_t^2\eta(0))
\end{aligned}
\end{align*}
for \eqref{eq:geometric} are in the space $X^\varepsilon$ defined via
 \begin{align}\label{initial_ep_nonlinear}
 \begin{aligned}
 X^\varepsilon:&=W^{2,q_+}(\Om)\times W^{1,q_+}(\Om)\times W^{2,q_+}(\Om)\times W^{3-1/q_+,q_+}(\Sigma)\times H^{1+\varepsilon_-/2}(\Om)\times H^0(\Om)\\
 &\times H^{1+\varepsilon_-/2}(\Om)\times H^{3/2+(\varepsilon_--\alpha)/2}(\Sigma)\times H^0(\Om)\times H^0(\Sigma)\times H^1(\Sigma)\times W^{3-1/q_+,q_+}(\Sigma)\\
 &\times H^{3/2-\alpha}(\Sigma)\times W^{3-1/q_+,q_+}(\Sigma)
 \end{aligned}
 \end{align}
 and satisfy the compatibility conditions \eqref{compat_C2}
as well as zero average conditions \eqref{cond:zero},
 where $X^\varepsilon$ is a Banach space, with the square norm
\begin{align}
\begin{aligned}
&\|(u_0, p_0, \theta_0, \eta_0, \pa_tu(0), \pa_tp(0), \pa_t\theta(0), \pa_t\eta(0), \pa_t^2u(0), \pa_t^2\theta(0), \pa_t^2\eta(0), \pa_t^3\eta(0))\|_{X^\varepsilon}^2\\
&=\mathcal{E}(0)+\varepsilon^2\|\pa_t\eta(0)\|_{ W^{3-1/q_+,q_+}}^2+\varepsilon^2\|\pa_t^2\eta(0)\|_{ W^{3-1/q_+,q_+}}^2+\varepsilon\|\pa_t^2\eta(0)\|_{ H^{3/2-\alpha}}^2,
\end{aligned}
\end{align}
where $\mathcal{E}(0)$ is defined via \eqref{energy},

From the Theorem \ref{thm:initial}, we have known that our initial data really exist, provided that $\|\pa_t^2u(0)\|_{H^0}^2+\|\pa_t^2\eta(0)\|_{H^1}^2+\varepsilon^2\|\pa_t\eta(0)\|_{ W^{3-1/q_+,q_+}}^2+\|\pa_t^2\eta(0)\|_{W^{2-1/q_+,q_+}}^2+\varepsilon^2\|\pa_t^2\eta(0)\|_{ W^{3-1/q_+,q_+}}^2+\varepsilon\|\pa_t^2\eta(0)\|_{ H^{3/2-\alpha}}^2+\|\pa_t^2\theta(0)\|_{H^0}^2$ is sufficiently small.

\subsection{Existence of Solutions to nonlinear $\varepsilon$--regularized system}

We now consider the local well-posedness of the nonlinear problem prescribed by \eqref{eq:epsilon}. Our strategy is to
work in a complete metric space, that requires high regularity bounds,
but that is endowed with a metric only involves low-regularity. First we will find a complete metric space, endowed with a weak choice of a metric, compatible with the linear estimates in Theorem \ref{thm:higher order}. Then we will prove that the fixed point on this metric space gives a solution to \eqref{eq:epsilon}.

\begin{definition}\label{def:S}
  Suppose that $T>0$. For $\delta\in(0,\infty)$ we define the space
  \begin{equation}
  \begin{aligned}
S(T, \delta)
&=\Big\{(u, p, \theta, \eta):\Om\to\mathbb{R}^2\times\mathbb{R}\times\mathbb{R}\Big|(u,p, \theta, \eta))\in\mathcal{X}\cap\mathcal{Y},\ \text{with}\ \mathfrak{K}(u, p, \theta, \eta)^{1/2}\le\delta\\
&\quad\text{and}\ (u,p,\theta,\eta)\ \text{achieve the initial data in Appendix \ref{sec:initial}}\Big\}.
  \end{aligned}
  \end{equation}
  We endow this space with the metric
  \begin{equation}\label{def:metric}
  \begin{aligned}
  d((u,p,\theta,\eta),(v,q, \vartheta,\xi))&=\|u-v\|_{L^\infty H^1}+\|u-v\|_{L^2 H^1}+\|u-v\|_{L^2 W^{2,q_+}}+\|p-q\|_{L^2 W^{1,q_+}}\\
  &\quad+\|\theta-\vartheta\|_{L^\infty H^1}+\|\theta-\vartheta\|_{L^2 H^1}+\|\theta-\vartheta\|_{L^2 W^{2,q_+}}+\|\eta-\xi\|_{L^\infty H^1}\\
  &\quad+\|\eta-\xi\|_{L^2 H^{3/2-\alpha}}+\|\eta-\xi\|_{L^2 W^{3-1/q_+.q_+}}+\|\pa_tu-\pa_tv\|_{L^\infty H^1}\\
  &\quad+\|\pa_tu-\pa_tv\|_{L^2 H^1}+\|\pa_t\eta-\pa_t\xi\|_{L^\infty H^1}+\|\pa_t\eta-\pa_t\xi\|_{L^2 H^{3/2-\alpha}}\\
  &\quad+\varepsilon\|\pa_t^2\eta-\pa_t^2\xi\|_{L^2H^1}+\varepsilon^{1/2}\|\pa_t\eta-\pa_t\xi\|_{L^2 H^1}\\
  &\quad+\|[\pa_t\eta-\pa_t\xi]_\ell\|_{L^2([0,T])}+\|[\pa_t^2\eta-\pa_t^2\xi]_\ell\|_{L^2([0,T])}
  \end{aligned}
  \end{equation}
 for any $(u,p,\eta),(v,q,\xi)\in S(T, \delta)$, where the temporal norm is evaluated on the set $[0,T]$.
\end{definition}

It is easy to see that the space $S(T, \delta)$ is complete for each fixed $\delta$. Now we employ the metric space $S(T,\delta)$ and a contraction mapping argument to construct a sequence of approximate solutions $(u^\varepsilon, p^\varepsilon, \theta^\varepsilon, \eta^\varepsilon)$ for each $0<\varepsilon\le1$ satisfying \eqref{eq:epsilon}. For simplicity, the superscript $\varepsilon$ will be typically suppressed in the notations and  we denote the unknown as $(u,p,\theta,\eta)$ instead of $(u^\varepsilon, p^\varepsilon, \theta^\varepsilon, \eta^\varepsilon)$.

\begin{theorem}\label{thm: fixed point}
Suppose that there exists a universal constant $\delta>0$ sufficiently small such that the initial data for the nonlinear system are given as in Section \ref{sec:initial_nonlinear1} satisfying
\begin{align}
\|(u_0, p_0, \theta_0, \eta_0, \pa_tu(0), \pa_tp(0), \pa_t\theta(0), \pa_t\eta(0), \pa_t^2u(0), \pa_t^2\theta(0), \pa_t^2\eta(0), \pa_t^3\eta(0))\|_{X^\varepsilon}^2\le \delta.
\end{align}
Then there exists a constant $C>0$ such that for each $0<\varepsilon\le\min\{1, 1/(8C)\}$ there exists a unique solution $(u, p, \theta, \eta)$ to \eqref{eq:epsilon}, belonging to the metric space $S(T_\varepsilon, \delta)$, where $T_\varepsilon>0$ is sufficiently small.  In particular $(u, p, \theta, \eta)\in \mathcal{X}\cap\mathcal{Y}$, where $\mathcal{X}$ and $\mathcal{Y}$ are defined in \eqref{def:xy}.
\end{theorem}
\begin{proof}
\
\paragraph{\underline{Step 1 -- Solving the Linear Problem}}
Suppose that $\mathfrak{K}(u,p,\theta,\eta)\le\delta^2$ is sufficiently small. For every $(u,p,\theta,\eta)\in S(T_\varepsilon, \delta)$ given, let $(\tilde{u},\tilde{p},\tilde{\theta},\tilde{\eta})$ instead of $(u, p, \theta, \xi)$ be the unique solution to the linear problems \eqref{eq:epsilon}.
Then we apply the Theorem \ref{thm:linear_low} and \ref{thm:higher order} to system \eqref{eq:epsilon}, with the Propositions \ref{prop:force_bulk} under the assumption on $\mathfrak{K}(u,p,\theta,\eta)$ to obtain that
\begin{equation}
\mathfrak{K}(\tilde{u}, \tilde{p}, \tilde{\theta}, \tilde{\eta})\le C_0\Big[1+\frac T\varepsilon\Big]e^{T\delta}\left(C_1\Big(1+\frac{T}{\varepsilon^2}\Big)\mathcal{E}(0)+C_2\Big(1+\frac{T}{\varepsilon}\Big)\delta^4\right),
\end{equation}
for some universal constants $C_0$, $C_1$ and $C_2$. We first choose $T=T_\varepsilon$ such that $T_\varepsilon\le \varepsilon^2$ for any $\varepsilon\in(0, 1]$. Then we choose $\delta\in (0, 1)$ such that $4C_0C_2e\delta^2\le\frac12$. Finally, we restrict the initial data $\mathcal{E}(0)$ so that $4C_0C_1e\mathcal{E}(0)\le\frac12\delta^2$. Consequently, we have $\mathfrak{K}(\tilde{u}, \tilde{p}, \tilde{\theta}, \tilde{\eta})\le\delta^2$ so that $(\tilde{u}, \tilde{p}, \tilde{\theta}, \tilde{\eta})\in S(T_\varepsilon, \delta)$.

\paragraph{\underline{Step 2 -- PDEs for the Differences}}
Define the operator
\begin{equation}\label{def:map_a}
A:S(T_\varepsilon, \delta)\rightarrow S(T_\varepsilon, \delta)
\end{equation}
 via $A(u,p,\theta,\eta)=(\tilde{u},\tilde{p}, \tilde{\theta}, \tilde{\eta})$ for each $(u,p,\theta,\eta)\in S(T_\varepsilon, \delta)$.

  In order to show that the operator $A$ is contract on $S(T_\varepsilon, \delta)$, we choose two elements $(u^i,p^i,\theta^i,\eta^i)\in S(T_\varepsilon, \delta)$, and define $A(u^i,p^i, \theta^i, \eta^i)=(\tilde{u}^i,\tilde{p}^i, \tilde{\theta}^i, \tilde{\eta}^i)$ as above, $i=1,2$. For simplicity, we will abuse notation and denote $u=u^1-u^2$, $p=p^1-p^2$, $\theta = \theta^1 - \theta^2$, $\eta=\eta^1-\eta^2$ and the same for $\tilde{u},\tilde{p}, \tilde{\theta}, \tilde{\eta}$. From the difference of equation for $(\tilde{u}^i,\tilde{p}^i, \tilde{\theta}^i, \tilde{\eta}^i)$, $i=1,2$, we know that
\begin{equation}\label{linear_fix1}
  \left\{
  \begin{aligned}
&\pa_t\tilde{\theta}+ k\dive_{\mathcal{A}^1}\nabla_{\mathcal{A}^1}\tilde{\theta}= k\dive_{\mathcal{A}^1}(\nabla_{\mathcal{A}^1-\mathcal{A}^2}\tilde{\theta}^2)+R^8 \quad &\text{in}&\quad \Om,\\
&k \nabla_{\mathcal{A}^1}\tilde{\theta} \cdot\mathcal{N}^1 + \tilde{\theta}|\mathcal{N}^1| = k\nabla_{\mathcal{A}^1-\mathcal{A}^2}\tilde{\theta}^2\cdot\mathcal{N}^1 +R^9 \quad &\text{on}&\quad\Sigma,\\
&\tilde{\theta}=0 \quad &\text{on}&\quad\Sigma_s
  \end{aligned}
  \right.
\end{equation}
with zero initial data, where $R^8$  and $R^9$ are defined by
\begin{equation}\label{remainder3}
  \begin{aligned}
R^8&= k\dive_{(\mathcal{A}^1-\mathcal{A}^2)}(\nabla_{\mathcal{A}^2}\tilde{\theta}^2)-(\pa_t\bar{\eta}^1-\pa_t\bar{\eta}^2)K^1W\pa_2\theta^1+\pa_t\bar{\eta}^2(K^1-K^2)W\pa_2\theta^1\\
&\quad+\pa_t\bar{\eta}^2K^2W\pa_2\theta+u\cdot\nabla_{\mathcal{A}^1}\theta^1+u^2\cdot\nabla_{\mathcal{A}^1-\mathcal{A}^2}\theta^1+u^2\cdot\nabla_{\mathcal{A}^2}\theta,\\
R^9&= k\nabla_{\mathcal{A}^2}\tilde{\theta}^2\cdot(\mathcal{N}^1-\mathcal{N}^2)+\tilde{\theta}^2 |\mathcal{N}^1-\mathcal{N}^2|
  \end{aligned}
  \end{equation}
and $\mathcal{A}^i$, $\mathcal{N}^i$, $J^i$ and $K^i$ are in terms of $\eta^i$, $i=1,2$.

We now focus on \eqref{linear_fix1}. The solutions are regular enough to be differentiated in time to result in the system
\begin{equation}\label{linear_fix2}
   \left\{
  \begin{aligned}
&\pa_t^2\tilde{\theta} - k \dive_{\mathcal{A}^1}\nabla_{\mathcal{A}^1}\pa_t\tilde{\theta}= k \dive_{\mathcal{A}^1}(\nabla_{\mathcal{A}^1-\mathcal{A}^2}\pa_t\tilde{\theta}^2) +k \dive_{\mathcal{A}^1}(\nabla_{\pa_t(\mathcal{A}^1-\mathcal{A}^2)}\tilde{\theta}^2)\\
&\quad \quad \quad \quad \quad \quad \quad \quad \quad \quad \quad \quad \quad \quad + k \dive_{\mathcal{A}^1}\nabla_{\pa_t\mathcal{A}^1}\pa_t\tilde{\theta}+R^{8,1} \quad &\text{in}&\quad \Om,\\
&k \nabla_{\mathcal{A}^1}\pa_t\tilde{\theta} \cdot\mathcal{N}^1 + \pa_t\tilde{\theta}|\mathcal{N}^1| =- k\nabla_{\mathcal{A}^1-\mathcal{A}^2}\pa_t\tilde{\theta}^2\cdot\mathcal{N}^1 -k \nabla_{\pa_t(\mathcal{A}^1-\mathcal{A}^2)}\tilde{\theta}^2\cdot\mathcal{N}^1\\
&\quad \quad \quad \quad \quad \quad \quad \quad \quad \quad \quad \quad \quad \quad
-k \nabla_{\pa_t\mathcal{A}^1}\tilde{\theta} \cdot\mathcal{N}^1+R^{9,1} \quad &\text{on}&\quad\Sigma,\\
&\pa_t\tilde{\theta}=0 \quad &\text{on}&\quad\Sigma_s
  \end{aligned}
  \right.
\end{equation}
with zero initial data, where $R^{8,1}$, $R^{9,1}$ are defined by
  \begin{equation}\label{remainder2'}
  \begin{aligned}
R^{8,1}&=\pa_tR^8+ k \dive_{\pa_t\mathcal{A}^1}(\nabla_{(\mathcal{A}^1-\mathcal{A}^2)}\tilde{\theta}^2)+ k \dive_{\pa_t\mathcal{A}^1}(\nabla_{\mathcal{A}^1}\tilde{\theta}),\\
R^{9,1}&=\pa_tR^9- k\nabla_{\mathcal{A}^1}\tilde{\theta} \cdot\pa_t\mathcal{N}^1 -\tilde{\theta}\pa_t|\mathcal{N}^1|- k\nabla_{\mathcal{A}^1-\mathcal{A}^2}\tilde{\theta}^2\cdot\pa_t\mathcal{N}^1.
  \end{aligned}
  \end{equation}

\paragraph{\underline{Step 3 -- Energy Estimates for $\pa_t\tilde{\theta}$}}
By multiplying the first equation of \eqref{linear_fix2} by $\pa_t\tilde{\theta}J^1$, integrating over $\Om$ and integrating by parts, we have the equation
\begin{equation}\label{energy_fix}
\begin{aligned}
  &\frac{d}{dt}\int_\Om\frac{|\pa_t\tilde{\theta}|^2}2J^1+\frac k2\int_\Om|\nabla_{\mathcal{A}^1}\pa_t\tilde{\theta}|^2J^1+ \int_{\Sigma}|\pa_t\tilde{\theta}|^2|\mathcal{N}^1| \\
  &=\int_\Om\frac{|\pa_t\tilde{\theta}|^2}2\pa_tJ^1-\frac k2\int_\Om\left( \nabla_{\pa_t\mathcal{A}^1-\pa_t\mathcal{A}^2}\tilde{\theta}^2+\nabla_{\mathcal{A}^1-\mathcal{A}^2}\pa_t\tilde{\theta}^2+\nabla_{\pa_t\mathcal{A}^1}\tilde{\theta}\right)\cdot\nabla_{\mathcal{A}^1}\pa_t\tilde{\theta}J^1\\
  &\quad+\int_\Om R^{8,1}\cdot\pa_t\tilde{\theta}J^1 + \int_{-\ell}^\ell R^{9,1}\pa_t\tilde{\theta}.
\end{aligned}
\end{equation}

We now estimate the terms in right-hand side of \eqref{energy_fix} and below.  For terms in bulk, by definition in \eqref{remainder3}, we first handle terms involving $\pa_tR^8$:
\begin{equation}\label{est:remain_bulk11}
\begin{aligned}
&\int_\Om \pa_tR^8\cdot\pa_t\tilde{\theta}J^1\\
&\lesssim\int_\Om |\nabla\pa_t\bar{\eta}|(|\nabla^2\bar{\eta}^2||\nabla\tilde{\theta}^2|+|\nabla^2\tilde{\theta}^2|)|\pa_t\tilde{\theta}|+|\nabla\bar{\eta}|(|\nabla^2\pa_t\bar{\eta}^2||\nabla\tilde{\theta}^2|+|\nabla\pa_t\bar{\eta}^2||\nabla^2\tilde{\theta}^2|)|\pa_t\tilde{\theta}|\\
&\quad+\int_\Om |\nabla\bar{\eta}| (|\nabla^2\bar{\eta}^2||\nabla\pa_t\tilde{\theta}^2|+|\nabla^2\pa_t\tilde{\theta}^2|)|\pa_t\tilde{\theta}|+((|\pa_t\bar{\eta}^1|+|\pa_t\bar{\eta}^2|)|\nabla\pa_t\bar{\eta}|+|\pa_t^2\bar{\eta}^2||\nabla\bar{\eta}|)|\pa_2\theta^1||\pa_t\tilde{\theta}|\\
&\quad+\int_\Om(|\pa_t\bar{\eta}|+|\pa_t\bar{\eta}^2||\nabla\bar{\eta}|)|\pa_2\pa_t\theta^1||\pa_t\tilde{\theta}|+(|\pa_t^2\bar{\eta}^2|+|\pa_t\bar{\eta}^2||\nabla\pa_t\bar{\eta}^2|)|\pa_2\theta||\pa_t\tilde{\theta}|+|\pa_t\bar{\eta}^2||\pa_2\pa_t\theta||\pa_t\tilde{\theta}|\\
&\quad+\int_\Om(|\pa_tu|+|u||\nabla\pa_t\bar{\eta}^1|+|\pa_t u^2||\nabla\bar{\eta}|+| u^2||\nabla\pa_t\bar{\eta}|)|\nabla \theta^1||\pa_t\tilde{\theta}|+(|u|+|u^2||\nabla\bar{\eta}|)|\nabla \pa_t\theta^1||\pa_t\tilde{\theta}|\\
&\quad+\int_\Om(|\pa_tu^2|+|u^2||\nabla\pa_t\tilde{\eta}^2|)|\nabla \theta||\pa_t\tilde{\theta}|+|u^2||\nabla\pa_t\theta||\pa_t\tilde{\theta}|+|\pa_t^2\bar{\eta}||\pa_2\theta^1||\pa_t\tilde{\theta}|,
\end{aligned}
\end{equation}
where we have used $\|\mathcal{A}^i\|_{L^\infty}+\|J^i\|_{L^\infty}+\|K^i\|_{L^\infty}\lesssim 1$, $i=1, 2$.
Then by H\"older inequality, Sobolev embedding and trace theory, after integrating over $[0, T]$, the second line in \eqref{est:remain_bulk1} is bounded by
\begin{equation}\label{est:remain_bulk12}
\begin{aligned}
&\int_0^T\|\nabla\pa_t\bar{\eta}\|_{L^{2/\alpha}}(\|\nabla^2\bar{\eta}^2\|_{L^{2/(1-\varepsilon_+)}}\|\nabla\tilde{\theta}^2\|_{L^2}
+\|\nabla^2\tilde{\theta}^2\|_{L^{q_+}})\|\pa_t\tilde{\theta}\|_{L^{2/(\varepsilon_+-\alpha)}}\\
&\quad+\|\nabla\bar{\eta}\|_{L^{4/(\alpha-\varepsilon_-)}}\|\nabla^2\pa_t\bar{\eta}^2\|_{L^{4/(2-\varepsilon_-+\alpha)}}\|\nabla\tilde{\theta}^2\|_{L^{2/(1-\varepsilon_+)}}\|\pa_t\tilde{\theta}\|_{L^{2/\varepsilon_+}}\\
&\quad+\|\nabla\bar{\eta}\|_{L^\infty}\|\nabla\pa_t\bar{\eta}^2\|_{L^{2/(\varepsilon_+-\varepsilon_-)}}\|\nabla^2\tilde{\theta}^2\|_{L^{q_+}}\|\pa_t\tilde{\theta}\|_{L^{2/\varepsilon_-}},
\end{aligned}
\end{equation}
so that we have the estimate
\begin{equation}
\begin{aligned}
\eqref{est:remain_bulk12}&\lesssim \int_0^T\|\pa_t\eta\|_{3/2-\alpha}(\|\eta^2\|_{W^{3-1/q_+,q_+}}\|\tilde{\theta}^2\|_1+\|\tilde{\theta}^2\|_{W^{2,q_+}})\|\pa_t\tilde{\theta}\|_1\\
&\quad+\|\eta\|_{W^{3-1/q_+,q_+}}\|\pa_t\eta^2\|_{3/2+(\varepsilon_--\alpha)/2}\|\tilde{\theta}^2\|_{W^{2,q_+}}\|\pa_t\tilde{\theta}\|_{1}\\
&\lesssim \delta^{1/2}(\|\pa_t\eta\|_{L^2H^{3/2-\alpha}}+\|\eta\|_{L^2 W^{3-1/q_+,q_+}})\|\pa_t\tilde{\theta}\|_{L^2H^1}.
\end{aligned}
\end{equation}
The third line in \eqref{est:remain_bulk1}, after integrating over $[0, T]$ in time, are estimated by
\begin{equation}\label{est:remain_bulk13}
\begin{aligned}
&\int_0^T\|\nabla\bar{\eta}\|_{L^{2/\alpha}}\bigg(\|\nabla^2\bar{\eta}^2\|_{L^{2/(1-\varepsilon_+)}}\|\nabla\pa_t\tilde{\theta}^2\|_{L^{4/(2-\varepsilon_-)}}\|\pa_t\tilde{\theta}\|_{L^{4/(2\varepsilon_++\varepsilon_--2\alpha)}}+\|\nabla^2\pa_t\tilde{\theta}^2\|_{L^{q_-}}\\
&\quad\times\|\pa_t\tilde{\theta}\|_{L^{2/(\varepsilon_--\alpha)}}\bigg)+(\|\pa_t\bar{\eta}^1\|_{L^\infty}+\|\pa_t\bar{\eta}^2\|_{L^\infty})\|\nabla\pa_t\bar{\eta}\|_{L^{2/\alpha}}\|\pa_2\theta^1\|_{L^2}\|\pa_t\tilde{\theta}\|_{L^{2/(1-\alpha)}}\\
&\quad+\|\pa_t^2\bar{\eta}^2\|_{L^\infty}\|\nabla\bar{\eta}\|_{L^\infty}\|\pa_2\theta^1\|_{L^2}\|\pa_t\tilde{\theta}\|_{L^2}\\
&\lesssim \int_0^T\|\eta\|_{H^{3/2-\alpha}}(\|\bar{\eta}^2\|_{W^{3-1/q_+, q_+}}\|\pa_t\tilde{\theta}^2\|_{H^{1+\varepsilon_-/2}}+\|\pa_t\tilde{\theta}^2\|_{W^{2,q_-}})\|\pa_t\tilde{\theta}\|_{H^1}\\
&\quad+(\|\pa_t\eta^1\|_{3/2+(\varepsilon_--\alpha)/2}+\|\pa_t\eta^2\|_{3/2+(\varepsilon_--\alpha)/2})\|\pa_t\eta\|_{3/2-\alpha})\|\theta^1\|_{1}\|\pa_t\tilde{\theta}\|_{H^1}\\
&\quad+\|\pa_t^2\eta^2\|_{1}\|\eta\|_{W^{3-1/q_+,q_+}}\|\theta^1\|_{1}\|\pa_t\tilde{\theta}\|_{L^2}\\
&\lesssim \delta^{1/2}(\|\pa_t\eta\|_{L^2 H^{3/2-\alpha}}+\|\eta\|_{L^\infty H^{3/2-\alpha}} )\|\pa_t\tilde{\theta}\|_{L^2H^1}+\delta^{1/2}\|\eta\|_{L^2 W^{3-1/q_+,q_+}}\|\pa_t\tilde{\theta}\|_{L^\infty H^0},
\end{aligned}
\end{equation}
where we have used the Sobolev embedding $H^{\varepsilon_-/2}\hookrightarrow L^{4/(2-\varepsilon_-)}$ and $H^1\hookrightarrow L^{4/(2\varepsilon_++\varepsilon_--2\alpha)}$ for bounded domain $\Om$.
The fourth line in \eqref{est:remain_bulk1}, after integrating over $[0, T]$ in time (without loss of generality, we may assume $T<1$), are estimated by
\begin{equation}\label{est:remain_bulk14}
\begin{aligned}
&\int_0^T(\|\pa_t\bar{\eta}\|_{L^\infty}+\|\pa_t\bar{\eta}^2\|_{L^\infty}\|\nabla\bar{\eta}\|_{L^\infty})\|\nabla\pa_t\theta^1\|_{L^2}\|\pa_t\tilde{\theta}\|_{L^2}+\|\pa_2\theta\|_{L^2}\|\pa_t\tilde{\theta}\|_{L^2}(\|\pa_t^2\bar{\eta}^2\|_{L^\infty}\\
&\quad+\|\pa_t\bar{\eta}^2\|_{L^\infty}\|\nabla\pa_t\bar{\eta}^2\|_{L^\infty})+\|\pa_t\bar{\eta}^2\|_{L^\infty}\|\pa_2\pa_t\theta\|_{L^2}\|\pa_t\tilde{\theta}\|_{L^2}\\
&\lesssim \int_0^T(\|\pa_t\eta\|_{H^1}+\|\pa_t\eta^1\|_{H^1}\|\eta\|_{W^{3-1/q_+,q_+}})\|\pa_t\theta^1\|_{1+\varepsilon_-/2}\|\pa_t\tilde{\theta}\|_{L^2}+\|\theta\|_{1}\|\pa_t\tilde{\theta}\|_{L^2}(\|\pa_t^2\eta^2\|_{1}\\
&\quad+\|\pa_t\eta^2\|_{1}\|\pa_t\eta^2\|_{3/2+(\varepsilon_--\alpha)/2})+\|\pa_t\eta^2\|_{1}\|\pa_t\theta\|_{1}\|\pa_t\tilde{\theta}\|_{L^2}\\
&\lesssim \sqrt{T}\delta^{1/2}(\|\pa_t\eta\|_{L^\infty H^1}+\|\eta\|_{L^2 W^{3-1/q_+,q_+}}+\|\theta\|_{L^2W^{2,q_+}}+\|\pa_t\theta\|_{L^2H^1})\|\pa_t\tilde{\theta}\|_{L^\infty H^0}.
\end{aligned}
\end{equation}
After integrating over $[0, T]$ in time, the fifth line in \eqref{est:remain_bulk1} are estimated by
\begin{equation}\label{est:remain_bulk15_1}
\begin{aligned}
  &\int_0^T(\|\pa_tu\|_{L^2}+\|u\|_{L^2}\|\nabla\pa_t\bar{\eta}^1\|_{L^\infty}+\|\pa_t u^2\|_{L^2}\|\nabla\bar{\eta}\|_{L^\infty})\|\nabla \theta^1\|_{L^{2/(1-\varepsilon_+)}}\|\pa_t\tilde{\theta}\|_{L^{2/\varepsilon_+}}\\
&\quad+\|u^2\|_{L^2}\|\nabla\pa_t\bar{\eta}\|_{L^{2/\alpha}}\|\nabla \theta^1\|_{L^{2/(1-\varepsilon_+)}}\|\pa_t\tilde{\theta}\|_{L^{2/(\varepsilon_+-\alpha)}}+(\|u\|_{L^\infty}+\|u^2\|_{L^\infty}\|\nabla\bar{\eta}\|_{L^\infty})\\
&\quad \times \|\nabla \pa_t\theta^1\|_{L^2}\|\pa_t\tilde{\theta}\|_{L^2}\\
&\lesssim \int_0^T(\|\pa_tu\|_{L^2}+\|u\|_{L^2}\|\pa_t\eta^1\|_{3/2+(\varepsilon_--\alpha)/2}+\|\pa_t u^2\|_{L^2}\|\eta\|_{W^{3-1/q_+,q_+}})\| \theta^1\|_{W^{2,q_+}}\|\pa_t\tilde{\theta}\|_{1}\\
&\quad+(\|u\|_{W^{2,q_+}}+\|u^2\|_{W^{2,q_+}}\|\eta\|_{W^{3-1/q_+,q_+}})\| \pa_t\theta^1\|_{1}\|\pa_t\tilde{\theta}\|_{L^2}\\
&\quad +\|u^2\|_{L^2}\|\pa_t\eta\|_{3/2-\alpha}\| \theta^1\|_{W^{2,q_+}}\|\pa_t\tilde{\theta}\|_{1},
\end{aligned}
\end{equation}
so that
\begin{equation}\label{est:remain_bulk15}
\begin{aligned}
\eqref{est:remain_bulk15_1}&\lesssim \delta^{1/2}(\|\pa_t\eta\|_{L^2 H^{3/2-\alpha}}+\|\eta\|_{L^2 W^{3-1/q_+,q_+}}+\|u\|_{L^2W^{2,q_+}}+\|\pa_tu\|_{L^2H^1})\|\pa_t\tilde{\theta}\|_{L^2H^1}\\
&\quad+\delta^{1/2}(\|\eta\|_{L^2 W^{3-1/q_+,q_+}}+\|u\|_{L^2W^{2,q_+}})\|\pa_t\tilde{\theta}\|_{L^\infty H^0},
\end{aligned}
\end{equation}
Similarly, the sixth line in \eqref{est:remain_bulk1} are estimated by
\begin{equation}\label{est:remain_bulk16}
\begin{aligned}
&\int_0^T(\|\pa_tu^2\|_{L^\infty}+\|u^2\|_{L^\infty}\|\nabla\pa_t\tilde{\eta}^2\|_{L^\infty})\|\nabla \theta\|_{L^2}\|\pa_t\tilde{\theta}\|_{L^2}+\|u^2\|_{L^\infty}\|\nabla\pa_t\theta\|_{L^2}\|\pa_t\tilde{\theta}\|_{L^2}\\
&\quad+\|\pa_t^2\bar{\eta}\|_{L^\infty}\|\pa_2\theta^1\|_{L^2}\|\pa_t\tilde{\theta}\|_{L^2},
\end{aligned}
\end{equation}
so that
\begin{equation}
\begin{aligned}
\eqref{est:remain_bulk16}&\lesssim \int_0^T(\|\pa_tu^2\|_{1+\varepsilon_-/2}+\|u^2\|_{W^{2,q_+}}\|\pa_t\eta^2\|_{3/2+(\varepsilon_--\alpha)/2})\| \theta\|_{W^{2,q_+}}\|\pa_t\tilde{\theta}\|_{L^2}\\
&\quad+\|u^2\|_{W^{2,q_+}}\|\pa_t\theta\|_{1}\|\pa_t\tilde{\theta}\|_{L^2}+\|\pa_t^2\eta\|_{1}\|\theta^1\|_{W^{2,q_+}}\|\pa_t\tilde{\theta}\|_{L^2}\\
&\lesssim \delta^{1/2}(\sqrt{T}\|\pa_t^2\eta\|_{L^2 H^1}+\|u\|_{L^2W^{2,q_+}}+\|\pa_t\theta\|_{L^2H^1})\|\pa_t\tilde{\theta}\|_{L^\infty H^0},
\end{aligned}
\end{equation}
We combine \eqref{est:remain_bulk11}--\eqref{est:remain_bulk16} to obtain that
\begin{equation}\label{est:remain_bulk1}
\begin{aligned}
&\int_\Om \pa_tR^1\cdot\pa_t\tilde{\theta}J^1\\
&\lesssim\delta^{1/2}(\|\pa_t\eta\|_{L^2 H^{3/2-\alpha}}+\|\eta\|_{L^2 W^{3-1/q_+,q_+}}+\|\eta\|_{L^\infty H^{3/2-\alpha}} +\|u\|_{L^2W^{2,q_+}}+\|\pa_tu\|_{L^2H^1})\\
&\quad \times \|\pa_t\tilde{u}\|_{L^2H^1} +\delta^{1/2}(\sqrt{T}\|\pa_t^2\eta\|_{L^2 H^1}+\|\pa_t\eta\|_{L^\infty H^1}+\|\eta\|_{L^2 W^{3-1/q_+,q_+}}\\
&\quad+\|u\|_{L^2W^{2,q_+}}+\|\pa_tu\|_{L^2H^1})\|\pa_t\tilde{u}\|_{L^\infty H^0}.
\end{aligned}
\end{equation}

We now estimate other terms in the bulk such that
\begin{equation}\label{est:remain_bulk21}
\begin{aligned}
&\int_0^T\int_\Om \dive_{\pa_t\mathcal{A}^1}(\nabla_{(\mathcal{A}^1-\mathcal{A}^2)}\tilde{\theta}^2)\cdot\pa_t\tilde{\theta}J^1\lesssim\int_0^T\int_\Om |\nabla\pa_t\bar{\eta}^1|\left(|\nabla^2\bar{\eta}||\nabla\tilde{u}^2|+|\nabla\bar{\eta}||\nabla^2\tilde{\theta}^2|\right)|\pa_t\tilde{\theta}|\\
&\lesssim\int_0^T\|\nabla\pa_t\bar{\eta}^1\|_{L^\infty}(\|\nabla^2\bar{\eta}\|_{L^{2/(1-\varepsilon_+)}}\|\nabla\tilde{\theta}^2\|_{L^{2/(1-\varepsilon_+)}}\|\pa_t\tilde{\theta}\|_{L^{1/\varepsilon_+}}+\|\nabla\bar{\eta}\|_{L^\infty}\|\nabla^2\tilde{\theta}^2\|_{L^{q_+}}\|\pa_t\tilde{\theta}\|_{L^{2/\varepsilon_+}})\\
&\lesssim\int_0^T\|\pa_t\eta^1\|_{3/2+(\varepsilon_--\alpha)/2}\|\eta\|_{W^{3-1/q_+,q_+}}\|\tilde{\theta}^2\|_{W^{2,q_+}}\|\pa_t\tilde{\theta}\|_{1}\lesssim\delta\|\eta\|_{L^2W^{3-1/q_+,q_+}}\|\pa_t\tilde{\theta}\|_{L^2H^1},
\end{aligned}
\end{equation}
and
\begin{equation}\label{est:remain_bulk22}
\begin{aligned}
&\int_0^T\int_\Om \dive_{\pa_t\mathcal{A}^1}(\nabla_{\mathcal{A}^1}\tilde{\theta})\cdot\pa_t\tilde{\theta}J^1\lesssim\int_0^T\int_\Om |\nabla\pa_t\bar{\eta}^1|\left(|\nabla^2\bar{\eta}^1||\nabla\tilde{\theta}|+|\nabla^2\tilde{\theta}|\right)|\pa_t\tilde{\theta}|\\
&\lesssim\int_0^T\|\nabla\pa_t\bar{\eta}^1\|_{L^\infty}\big[\|\nabla^2\bar{\eta}^1\|_{L^{2/(1-\varepsilon_+)}}\|\nabla\tilde{\theta}\|_{L^{2/(1-\varepsilon_+)}}\|\pa_t\tilde{\theta}\|_{L^{1/\varepsilon_+}}+(\|\nabla^2\tilde{\theta}\|_{L^{q_+}})\|\pa_t\tilde{\theta}\|_{L^{2/\varepsilon_+}}\big]\\
&\lesssim\int_0^T\|\pa_t\eta^1\|_{3/2+(\varepsilon_--\alpha)/2}\big[(1+\|\eta^1\|_{W^{3-1/q_+,q_+}})\|\tilde{\theta}\|_{W^{2,q_+}}\big]\|\pa_t\tilde{\theta}\|_{1}\\
&\lesssim\delta^{1/2}\|\tilde{\theta}\|_{L^2W^{2,q_+}}\|\pa_t\tilde{\theta}\|_{L^2H^1}.
\end{aligned}
\end{equation}
Consequently, the result in \eqref{est:remain_bulk1}, \eqref{est:remain_bulk21} and \eqref{est:remain_bulk22} imply that
\begin{equation}\label{est:remain_bulk2}
\begin{aligned}
&\int_\Om R^{8,1}\cdot\pa_t\tilde{\theta}J^1\\
&\lesssim \delta^{1/2}(\|\pa_t\eta\|_{L^2 H^{3/2-\alpha}}+\|\eta\|_{L^2 W^{3-1/q_+,q_+}}+\|\eta\|_{L^\infty H^{3/2-\alpha}} +\|u\|_{L^2W^{2,q_+}}\\
&\quad +\|\pa_tu\|_{L^2H^1})\|\pa_t\tilde{\theta}\|_{L^2H^1}+\delta^{1/2}(\sqrt{T}\|\pa_t^2\eta\|_{L^2 H^1}+\|\pa_t\eta\|_{L^\infty H^1}+\|\eta\|_{L^2 W^{3-1/q_+,q_+}}\\
&\quad+\|u\|_{L^2W^{2,q_+}}+\|\pa_tu\|_{L^2H^1})\|\pa_t\tilde{\theta}\|_{L^\infty H^0}+\delta^{1/2}\|\tilde{u}\|_{L^2W^{2,q_+}}\|\pa_t\tilde{\theta}\|_{L^2H^1}
\end{aligned}
\end{equation}

The integrals in terms of symmetric gradient of $\pa_t\tilde{u}$ are estimated by
\begin{equation}\label{est:remain_bulk3}
  \begin{aligned}
&|\int_0^T\int_\Om\left( \nabla_{\pa_t\mathcal{A}^1-\pa_t\mathcal{A}^2}\tilde{\theta}^2+\nabla_{\mathcal{A}^1-\mathcal{A}^2}\pa_t\tilde{\theta}^2+\nabla_{\pa_t\mathcal{A}^1}\tilde{\theta}\right)\cdot \nabla_{\mathcal{A}^1}\pa_t\tilde{\theta}J^1|\\
&\lesssim \int_0^T(\|\nabla\pa_t\tilde{\eta}\|_{L^{2/\alpha}(\Om)}\|\nabla\tilde{\theta}^2\|_{L^{2/(1-\varepsilon_+)}}
+\|\nabla\tilde{\eta}\|_{L^{2/\alpha}(\Om)}\|\nabla\pa_t\tilde{\theta}^2\|_{L^{2/(1-\varepsilon_+)}}\\
&\quad+\|\nabla\pa_t\bar{\eta}^1\|_{L^\infty}\|\nabla\tilde{\theta}\|_{L^2})\|\nabla\pa_t\tilde{\theta}\|_{L^2}\\
&\lesssim \delta^{1/2}(\|\eta\|_{L^2H^{3/2-\alpha}}+\|\pa_t\eta\|_{L^2H^{3/2-\alpha}}+\|\tilde{\theta}\|_{L^2H^1})\|\tilde{\theta}\|_{L^2H^1}.
  \end{aligned}
\end{equation}

We now consider the integrals on the boundary.  We first estimate the terms of $\pa_tR^9$:
\begin{equation}\label{est:remain_b21}
  \begin{aligned}
|\int_0^T\int_{-\ell}^\ell \pa_tR^9\pa_t\tilde{\theta}|&\lesssim\int_0^T\int_{-\ell}^\ell\left[(|\nabla\pa_t\bar{\eta}^2||\nabla \tilde{\theta}^2|+|\nabla\pa_t\tilde{\theta}^2|)|\pa_1\eta|+(|\nabla \tilde{\theta}^2|)|\pa_1\pa_t\eta|\right]|\pa_t\tilde{\theta}|\\
&\quad+(|\pa_t\tilde{\theta}^2||\pa_1\eta|+|\tilde{\theta}^2||\pa_1\pa_t\eta|)|\pa_t\tilde{\theta}|.
  \end{aligned}
\end{equation}
The first line in the right-hand side of \eqref{est:remain_b21}, by H\"older and Sobolev inequalities, are estimated by
\begin{equation}\label{est:remain_b22}
  \begin{aligned}
&\int_0^T\bigg[(\|\nabla\pa_t\bar{\eta}^2\|_{L^\infty}\|\nabla \tilde{\theta}^2\|_{L^{1/(1-\varepsilon_-)}(\Sigma)}+\|\nabla\pa_t\tilde{\theta}^2\|_{L^{1/(1-\varepsilon_-)}(\Sigma)})\bigg]\|\pa_t\tilde{\theta}\|_{L^{1/(\varepsilon_--\alpha)}(\Sigma)}\\
 &\quad\quad\times\|\pa_1\eta\|_{L^{1/\alpha}}+\|\nabla \tilde{\theta}^2\|_{L^{1/(1-\varepsilon_+)}(\Sigma)}\|\pa_1\pa_t\eta\|_{L^{1/\alpha}}\|\pa_t\tilde{\theta}\|_{L^{1/(\varepsilon_+-\alpha)}(\Sigma)}\\
&\lesssim\int_0^T\bigg[(|\pa_t\eta^2\|_{3/2+(\varepsilon_--\alpha)/2}\| \tilde{\theta}^2\|_{W^{2,q_+}}+\|\pa_t\tilde{\theta}^2\|_{W^{2,q_-}})\bigg]\|\pa_t\tilde{\theta}\|_{1}\|\eta\|_{3/2-\alpha}\\
&\quad\quad+\|\tilde{\theta}^2\|_{W^{2,q_+}}\|\pa_t\eta\|_{3/2-\alpha}\|\pa_t\tilde{\theta}\|_{1}\\
&\lesssim\delta^{1/2}(\|\eta\|_{L^{\infty}H^{3/2-\alpha}}+\|\pa_t\eta\|_{L^2H^{3/2-\alpha}})\|\pa_t\tilde{\theta}\|_{L^2H^1}.
  \end{aligned}
\end{equation}
The second line in the right-hand side of \eqref{est:remain_b21}, by the similar method, are bounded by
\begin{equation}\label{est:remain_b25}
  \begin{aligned}
&\int_0^T(\|\pa_t\theta^2\|_{L^{1/(1-\varepsilon_-)}}\|\pa_1\eta\|_{L^{1/\alpha}}+\|\theta^2\|_{L^{1/(1-\varepsilon_-)}}\|\pa_1\pa_t\eta\|_{L^{1/\alpha}})\|\pa_t\tilde{\theta}\|_{L^{1/(\varepsilon_--\alpha)}(\Sigma)}\\
&\lesssim \delta^{1/2}(\|\eta\|_{L^\infty H^{3/2-\alpha}}+\|\pa_t\eta\|_{L^2 H^{3/2-\alpha}})\|\pa_t\tilde{\theta}\|_{L^2 H^1}.
  \end{aligned}
\end{equation}
The combination of \eqref{est:remain_b21} and \eqref{est:remain_b25} gives the bound
\begin{equation}\label{est:remain_b2}
  \begin{aligned}
|\int_0^T\int_{-\ell}^\ell \pa_tR^9\cdot\pa_t\tilde{\theta}|\lesssim \delta^{1/2}(\|\pa_t\eta\|_{L^\infty H^1}+\|\eta\|_{L^{\infty}H^{3/2-\alpha}}+\|\pa_t\eta\|_{L^2H^{3/2-\alpha}})\|\pa_t\tilde{\theta}\|_{L^2H^1}.
  \end{aligned}
\end{equation}
Similarly, we estimate
\begin{equation}\label{est:remain_b31}
  \begin{aligned}
|\int_0^T\int_{-\ell}^\ell \nabla_{\mathcal{A}^1-\mathcal{A}^2}\tilde{\theta}^2\pa_t\mathcal{N}^1 \pa_t\tilde{\theta} |
&\lesssim \int_0^T\int_{-\ell}^\ell|\nabla\bar{\eta}||\nabla \tilde{\theta}^2||\pa_t\pa_1\eta^1||\pa_t\tilde{\theta}|\\
&\lesssim \int_0^T\|\nabla\bar{\eta}\|_{L^{1/\alpha}}\|\nabla \tilde{\theta}^2\|_{L^{1/(1-\varepsilon_+)}(\Sigma)}\|\pa_t\pa_1\eta^1\|_{L^{1/\alpha}}\|\pa_t\tilde{\theta}\|_{L^{1/(\varepsilon_+-2\alpha)}}\\
&\lesssim\int_0^T \|\tilde{\eta}\|_{W^{3-1/q_+,q_+}}\|\tilde{\theta}^2\|_{W^{2,q_+}}\|\pa_t\eta^1\|_{3/2+}\|\pa_t\tilde{\theta}\|_1\\
&\lesssim \delta^{1/2}\|\tilde{\eta}\|_{L^2W^{3-1/q_+,q_+}}\|\pa_t\tilde{\theta}\|_{L^2H^1}.
  \end{aligned}
\end{equation}

So the results in \eqref{est:remain_b2} and \eqref{est:remain_b31} imply
\begin{equation}\label{est:remain_b3}
  \begin{aligned}
|\int_0^T\int_{-\ell}^\ell R^{9,1}\cdot\pa_t\tilde{\theta}\lesssim\delta^{1/2}(\|\pa_t\eta\|_{L^\infty H^1}+\|\eta\|_{L^{\infty}H^{3/2-\alpha}}+\|\pa_t\eta\|_{L^2H^{3/2-\alpha}}\\
+\|\tilde{\eta}\|_{L^2W^{3-1/q_+,q_+}}+\|\tilde{\theta}\|_{L^2W^{2,q_+}})\|\pa_t\tilde{\theta}\|_{L^2H^1}
  \end{aligned}
\end{equation}

After integrating \eqref{energy_fix} over $[0,T]$, we use the Cauchy inequality and Lemma \eqref{lem:equivalence_norm}, Gronwall's Lemma along with \eqref{est:remain_bulk2}, \eqref{est:remain_bulk3} and\eqref{est:remain_b3} to deduce the bound
\begin{equation}\label{est:contract_1}
  \begin{aligned}
&\|\pa_t\tilde{\theta}\|_{L^\infty H^0}^2+\|\pa_t\tilde{\theta}\|_{L^2H^1}^2+\|\pa_t\tilde{\theta}\|_{L^2H^0(\Sigma)}^2\\
&\le e^{C_1T\delta^{1/2}}C_2\delta(\|\eta\|_{L^\infty H^1}^2+\|\eta\|_{L^2 W^{3-1/q_+,q_+}}^2+\|\eta\|_{L^2 H^{3/2-\alpha}}^2 +\|\pa_t\eta\|_{L^2 H^{3/2-\alpha}}^2\\
&\quad+\|\pa_t\eta\|_{L^\infty H^1}^2+T\|\pa_t^2\eta\|_{L^2 H^1}^2+\|u\|_{L^2W^{2,q_+}}^2+\|\pa_tu\|_{L^2H^1}^2+\|\tilde{u}\|_{L^2W^{2,q_+}}^2\\
&\quad+\|\tilde{u}\|_{L^\infty H^0}^2+\|\pa_t\tilde{u}\|_{L^2H^1}^2+\|\tilde{\eta}\|_{L^2W^{3-1/q_+,q_+}}^2+\|\pa_t\tilde{\theta}\|_{L^2H^1}^2+\|\tilde{\theta}\|_{L^2W^{2,q_+}}^2),
  \end{aligned}
\end{equation}
where the universal constants $C_1$ and $C_2$ are determined by the estimates \eqref{est:remain_bulk2}, \eqref{est:remain_bulk3} and \eqref{est:remain_b3}.

\paragraph{\underline{Step 4 -- Energy Estimates for $\tilde{\theta}$}}
By multiplying the first equation of \eqref{linear_fix1} by $\tilde{\theta}J^1$, integrating over $\Om$ and integrating by parts, we have the equation
\begin{equation}\label{energy_fix2}
\begin{aligned}
  &\frac{d}{dt}\int_\Om\frac{|\tilde{\theta}|^2}2J^1+ \int_\Om|\nabla_{\mathcal{A}^1}\tilde{\theta}|^2J^1+\int_{\Sigma}|\tilde{\theta}|^2|\mathcal{N}^1| \\
  &=\int_\Om\frac{|\tilde{\theta}|^2}2\pa_tJ^1-\frac\mu2\int_\Om \nabla_{\mathcal{A}^1-\mathcal{A}^2}\tilde{\theta}^2 \cdot \nabla_{\mathcal{A}^1}\tilde{\theta}J^1+\int_\Om R^8\pa_t\tilde{\theta}J^1 + \int_{-\ell}^\ell R^9\tilde{\theta}.
\end{aligned}
\end{equation}
Then we use the same arguments as in Step 3 to estimate the terms in the third line of \eqref{energy_fix2}:
\begin{equation}
  \begin{aligned}
&|\int_0^T\int_\Om\frac{|\tilde{\theta}|^2}2\pa_tJ^1-\int_0^T\int_\Om \nabla_{\mathcal{A}^1-\mathcal{A}^2}\tilde{\theta}^2 \cdot \nabla_{\mathcal{A}^1}\tilde{\theta}J^1|\\
&\lesssim \int_0^T\|\nabla\pa_t\bar{\eta}^1\|_{L^\infty(\Om)}\|\tilde{\theta}\|_{L^2}^2+\|\nabla\bar{\eta}\|_{L^\infty(\Om)}\|\tilde{\theta}^2\|_{H^1}\|\tilde{\theta}\|_{H^1}\\
&\lesssim\delta^{1/2}(T\|\tilde{\theta}\|_{L^\infty H^0}^2+\|\eta\|_{L^2W^{3-1/q_+,q_+}}\|\tilde{\theta}\|_{L^2H^1}),
  \end{aligned}
\end{equation}
and
\begin{equation}\label{est:remain_bulk41}
  \begin{aligned}
&|\int_0^T\int_\Om R^8 \tilde{\theta}J^1|\\
&\lesssim\int_0^T\|\nabla\bar{\eta}\|_{L^\infty(\Om)}(\|\nabla^2\tilde{\theta}^2\|_{L^{q_+}}\|\tilde{\theta}\|_{L^{2/\varepsilon_+}}+\|\nabla\tilde{\theta}^2\|_{L^{2/(1-\varepsilon_+)}}\|\nabla^2\bar{\eta}^2\|_{L^{2/(1-\varepsilon_+)}}\|\tilde{\theta}\|_{L^{1/\varepsilon_+}})\\
  &\quad+\int_0^T(\|\pa_t\bar{\eta}\|_{L^\infty(\Om)}+\|\pa_t\bar{\eta}^1\|_{L^\infty(\Om)}\|\nabla\bar{\eta}\|_{L^\infty(\Om)})\|\pa_2\theta^1\|_{L^2}\|\tilde{\theta}\|_{L^2} + \|u^2\|_{L^\infty}\|\nabla \theta\|_{L^2}\|\tilde{\theta}\|_{L^2}\\
  &\quad+\int_0^T\|\pa_t\bar{\eta}^1\|_{L^\infty(\Om)}\|\pa_2\theta\|_{L^2}\|\tilde{\theta}\|_{L^2}+(\|u\|_{L^\infty}+\|u^2\|_{L^\infty}\|\nabla\bar{\eta}\|_{L^\infty(\Om)})\|\nabla \theta^1\|_{L^2}\|\tilde{\theta}\|_{L^2}\\
  &\lesssim \delta^{1/2}(\|\eta\|_{L^2W^{3-1/q_+,q_+}}+\|\pa_t\eta\|_{L^\infty H^1}+\|u\|_{L^2W^{2,q_+}}+\|\theta\|_{L^2W^{2,q_+}})\|\tilde{\theta}\|_{L^\infty H^0}\\
  &\quad + \delta^{1/2}(\|\eta\|_{L^2W^{3-1/q_+,q_+}}\|\tilde{\theta}\|_{L^2H^1}).
  \end{aligned}
\end{equation}
The term with $R^9$ are estimated by
\begin{equation}\label{est:remain_b5}
  \begin{aligned}
\bigg|\int_0^T\int_{-\ell}^\ell R^9\cdot\tilde{\theta}\bigg|
&\lesssim\int_0^T\|\nabla\bar{\eta}^2\|_{L^\infty(\Sigma)}\|\nabla\tilde{\theta}^2\|_{L^{1/(1-\varepsilon_+)}(\Sigma)}\|\pa_1\eta\|_{L^\infty}\|\tilde{\theta}\|_{L^{1/\varepsilon_+}(\Sigma)}\\
&\quad+\|\pa_1\eta^2\|_{L^\infty}^2\|\pa_1\eta\|_{L^2}\|\tilde{\theta}\|_{L^2(\Sigma)}\\
  &\lesssim \delta^{1/2}(\|\eta\|_{L^2W^{3-1/q_+,q_+}}+\|\eta\|_{L^2H^1})\|\tilde{\theta}\|_{L^2H^1},
  \end{aligned}
\end{equation}

Then by combining \eqref{energy_fix2}--\eqref{est:remain_b5}, we have
\begin{equation}\label{est:contract_2}
  \begin{aligned}
&\|\tilde{\theta}\|_{L^\infty H^0}^2+\|\tilde{\theta}\|_{L^2H^1}^2+\|\tilde{\theta}\|_{L^2H^0(\Sigma)}^2\\
&\le e^{C_4T\delta^{1/2}}C_5\delta(\|\eta\|_{L^\infty H^1}^2+\|\eta\|_{L^2 W^{3-1/q_+,q_+}}^2+\|\eta\|_{L^2 H^{3/2-\alpha}}^2 +\|\pa_t\eta\|_{L^\infty H^1}^2\\
&\quad+\|u\|_{L^2W^{2,q_+}}^2+\|\theta\|_{L^2W^{2,q_+}}^2+\|\tilde{\theta}\|_{L^2W^{2,q_+}}^2),
  \end{aligned}
\end{equation}
for some uniform constants $C_4$ and $C_5$.

\paragraph{\underline{Step 5 -- Elliptic Estimates for $\tilde{\theta}$}}

The elliptic estimates for \eqref{linear_fix1} enjoys the esitmate
\begin{equation}\label{est:contract_ellip1}
  \begin{aligned}
\|\tilde{\theta}\|_{L^2W^{2,q_+}}^2& \lesssim\|-\dive_{\mathcal{A}^1}(\nabla_{\mathcal{A}^1-\mathcal{A}^2}\tilde{\theta}^2)+R^8-\pa_t\tilde{\theta}\|_{L^2L^{q_+}}^2\\
&\quad+\|\mathbb{D}_{\mathcal{A}^1-\mathcal{A}^2}\tilde{\theta}^2\mathcal{N}^1+R^9\|_{L^2W^{1-1/q_+,q_+}}^2.
  \end{aligned}
\end{equation}
We now estimate the terms under the first line in \eqref{est:contract_ellip1}. By $\|\mathcal{A}^1\|_{L^\infty}\lesssim1$,
\begin{equation}
  \begin{aligned}
\|\dive_{\mathcal{A}^1}(\nabla_{\mathcal{A}^1-\mathcal{A}^2}\tilde{\theta}^2)\|_{L^2L^{q_+}}^2 & \lesssim \||\nabla\bar{\eta}|\nabla^2\tilde{\theta}^2\|_{L^2L^{q_+}}^2+\||\nabla^2\bar{\eta}|\nabla\tilde{\theta}^2\|_{L^2L^{q_+}}^2\\
&\lesssim \|\nabla\bar{\eta}\|_{L^2L^\infty}^2\|\nabla^2\tilde{\theta}^2\|_{L^\infty L^{q_+}}^2+\|\nabla^2\bar{\eta}\|_{L^2L^2}^2\|\nabla\tilde{\theta}^2\|_{L^\infty L^{2/(1-\varepsilon)}}^2\\
& \lesssim \|\eta\|_{L^2W^{3-1/q_+,q_+}}^2\|\tilde{\theta}^2\|_{L^\infty W^{2,q_+}}^2\\
& \lesssim\delta \|\eta\|_{L^2W^{3-1/q_+,q_+}}^2.
  \end{aligned}
\end{equation}
The definition of $R^1$, with $\|K^i\|_{L^\infty}\lesssim1$ and  $\delta<1$, shows that
\begin{equation}
  \begin{aligned}
\|R^8\|_{L^2L^{q_+}}^2&\lesssim \|\nabla\bar{\eta}\|_{L^2L^\infty(\Om)}^2\|\nabla^2\tilde{\theta}^2\|_{L^\infty L^{q_+}}^2+\|\pa_t\bar{\eta}\|_{L^2L^2(\Om)}\|\nabla\tilde{\theta}^1\|_{L^\infty L^{2/(1-\varepsilon_+)}}\\
&\quad+\|\pa_t\bar{\eta}^2\|_{L^\infty L^2(\Om)}^2\|\nabla\bar{\eta}\|_{L^2L^\infty(\Om)}^2\|\nabla\tilde{\theta}^1\|_{L^\infty L^{2/(1-\varepsilon_+)}}^2+\|\pa_t\bar{\eta}^2\|_{L^\infty L^2(\Om)}^2\|\nabla\tilde{\theta}\|_{L^2 L^{2/(1-\varepsilon_+)}}^2\\
&\quad+\|u\|_{L^2L^2}\|\nabla\tilde{\theta}^1\|_{L^\infty L^{2/(1-\varepsilon_+)}}+\|u^2\|_{L^\infty L^2}^2\|\nabla\bar{\eta}\|_{L^2L^\infty}^2\|\nabla\tilde{\theta}^1\|_{L^\infty L^{2/(1-\varepsilon_+)}}^2\\
&\quad+\|u^2\|_{L^\infty L^2}^2\|\nabla\tilde{\theta}\|_{L^2 L^{2/(1-\varepsilon_+)}}^2\\
&\lesssim \|\eta\|_{L^2W^{3-1/q_+,q_+}}^2\|\tilde{\theta}^2\|_{L^\infty W^{2,q_+}}^2+(\|\pa_t\eta\|_{L^2H^{3/2-\alpha}}^2+\|\pa_t\eta^2\|_{L^\infty H^{3/2}}^2\|\eta\|_{L^2W^{3-1/q_+,q_+}}^2\\
&\quad+\|u\|_{L^2H^1}^2+\|u^2\|_{L^\infty L^2}^2\|\eta\|_{L^2W^{3-1/q_+,q_+}}^2)\|\tilde{u}^1\|_{L^\infty W^{2,q_+}}^2\\
&\quad+(\|\pa_t\eta^2\|_{L^\infty H^{3/2}}^2+\|u^2\|_{L^\infty L^2}^2)\|\tilde{\theta}\|_{L^2W^{2,q_+}}^2\\
&\lesssim \delta(\|\eta\|_{L^2W^{3-1/q_+,q_+}}^2+\|u\|_{L^2H^1}^2+\|\tilde{\theta}\|_{L^2W^{2,q_+}}^2),
  \end{aligned}
\end{equation}

It is trivial that
\begin{equation}
  \|\pa_t\tilde{\theta}\|_{L^2L^{q_+}}^2\lesssim T\|\pa_t\tilde{\theta}\|_{L^\infty L^2}^2, \quad q_+<2.
\end{equation}

Note that $W^{1,q_+}(\Sigma)$ is also an algebra. Then it can  be used to bound
\begin{equation}
  \begin{aligned}
\|\nabla_{\mathcal{A}^1-\mathcal{A}^2}\tilde{\theta}^2\mathcal{N}^1\|_{L^2W^{1-1/q_+,q_+}}^2&\lesssim \|\nabla\tilde{\theta}^2\|_{L^\infty W^{1-1/q_+,q_+}(\Sigma)}^2\|\nabla\bar{\eta}\|_{L^2W^{1,q_+}(\Sigma)}^2\lesssim \delta\|\eta\|_{L^2W^{3-1/q_+,q_+}}^2,
  \end{aligned}
\end{equation}
and
\begin{equation}\label{est:contract_ellip7}
  \begin{aligned}
\|R^9\|_{L^2W^{1-1/q_+,q_+}}^2&\lesssim \bigg(\|\mathcal{A}^1\|_{L^\infty W^{1,q_+}(\Sigma)}\|\nabla\tilde{\theta}^2\|_{L^\infty W^{1-1/q_+,q_+}(\Sigma)}^2 + \|\tilde{\theta}^2\|_{L^\infty W^{1, q_+}(\Sigma)}\bigg)\|\pa_1\eta\|_{L^2W^{1,q_+}}^2\\
&\lesssim\|\tilde{\theta}^2\|_{L^\infty W^{2,q_+}}^2\|\tilde{\eta}\|_{L^2 W^{3-1/q_+,q_+}}^2\\
&\lesssim\delta\|\tilde{\eta}\|_{L^2 W^{3-1/q_+,q_+}}^2.
  \end{aligned}
\end{equation}

The combination of \eqref{est:contract_ellip1}--\eqref{est:contract_ellip7}, after restricting $\delta$ to be smaller if necessary, implies that
\begin{equation}\label{est:contract_6}
  \begin{aligned}
\|\tilde{\theta}\|_{L^2W^{2,q_+}}^2\lesssim \delta(\|\eta\|_{L^2W^{3-1/q_+,q_+}}^2+\|u\|_{L^2H^1}^2+\|\tilde{\eta}\|_{L^2 W^{3-1/q_+,q_+}}^2)+\|\pa_t\tilde{\theta}\|_{L^\infty L^2}^2.
  \end{aligned}
\end{equation}

\paragraph{\underline{Step 6 -- Fixed Point}}
We now choose $T<T_\varepsilon<\varepsilon$ for each $\varepsilon\in (0,1)$. Then the differences for Navier-Stokes equations are handled like \cite{GTWZ2023}, which along with the estimates \eqref{est:contract_1} and \eqref{est:contract_2} allow us to get the estimate
\begin{equation}\label{est:fixed_point1}
  \begin{aligned}
&\|\pa_t\tilde{u}\|_{L^\infty H^0}^2+\|\pa_t\tilde{u}\|_{L^2H^1}^2+\|\pa_t\tilde{\eta}\|_{L^\infty H^1}^2+\varepsilon\|\pa_t^2\tilde{\eta}\|_{L^2 H^1}^2+\|[\pa_t^2\tilde{\eta}]_\ell\|_{L^2}^2+\|\pa_t\tilde{\eta}\|_{L^2H^{3/2-\alpha}}^2\\
&\quad+\|\tilde{u}\|_{L^\infty H^1}^2+\|\tilde{u}\|_{L^2H^1}^2+\|\tilde{\eta}\|_{L^\infty H^1}^2+\|[\pa_t\tilde{\eta}]_\ell\|_{L^2}^2+\|\tilde{\eta}\|_{L^2H^{3/2-\alpha}}^2+\|\tilde{p}\|_{L^\infty H^0}^2\\
&\quad + \|\pa_t\tilde{\theta}\|_{L^\infty H^0}^2 + \|\pa_t\tilde{\theta}\|_{L^2H^1}^2 + \|\tilde{\theta}\|_{L^\infty H^1}^2 + \|\tilde{\theta}\|_{L^2H^1}^2 +\|\tilde{\theta}\|_{L^2W^{2,q_+}}^2\\
&\quad+\|\tilde{u}\|_{L^2W^{2,q_+}}^2+\|\tilde{p}\|_{L^2W^{1,q_+}}^2+\|\tilde{\eta}\|_{L^2W^{3-1/q_+,q_+}}^2+\varepsilon^2\|\pa_t\tilde{\eta}\|_{L^2W^{3-1/q_+,q_+}}^2\\
&\le C\delta^{1/2}(\|\eta\|_{L^\infty H^1}^2+\|\eta\|_{L^2 W^{3-1/q_+,q_+}}^2+\|\eta\|_{L^2 H^{3/2-\alpha}}^2 +\|\pa_t\eta\|_{L^2 H^{3/2-\alpha}}^2+\|\pa_t\eta\|_{L^\infty H^1}^2+\varepsilon\|\pa_t^2\eta\|_{L^2 H^1}^2\\
&\quad+\|u\|_{L^2W^{2,q_+}}^2+\|\pa_tu\|_{L^2H^1}^2+\|u\|_{L^2H^1}^2+\|[\pa_t^2\eta]_\ell\|_{L^2}^2+\|[\pa_t\eta]_\ell\|_{L^2}^2+ \|\theta\|_{L^2W^{2,q_+}}^2 + \|\pa_t\theta\|_{L^2H^1}^2\\
&\quad + \|\theta\|_{L^2H^1}^2+\|\tilde{\theta}\|_{L^\infty H^0}^2+ \|\tilde{\theta}\|_{L^2W^{2,q_+}}^2 + \|\pa_t\tilde{\theta}\|_{L^2H^1}^2 +\|\tilde{u}\|_{L^\infty H^0}^2+ \|\tilde{u}\|_{L^2W^{2,q_+}}^2 + \|\pa_t\tilde{u}\|_{L^2H^1}^2\\
&\quad +\|\tilde{p}\|_{L^2W^{1,q_+}}^2 +\|\tilde{p}\|_{L^\infty H^0}^2+\|\tilde{\eta}\|_{L^2W^{3-1/q_+,q_+}}^2+\varepsilon^2\|\tilde{\eta}\|_{L^2W^{3-1/q_+,q_+}}^2+\|[\pa_t^2\tilde{\eta}]_\ell\|_{L^2}^2 ),
  \end{aligned}
\end{equation}
where $C$ is a universal constant allowed to change from line to line, and we have used $\delta<\frac12$. Then we might restrict $\delta$ to be smaller such that all the terms with superscript $\widetilde{}$ below the sign of inequality in \eqref{est:fixed_point1} be absorbed into terms above the sign of inequality in \eqref{est:fixed_point1}. Moreover, we might restrict $\delta$ to be smaller again so that we can deduce the following inequality:
\begin{equation}\label{est:fixed_point2}
  \begin{aligned}
&\|\pa_t\tilde{u}\|_{L^\infty H^0}^2+\|\pa_t\tilde{u}\|_{L^2H^1}^2+\|\pa_t\tilde{\eta}\|_{L^\infty H^1}^2+\varepsilon\|\pa_t^2\tilde{\eta}\|_{L^2 H^1}^2+\|[\pa_t^2\tilde{\eta}]_\ell\|_{L^2}^2+\|\pa_t\tilde{\eta}\|_{L^2H^{3/2-\alpha}}^2\\
&\quad+\|\tilde{u}\|_{L^\infty H^1}^2+\|\tilde{u}\|_{L^2H^1}^2+\|\tilde{\eta}\|_{L^\infty H^1}^2+\|[\pa_t\tilde{\eta}]_\ell\|_{L^2}^2+\|\tilde{\eta}\|_{L^2H^{3/2-\alpha}}^2+\|\tilde{p}\|_{L^\infty H^0}^2\\
&\quad + \|\pa_t\tilde{\theta}\|_{L^\infty H^0}^2 + \|\pa_t\tilde{\theta}\|_{L^2H^1}^2 + \|\tilde{\theta}\|_{L^\infty H^1}^2 + \|\tilde{\theta}\|_{L^2H^1}^2 +\|\tilde{\theta}\|_{L^2W^{2,q_+}}^2\\
&\quad+\|\tilde{u}\|_{L^2W^{2,q_+}}^2+\|\tilde{p}\|_{L^2W^{1,q_+}}^2+\|\tilde{\eta}\|_{L^2W^{3-1/q_+,q_+}}^2+\varepsilon^2\|\pa_t\tilde{\eta}\|_{L^2W^{3-1/q_+,q_+}}^2\\
&\le\tilde{C}(\|\eta\|_{L^\infty H^1}^2+\|\eta\|_{L^2 W^{3-1/q_+,q_+}}^2+\|\eta\|_{L^2 H^{3/2-\alpha}}^2 +\|\pa_t\eta\|_{L^2 H^{3/2-\alpha}}^2+\|\pa_t\eta\|_{L^\infty H^1}^2+\varepsilon\|\pa_t^2\eta\|_{L^2 H^1}^2\\
&\quad+\|u\|_{L^2W^{2,q_+}}^2+\|\pa_tu\|_{L^2H^1}^2+\|u\|_{L^2H^1}^2 +\|\theta\|_{L^2W^{2,q_+}}^2 + \|\pa_t\theta\|_{L^2H^1}^2 + \|\theta\|_{L^2H^1}^2\\
&\quad +\|[\pa_t^2\eta]_\ell\|_{L^2}^2+\|[\pa_t\eta]_\ell\|_{L^2}^2),
  \end{aligned}
\end{equation}
where $\tilde{C}$ is a small universal constant that guarantees the square root of \eqref{est:fixed_point2} satisfying
\begin{equation}
  d((\tilde{u}^1,\tilde{p}^1,\tilde{\theta}^1, \tilde{\eta}^1),(\tilde{u}^2,\tilde{p}^2,\tilde{\theta}^2,\tilde{\eta}^2))<\frac12d((u^1,0, \theta^1, \eta^1), (u^2, 0, \theta^2, \eta^2)).
\end{equation}
Here $d(\cdot,\cdot)$ is the metric defined in \eqref{def:metric}. Consequently, the map $A$ defined in \eqref{def:map_a} is contracted. So the Banach's fixed point theory implies that the nonlinear $\varepsilon-$ regularized system \eqref{eq:epsilon} with $(D_t^ju, \pa_t^jp, \pa_t^j\theta, \pa_t^j\eta)$ instead of $(D_t^ju, \pa_t^jp, \pa_t^j\theta, \pa_t^j\xi)$, has a solution $(u, p, \theta, \eta)\in S(T,\delta)$ for $T<T_\varepsilon\le\varepsilon$.
The uniqueness for \eqref{eq:epsilon} with $(D_t^ju, \pa_t^jp, \pa_t^j\theta, \pa_t^j\eta)$ instead of $(D_t^ju, \pa_t^jp, \pa_t^j\theta, \pa_t^j\xi)$ can  be proved similarly.
\end{proof}

\section{Global Well-Posedness for the Boussinesq System}

\subsection{Uniform Bounds for $\varepsilon$--Regularized Solutions}

We have proved the the existence and uniqueness of solutions $(u^\varepsilon, p^\varepsilon, \theta^\varepsilon, \eta^\varepsilon)$ to the regularized system \eqref{eq:geometric1} when $\varepsilon>0$. In order to get the solutions solving \eqref{eq:geometric} when $\varepsilon=0$, we have to prove that the solutions $(u^\varepsilon, p^\varepsilon, \theta^\varepsilon, \eta^\varepsilon)$ obtained to \eqref{eq:geometric1} is uniformly bounded with respect to $\varepsilon$, and pass the limit $\varepsilon\to0$ to get the results desired.

Note that $(\pa_t^j u^\varepsilon, \pa_t^j p^\varepsilon, \pa_t^j \theta^\varepsilon, \pa_t^j \eta^\varepsilon)$, $j=0, 1$ and $(\pa_t^2 u^\varepsilon, \pa_t^2 \theta^\varepsilon, \pa_t^2 \eta^\varepsilon)$ satisfy the equations
 \begin{equation}\label{eq:geometric1}
\left\{
\begin{aligned}
 &\pa_t^{j+1}u^\varepsilon+ \mu\dive_{\mathcal{A}^\varepsilon}S_{\mathcal{A}^\varepsilon}(\pa_t^jp^\varepsilon,\pa_t^ju^\varepsilon) = g \pa_t^j\theta e_2+ F^{1,j,\varepsilon}\quad&\text{in}&\ \Om,\\
 &\dive_{\mathcal{A}^\varepsilon}\pa_t^ju^\varepsilon=F^{2,j,\varepsilon}\quad&\text{in}&\ \Om,\\
 &\pa_t^{j+1}\theta^\varepsilon + k\dive_{\mathcal{A}^\varepsilon}\nabla_{\mathcal{A}^\varepsilon} \pa_t^j\theta^\varepsilon =  F^{8,j,\varepsilon}\quad&\text{in}&\ \Om,\\
 &S_{\mathcal{A}^\varepsilon}(\pa_t^jp^\varepsilon,\pa_t^ju^\varepsilon)\mathcal{N}^\varepsilon=\left[g\pa_t^j\eta^\varepsilon-\sigma_1\pa_1\left(\frac{\pa_1\pa_t^j\eta+\varepsilon\pa_t^{j+1}\eta^\varepsilon}{(1+|\pa_1\zeta_0|^2)^{3/2}}+F^{3,j,\varepsilon}\right)\right]\mathcal{N}^\varepsilon+F^{4,j,\varepsilon}\quad&\text{on}&\ \Sigma,\\
 &\pa_t^{j+1}\eta^\varepsilon-\pa_t^ju^\varepsilon\cdot\mathcal{N}^\varepsilon=F^{6,j,\varepsilon}\quad&\text{on}&\ \Sigma,\\
 &k \nabla_{\mathcal{A}^\varepsilon} \pa_t^j\theta^\varepsilon \cdot \mathcal{N}^\varepsilon + \pa_t^j\theta^\varepsilon = F^{9,j,\varepsilon}\quad&\text{on}&\ \Sigma,\\
 &(S_{\mathcal{A}^\varepsilon}(\pa_t^jp^\varepsilon,\pa_t^ju^\varepsilon)\cdot\nu-\beta\pa_t^ju^\varepsilon)\cdot\tau=F^{5,j,\varepsilon}\quad&\text{on}&\ \Sigma_s,\\
 &\pa_t^ju^\varepsilon\cdot\nu=0\quad&\text{on}&\ \Sigma_s,\\
 &\pa_t^j\theta^\varepsilon = 0\quad&\text{on}&\ \Sigma_s,\\
 &\kappa\pa_t^{j+1}\eta^\varepsilon=\mp\sigma_1 \left(\frac{\pa_1\pa_t^j\eta+\varepsilon\pa_t^{j+1}\eta^\varepsilon}{(1+|\pa_1\zeta_0|^2)^{3/2}}+F^{3,j,\varepsilon}\right)-\kappa F^{7,j,\varepsilon}\quad&\text{on}&\ \pm\ell,
\end{aligned}
\right.
 \end{equation}
 in the strong and weak sense respectively, where the forcing terms are given in Appendix \ref{sec:dive_forcing}.

 We now show the natural energy-dissipation estimate. We denote $\mathcal{A}^\varepsilon$, $\mathcal{N}^\varepsilon$, $J^\varepsilon$ and $K^\varepsilon$ are defined in terms of $\eta^\varepsilon$, in the same way as $\mathcal{A}$, $\mathcal{N}$, $J$ and $K$ are defined in terms of $\eta$.

 \begin{theorem}\label{thm:uniform}
Suppose that there exists a universal constant $\delta\in (0,1)$ such that if a solution to \eqref{eq:geometric1} exists on the time interval $[0,T)$ for $0<T\le\infty$  and obeys the estimate
$\sup_{0\le t\le T}\mathcal{E}^\varepsilon(t)\le \delta_0$,
Then there exists universal constants $C>0$ and $\lambda>0$ such that
\begin{equation}
 \begin{aligned}
\sup_{0\le t\le T}e^{\lambda t}\mathcal{E}^\varepsilon(t)+\int_0^T\mathcal{D}^\varepsilon(t)\,\mathrm{d}t\le C\mathcal{E}^\varepsilon(0).
 \end{aligned}
\end{equation}
 \end{theorem}

 \begin{proof}[Proof of Theorem \ref{thm:uniform}.]
We only give the sketch of proof for the uniform bounds of $\theta^\varepsilon$. Other parts can be established by the argument in \cite{GTWZ2023} with the modifications to surface tension.

\paragraph{\underline{Step 1 -- Energy-Dissipation Estimates for temperature}}
By the standard argument, we have
  the energy-dissipation equality for $\theta^\varepsilon$ with $j=0$ in \eqref{eq:geometric1}, and for $\pa_t\theta^\varepsilon$ with $j=1$:
 \begin{equation}\label{eq:e_d1}
\begin{aligned}
\frac{d}{dt} \frac12\int_\Om|\theta^\varepsilon|^2J^\varepsilon +\frac{k}2\int_\Om|\nabla_{\mathcal{A}^\varepsilon}\theta^\varepsilon|^2J^\varepsilon + \int_{-\ell}^\ell |\mathcal{N}^\varepsilon||\theta^\varepsilon|^2=0,
\end{aligned}
 \end{equation}
 \begin{equation}\label{eq:e_d2}
\begin{aligned}
\frac{d}{dt}\frac12\int_\Om|\pa_t\theta^\varepsilon|^2J^\varepsilon +\frac{k}2\int_\Om|\nabla_{\mathcal{A}^\varepsilon}\pa_t\theta^\varepsilon|^2J^\varepsilon + \int_{-\ell}^\ell |\mathcal{N}^\varepsilon||\pa_t\theta^\varepsilon|^2 =\mathcal{F}^{8,\varepsilon}(\pa_tu^\varepsilon, \pa_t\theta^\varepsilon, \pa_t\eta^\varepsilon),
\end{aligned}
  \end{equation}
  \begin{equation}
  \begin{aligned}
  \text{where}\ \mathcal{F}^{8,\varepsilon}(\pa_tu^\varepsilon, \pa_tp^\varepsilon, \pa_t\eta^\varepsilon)
&=\int_\Om F^{8,1,\varepsilon}\cdot\pa_t\theta^\varepsilon J^\varepsilon -\int_{-\ell}^\ell F^{9,1,\varepsilon}\cdot\pa_t\theta^\varepsilon.
\end{aligned}
 \end{equation}

 Then the nonlinear interaction estimates for \eqref{eq:e_d1} and \eqref{eq:e_d2} imply the energy-dissipation estimate for $\theta^\varepsilon$:
 \begin{equation}\label{est:e_d1}
\begin{aligned}
 &\|\theta^\varepsilon(t)\|_{H^0}^2 +\int_s^t \left(\|\theta^\varepsilon\|_{H^1}^2+\|u^\varepsilon\|_{L^2(\Sigma)}^2 \right)
 &\lesssim \|\theta^\varepsilon(s)\|_{H^0}^2,
\end{aligned}
 \end{equation}
 and the energy-dissipation estimate for $\pa_t\theta^\varepsilon$:
 \begin{equation}\label{est:e_d2}
\begin{aligned}
 &\|\pa_t\theta^\varepsilon(t)\|_{H^0}^2+ +\int_s^t \left(\|\pa_t\theta^\varepsilon\|_{H^1}^2+\|\pa_t\theta^\varepsilon\|_{L^2(\Sigma)}^2 \right)\\
 &\lesssim \|\pa_t\theta^\varepsilon(s)\|_{H^0}^2 +\int_s^t\sqrt{\mathcal{E}^\varepsilon}\mathcal{D}^\varepsilon,
\end{aligned}
 \end{equation}
  as well as the energy-dissipation estimate for $\pa_t^2\theta^\varepsilon$:
 \begin{equation}\label{est:e_d3}
\begin{aligned}
 &\|\pa_t^2\theta^\varepsilon(t)\|_{H^0}^2 - \left(\mathcal{E}^\varepsilon(t)\right)^{3/2}+\int_s^t \left(\|\pa_t^2\theta^\varepsilon\|_{H^1}^2+\|\pa_t^2\theta^\varepsilon\|_{L^2(\Sigma)}^2 \right)\\
 &\lesssim \|\pa_t^2\theta^\varepsilon(s)\|_{H^0}^2 +(\mathcal{E}^\varepsilon(s))^{3/2}+\int_s^t\sqrt{\mathcal{E}^\varepsilon}\mathcal{D}^\varepsilon
\end{aligned}
 \end{equation}
 for all $0\le s\le t\le T$.  We set
 $
\mathcal{E}_\shortparallel^\varepsilon:=\sum_{j=0}^2\|\pa_t^j\theta^\varepsilon\|_{L^2}^2$,
 and
 $
\mathcal{D}_\shortparallel^\varepsilon:=\sum_{j=0}^2\|\pa_t^j\theta^\varepsilon\|_{H^1}^2+\|\pa_t^j\theta^\varepsilon\|_{L^2(\Sigma)}^2.$
 Then the inequalities \eqref{est:e_d1}--\eqref{est:e_d3} produce the estimate
 \begin{equation}\label{est:energy_dissipation1}
\mathcal{E}_\shortparallel^\varepsilon(t)-(\mathcal{E}^\varepsilon(t))^{3/2}+\int_s^t \mathcal{D}_\shortparallel^\varepsilon \lesssim\mathcal{E}_\shortparallel^\varepsilon(s)+(\mathcal{E}^\varepsilon(s))^{3/2}+\int_s^t\sqrt{\mathcal{E}^\varepsilon}\mathcal{D}^\varepsilon
 \end{equation}
 for all $0\le s\le t\le T$.

\paragraph{\underline{Step 2 -- Enhanced Dissipation and Energy Estimates}}
The elliptic theory for $\theta^\varepsilon$ in \eqref{eq:geometric1}, with $j=0, 1$, tells us that
 \begin{equation}
\begin{aligned}
 \int_s^t \|\theta^\varepsilon\|_{W^{2,q_+}}^2
 \lesssim \int_s^t\|\pa_t\theta^\varepsilon\|_{H^1}^2+\sqrt{\mathcal{E}^\varepsilon}\mathcal{D}^\varepsilon,
\end{aligned}
 \end{equation}
 and
 \begin{equation}
\begin{aligned}
 \int_s^t \|\pa_t\theta^\varepsilon\|_{W^{2,q_-}}^2\lesssim \int_s^t\|\pa_t^2\theta^\varepsilon\|_{H^1}^2+ \sqrt{\mathcal{E}^\varepsilon}\mathcal{D}^\varepsilon,
\end{aligned}
 \end{equation}
 for all $0\le s\le t\le T$.

 Consequently, we can  get the estimate
  \begin{equation}\label{est:energy_dissipation3}
 \begin{aligned}
\mathcal{E}_\shortparallel^\varepsilon(t)-(\mathcal{E}^\varepsilon(t))^{3/2} +\int_s^t \mathcal{D}_\shortparallel^\varepsilon + \|\theta^\varepsilon\|_{W^{2,q_+}}^2 +\|\pa_t\theta^\varepsilon\|_{W^{2,q_-}}^2 \\
\lesssim\mathcal{E}_\shortparallel^\varepsilon(s) +(\mathcal{E}^\varepsilon(s))^{3/2}+\int_s^t\sqrt{\mathcal{E}^\varepsilon}\mathcal{D}^\varepsilon.
 \end{aligned}
 \end{equation}

The interpolation theory allows us to enhance the energy:
 \begin{equation}
\begin{aligned}
 &\|\pa_t\theta^\varepsilon(t)\|_{H^{1+\varepsilon_-/2}}^2\lesssim\|\pa_tu^\varepsilon(s)\|_{H^{1+\varepsilon_-/2}}^2+\int_s^t\mathcal{D}^\varepsilon.
\end{aligned}
 \end{equation}

  The elliptic theory for the \eqref{eq:geometric1} with $j=0$ produces the bounded estimate:
 \begin{equation}
\begin{aligned}
 \|\theta^\varepsilon\|_{W^{2,q_+}}^2\lesssim \|\pa_t\theta^\varepsilon\|_{L^2}^2 +\mathcal{E}^\varepsilon.
\end{aligned}
 \end{equation}

 Thus, with the estimate in \eqref{est:energy_dissipation3}, we have
 \begin{equation}\label{est:energy_dissipation4}
 \begin{aligned}
&\mathcal{E}_\shortparallel^\varepsilon(t) +\|\theta^\varepsilon(t)\|_{W^{2,q_+}}^2
-(\mathcal{E}^\varepsilon(t))^{3/2} +\int_s^t (\mathcal{D}_\shortparallel^\varepsilon +\|\theta^\varepsilon\|_{W^{2,q_+}}^2+ \|\pa_t\theta^\varepsilon\|_{W^{2,q_-}}^2)\\
&\lesssim\mathcal{E}^\varepsilon(s)+(\mathcal{E}^\varepsilon(s))^{3/2}+\int_s^t\sqrt{\mathcal{E}^\varepsilon}\mathcal{D}^\varepsilon.
 \end{aligned}
 \end{equation}

Finally, with the argument in \cite{GTWZ2023}, we combine \eqref{est:energy_dissipation4} with the uniform bounds for $(u^\varepsilon, p^\varepsilon, \eta^\varepsilon)$, to deduce that
\begin{equation}\label{est:enhance_energy6}
  \begin{aligned}
\sup_{0\le t\le t^\prime} \left(\mathcal{E}^\varepsilon(t)-(\mathcal{E}^\varepsilon(t))^{3/2}\right) +\int_0^{t^\prime}\mathcal{D}^\varepsilon
\lesssim\mathcal{E}^\varepsilon(0)+(\mathcal{E}^\varepsilon(0))^{3/2}+\int_0^{t^\prime}\sqrt{\mathcal{E}^\varepsilon}\mathcal{D}^\varepsilon,
  \end{aligned}
\end{equation}
for any $t^\prime\le T$.

\paragraph{\underline{Step 3 -- Synthesis}}
We take $s=0$ and restrict $\delta$ to be smaller if necessary, so that  if
\begin{align}
\sup_{0\le t\le T}\mathcal{E}^\varepsilon(t)\le\delta \ll 1,
\end{align}
then we can absorb the power terms of $\mathcal{E}$ in \eqref{est:enhance_energy6} to obtain the estimate
\begin{equation}
  \sup_{0\le t\le t^\prime}\mathcal{E}^\varepsilon(t)+\int_0^{t^\prime}\mathcal{D}^\varepsilon(s)\,\mathrm{d}s\lesssim \mathcal{E}^\varepsilon(0),
\end{equation}
for each $t^\prime\le T$, so that
\begin{equation}\label{est:energy5}
  \mathcal{E}^\varepsilon(t)+\int_0^t\mathcal{D}^\varepsilon(s)\,\mathrm{d}s\lesssim \mathcal{E}^\varepsilon(0)
\end{equation}
for any $t\le T$

Since $\mathcal{E}^\varepsilon\lesssim \mathcal{D}^\varepsilon$, $\mathcal{E}^\varepsilon$ is integrable over $(0,T)$. The Gronwall's lemma guarantees that there exists a universal constant $\lambda>0$ such that
$
  e^{\lambda t}\mathcal{E}^\varepsilon(t)\lesssim\mathcal{E}^\varepsilon(0)$,
for all $t<T$. Then the inequality \eqref{est:energy5} also implies
\begin{equation}
  \int_0^T\mathcal{D}^\varepsilon\lesssim\mathcal{E}^\varepsilon(0).
\end{equation}
 \end{proof}

\subsection{Global well-posedness of nonlinear contact points problem}
In this section, we give the global existence of solutions to the original nonlinear system \eqref{eq:geometric} when $\varepsilon=0$. We first prescribe the initial conditions for \eqref{eq:geometric} with $\varepsilon=0$, then give the global results for the $\varepsilon$-regularized nonlinear system \eqref{eq:geometric} and finally passing to the limit $\varepsilon\to0$ to complete the proof.

The initial data $(u_0, p_0, \theta_0, \eta_0, \pa_tu(0), \pa_tp(0), \pa_t\theta(0), \pa_t\eta(0), \pa_t^2u(0), \pa_t^2\theta(0), \pa_t^2\eta(0))$ for \eqref{eq:geometric} are assumed in the space of $\mathcal{E}$.
In addition, we assume that the compatibility conditions \eqref{compat_C2} that the initial data satisfied and that $\pa_t^k\eta(0)$, $k=0, 1, 2$, satisfying the zero average condition \eqref{cond:zero}.

Assume the initial data $(u_0^\varepsilon, p_0^\varepsilon, \theta_0^\varepsilon, \eta_0^\varepsilon, \pa_tu^\varepsilon(0), \pa_tp^\varepsilon(0), \pa_t\theta^\varepsilon(0), \pa_t\eta^\varepsilon(0), \pa_t^2u^\varepsilon(0), \pa_t^2\theta^\varepsilon(0), \pa_t^2\eta^\varepsilon(0))$ for the $\varepsilon$- regularized nonlinear system \eqref{eq:epsilon} satisfy the compatibility conditions \eqref{compat_C2}, zero average conditions \eqref{cond:zero} and $\mathcal{E}^\varepsilon(0)<\delta$ for a universal constant $\delta>0$ sufficiently small and $0<\varepsilon<1$. Then the local well-posedness in Theorem \ref{thm: fixed point} and the uniform bounds in Theorem \ref{thm:uniform} allow us to apply the standard continuity argument to arrive at the global well-posedness for the nonlinear system \eqref{eq:epsilon} with $(D_t^ju, \pa_t^jp, \pa_t^j\theta, \pa_t^j\eta)$ instead of $(D_t^ju, \pa_t^jp, \pa_t^j\theta, \pa_t^j\xi)$.

\begin{theorem}
  Suppose that the initial data $(u_0, p_0, \theta_0, \eta_0, \pa_tu(0), \pa_tp(0), \pa_t\theta(0), \pa_t\eta(0), \pa_t^2u(0), \\ \pa_t^2\theta(0), \pa_t^2\eta(0))$ for \eqref{eq:epsilon} as well as compatibility conditions \eqref{compat_C2}
as well as zero average conditions \eqref{cond:zero}. Then there exists a unique solution $(u^\varepsilon, p^\varepsilon, \theta^\varepsilon, \eta^\varepsilon)$ to \eqref{eq:epsilon} global in time. Moreover, $(u^\varepsilon, p^\varepsilon, \theta^\varepsilon, \eta^\varepsilon)$ obeys
  \begin{equation}
 \sup_{t>0}e^{\lambda t}\mathcal{E}^\varepsilon(t)+\int_0^\infty\mathcal{D}^\varepsilon(t)\,\mathrm{d}t\le C\mathcal{E}^\varepsilon(0),
  \end{equation}
  for universal constants $C>0$ and $\lambda>0$, where $\mathcal{E}^\varepsilon$ and $\mathcal{D}^\varepsilon$ are defined in \eqref{def:ep_energy} and \eqref{def:ep_dissipation} respectively.
\end{theorem}

Consequently, we can state the global result for the original nonlinear contact points system \eqref{eq:geometric} at $\varepsilon=0$.

\begin{theorem}\label{thm:conditioal_global}
  Assume that if as $\varepsilon\to 0$, $\mathcal{E}^\varepsilon(0)\to \mathcal{E}(0)$ or specifically,
  \begin{align}
  \begin{aligned}
  &(u_0^\varepsilon, p_0^\varepsilon, \theta_0^\varepsilon, \eta_0^\varepsilon, \pa_tu^\varepsilon(0), \pa_tp^\varepsilon(0), \pa_t\theta^\varepsilon(0), \pa_t\eta^\varepsilon(0), \pa_t^2u^\varepsilon(0), \pa_t^2\theta^\varepsilon(0), \pa_t^2\eta^\varepsilon(0), \varepsilon \pa_t\eta^\varepsilon(0), \varepsilon \pa_t^2\eta^\varepsilon(0),\\ &\varepsilon^{1/2} \pa_t^2\eta^\varepsilon(0))
  \to (u_0, p_0, \theta_0, \eta_0, \pa_tu(0), \pa_tp(0),\pa_t\theta(0) \pa_t\eta(0), \pa_t^2u(0), \pa_t^2\theta(0), \pa_t^2\eta(0),0,0,0),\quad \text{as}\ \varepsilon\to0,
  \end{aligned}
  \end{align}
  in the space
  \begin{align}
  \begin{aligned}
  W^{2,q_+}(\Om)\times W^{1,q_+}(\Om)\times W^{2,q_+}(\Om)\times W^{3-1/q_+,q_+}(\Sigma)\times H^{1+\varepsilon_-/2}(\Om)\times H^0(\Om)\times H^{1+\varepsilon_-/2}(\Om)\\
  \times H^{3/2+(\varepsilon_--\alpha)/2}(\Sigma) \times H^0(\Om)
  \times H^0(\Om)\times H^1(\Sigma) \times W^{3-1/q_+,q_+}(\Sigma)\times W^{3-1/q_+,q_+}(\Sigma) \times H^{3/2-\alpha}(\Sigma).
  \end{aligned}
  \end{align}

  Then $(u^\varepsilon, p^\varepsilon, \theta^\varepsilon, \eta^\varepsilon)$ converges to the limit $(u, p, \theta, \eta)$ that is the unique solution to the original nonlinear system \eqref{eq:geometric} at $\varepsilon=0$ with initial data stated in Theorem \ref{thm:main}. The limit $(u, p, \theta, \eta)$ obeys the estimate
  \begin{equation}
 \sup_{t>0}e^{\lambda t}\mathcal{E}(t)+\int_0^\infty\mathcal{D}(t)\,\mathrm{d}t\le C\mathcal{E}(0),
  \end{equation}
  for universal constants $C>0$ and $\lambda>0$, where $\mathcal{E}$ and $\mathcal{D}$ are defined in \eqref{energy} and \eqref{dissipation} respectively.
\end{theorem}

We note that Theorem \ref{thm:conditioal_global} is conditional. In order to let Theorem \ref{thm:conditioal_global} make sense, we have to prove that there really exist the initial data satisfying the limiting conditions in Theorem \ref{thm:conditioal_global}. Thanks to Theorem \ref{thm:initial}, the one more term $\pa_t^2\eta(0)\in W^{2-1/q_+,q_+}(\Sigma)$ beyond the initial energy $\mathcal{E}(0)$ can guarantee that these initial data do exist and obey the conditions in Appendix \ref{sec:initial_nonlinear}. Consequently, it is sufficient to complete the global well-posedness in Theorem \ref{thm:main}.

\begin{appendix}

\makeatletter
\renewcommand \theequation {%
A.%
\@arabic\c@equation} \@addtoreset{equation}{section}
\makeatother

\section{Forcing Terms in $\dive_\mathcal{A}$--Free System}\label{sec:dive_forcing}

We now present the nonlinear interaction terms in \eqref{eq:linear_heat} and \eqref{eq:modified_linear}, $j=0, 1, 2$. From \eqref{eq:geometric}, we directly set
\begin{equation}\label{eq:force0}
 F^8=\pa_t\bar{\eta}WK\pa_2\theta-u\cdot\nabla_{\mathcal{A}}\theta,\quad F^9=0.
\end{equation}

\[
 \begin{aligned}
  &F^1=\pa_t\bar{\eta}WK\pa_2u-u\cdot\nabla_{\mathcal{A}}u - g\theta e_2,\quad F^3=\sigma_1\mathcal{R}(\pa_1\zeta_0,\pa_1\eta),\quad F^4= \sigma_2 \theta \pa_1\left(\frac{\pa_1\eta+\varepsilon\pa_1\pa_t\eta}{(1+|\pa_1\eta|)^{1/2}}\right)\mathcal{N},\\
  & F^5=0,\quad F^7=\kappa\hat{\mathscr{W}}(\pa_t\eta), \quad F^8=\pa_t\bar{\eta}WK\pa_2\theta-u\cdot\nabla_{\mathcal{A}}\theta,\quad F^9=0.
  \end{aligned}
\]
Then for $j=1$, the forcing terms are given via
\begin{align}\label{eq:force1}
\begin{aligned}
  F^{8,1}&=\pa_t F^8 +k\dive_{\pa_t\mathcal{A}}\nabla_{\mathcal{A}}\theta+ k \dive_{\mathcal{A}}\nabla_{\pa_t\mathcal{A}}\theta,\\
  F^{9,1}&=-k\nabla_{\pa_t\mathcal{A}}\theta \cdot \mathcal{N}-k\nabla_{\mathcal{A}}\theta \cdot \pa_t\mathcal{N} - \theta\pa_t|\mathcal{N}|.
  \end{aligned}
\end{align}
\[
\begin{aligned}
  F^{1,1}&=\pa_t F^1-\pa_tRu-RD_tu-R^2u-\dive_{\pa_t\mathcal{A}}S_{\mathcal{A}}(p,u)+\mu\dive_{\mathcal{A}}\mathbb{D}_{\pa_t\mathcal{A}}u+\mu\dive_{\mathcal{A}}\mathbb{D}_{\mathcal{A}}(Ru),\\
  F^{3,1}&=\sigma_1 \pa_t \mathcal{R}(\pa_1\zeta_0,\pa_1\eta),\\
  F^{4,1}&=\pa_t F^4 +\left[g(\eta+\varepsilon\pa_t\eta)-\sigma_1\pa_1\left(\frac{\pa_1\eta+\varepsilon\pa_1\pa_t\eta}{(1+|\pa_1\zeta_0|)^{3/2}}+F^3\right)\right]\pa_t\mathcal{N}-S_{\mathcal{A}}(p,u)\pa_t\mathcal{N}\\
&\quad+\mu\mathbb{D}_{\pa_t\mathcal{A}}u\mathcal{N}+\mu\mathbb{D}_{\mathcal{A}}(Ru)\mathcal{N},\\
  F^{5,1}&=\mu\mathbb{D}_{\pa_t\mathcal{A}}u\nu\cdot\tau+\mu\mathbb{D}_{\mathcal{A}}(Ru)\nu\cdot\tau, \\
  F^{7,1}&=\kappa\hat{\mathscr{W}}^\prime(\pa_t\eta)\pa_t^2\eta.
  \end{aligned}
\]

The forcing terms for $j=2$ are given via
\begin{align}\label{eq:force2}
\begin{aligned}
  F^{8,2}&=\pa_t F^{8,1} +k\dive_{\pa_t\mathcal{A}}\nabla_{\mathcal{A}}\pa_t\theta +k\dive_{\mathcal{A}}\nabla_{\pa_t\mathcal{A}}\pa_t\theta,\\
  F^{9,2}&=\pa_t F^{9,1}-k\nabla_{\mathcal{A}}\pa_t\theta \cdot \pa_t\mathcal{N} - k\nabla_{\pa_t\mathcal{A}}\pa_t\theta \cdot \mathcal{N} - \pa_t\theta\pa_t|\mathcal{N}|.
  \end{aligned}
\end{align}
\[
\begin{aligned}
  F^{1,2}&=\pa_t F^{1,1}-\pa_tRD_tu-RD_t^2u-R^2D_tu-\dive_{\pa_t\mathcal{A}}S_{\mathcal{A}}(\pa_tp,D_tu)+\mu\dive_{\mathcal{A}}\mathbb{D}_{\pa_t\mathcal{A}}D_tu\\
  &\quad+\mu\dive_{\mathcal{A}}\mathbb{D}_{\mathcal{A}}(RD_tu),\\
  F^{3,2}&=\sigma \pa_t^2 \mathcal{R}(\pa_1\zeta_0,\pa_1\eta),\\
  F^{4,2}&=\pa_t F^{4,1}+\left[g(\pa_t\eta+\varepsilon\pa_t^2\eta)-\sigma_1\pa_1\left(\frac{\pa_1\pa_t\eta+\varepsilon\pa_1\pa_t^2\eta}{(1+|\pa_1\zeta_0|)^{3/2}}+\pa_tF^3\right)\right]\pa_t\mathcal{N}\\
&\quad-S_{\mathcal{A}}(\pa_tp,D_tu)\pa_t\mathcal{N}+\mu\mathbb{D}_{\pa_t\mathcal{A}}D_tu\mathcal{N}+\mu\mathbb{D}_{\mathcal{A}}(RD_tu)\mathcal{N},\\
  F^{5,2}&=\pa_t F^{5,1}+\mu\mathbb{D}_{\pa_t\mathcal{A}}D_tu\nu\cdot\tau+\mu\mathbb{D}_{\mathcal{A}}(RD_tu)\nu\cdot\tau,\\
  F^{7,2}&=\kappa\hat{\mathscr{W}}^\prime(\pa_t\eta)\pa_t^3\eta+\kappa\hat{\mathscr{W}}''(\pa_t\eta)(\pa_t^2\eta)^2.
  \end{aligned}
\]

Then we present the forcing terms in \eqref{eq:geometric1}. For simplicity, we can directly denote $F^{i, j, \varepsilon}=F^{i,j}(u^\varepsilon, p^\varepsilon, \theta^\varepsilon, \eta^\varepsilon)$, $i=1, 3, 4, 5, 7, 8, 9$ and $j=0, 1, 2$. That means $F^{i, j, \varepsilon}$ is actually $F^{i,j}$ with $(u, p, \theta, \eta)$ replaced by $(u^\varepsilon, p^\varepsilon, \theta^\varepsilon, \eta^\varepsilon)$.

\makeatletter
\renewcommand \theequation {%
B.%
\@arabic\c@equation} \@addtoreset{equation}{section}
\makeatother

\section{Initial Data for $\varepsilon$--Regularized Nonlinear Problem}\label{sec:initial}

In the process of constructing solutions to \eqref{eq:linear_heat2} by the Galerkin method, one necessary ingredient is to ensure that the initial data of the sequence of approximate solutions converge to the initial data of PDE \eqref{eq:linear_heat2}.
Now we state our theorem for the construction of initial conditions for nonlinear system \eqref{eq:geometric}.
 \begin{theorem}\label{thm:initial}
Suppose that the initial data $\pa_t^2\eta(0)\in \mathring{H}^1(\Sigma)\cap W^{2-1/q_+, q_+}(\Sigma)$, $\pa_t^2\theta(0)\in H^0(\Om)$ and $ D_t^2u(0)\in H^0(\Om)$ are given and satisfy $\dive_{\mathcal{A}_0}D_t^2u(0)=0$. Then we can  construct $(u_0, p_0, \theta_0, \eta_0)$ and $(\pa_tu(0), \pa_tp(0), \pa_t\theta(0), \pa_t\eta(0))$
solving the system \eqref{eq:geometric} at $t=0$ such that
\begin{align}
\mathcal{E}^\varepsilon(0) + \|\pa_t^2\eta^\varepsilon(0)\|_{W^{2-1/q_+, q_+}}^2\le \delta
\end{align}
for a universal constant $\delta>0$ sufficiently small provided that
\begin{align}
\|D_t^2u(0)\|_{H^0}^2+\|\pa_t^2\theta(0)\|_{H^0}^2+\|\pa_t^2\eta(0)\|_{H^1}^2+\|\pa_t^2\eta(0)\|_{W^{2-1/q_+, q_+}}^2\le\delta_0
\end{align}
for a universal constant $\delta_0>0$ sufficiently small.
 \end{theorem}

 \begin{remark}
It is important to note that the initial data in Theorem \ref{thm:initial} implies the existence of $\mathcal{E}^\varepsilon(0)$ with $\mathcal{E}^\varepsilon$ defined in \eqref{def:ep_energy}. So Theorem \ref{thm:initial} guarantees that our initial data for the nonlinear system \eqref{eq:geometric} are well defined.
 \end{remark}

\subsection{Initial Data for Linear heat PDEs}\label{sec:initial_linear}

Our aim is to use the Galerkin method to give the existence and uniqueness of solutions to \eqref{eq:linear_heat}. So, the important thing is to ensure that the initial data of the sequence of approximate solutions converge to the initial data of PDE \eqref{eq:linear_heat}. We assume
the forcing terms
\begin{equation}
F^8(0),\ \pa_tF^8(0)\in L^{q_+}(\Om), \quad F^9(0),\ \pa_tF^9(0)\in W^{1-1/q_+,q_+}(\Sigma)
\end{equation}
 and $\mathcal{A}_0$, $\mathcal{N}(0)$, $J(0)=1/K(0)$ are given in terms of $\eta_0$.
 The construction of the initial data are presented in Lemma \ref{lem:initial_dt}.

We first prescribe the initial data of higher order $\pa_t^2\theta(0)$, then construct the initial data of lower order $\pa_t\theta(0)$ and $\theta_0$. This is a backward procedure for time regularity.  $\pa_t^2\theta(0)\in H^0(\Om)$ is natural for the
dynamical energy of the linear and nonlinear system for $t>0$.

We now use elliptic theory to construct $\pa_t\theta(0)$. First, it is necessary to
 assume the initial data $F^{8,1}(0)$ and $F^{9,1}(0)$  are defined as
\begin{equation}\label{def:F_11}
\begin{aligned}
&F^{8,1}:=\pa_tF^8+\mathfrak{G}^8(\theta), \quad F^{9,1}:=\pa_tF^9+\mathfrak{G}^9(\theta),
\end{aligned}
\end{equation}
where
\begin{equation}\label{def:map_g}
\begin{aligned}
\mathfrak{G}^8(\theta)=k\dive_{\mathcal{A}}\nabla_{\pa_t\mathcal{A}}\theta+ k \dive_{\pa_t\mathcal{A}}\mathbb{D}_{\mathcal{A}}\theta,\
\mathfrak{G}^9(\theta)=-k\nabla_{\pa_t\mathcal{A}}\theta \cdot \mathcal{N} -k\nabla_{\mathcal{A}}\theta \cdot \pa_t\mathcal{N} - \theta \pa_t|\mathcal{N}|.
\end{aligned}
\end{equation}

We define
\begin{equation}\label{def:forcej2}
  F^{8,2}:=\mathfrak{G}^8(\pa_t\theta)+\pa_tF^{8,1},\
  F^{9,2}:=\mathfrak{G}^9(\pa_t\theta)+\pa_tF^{9,1}.
\end{equation}

Then we have the following lemma.
\begin{lemma}\label{lem:initial_force1}
  $F^{8,1}(0)\in L^{q_+}(\Om),\
F^{9,1}(0)\in W^{1-1/q_+, q_+}(\Sigma)$T are implied by the following estimates
\begin{equation}\label{est:initial_f11}
\|F^{8,1}(0)\|_{L^{q_+}(\Om)}\lesssim \|\pa_tF^8(0)\|_{L^{q_+}}+\|\pa_t\eta(0)\|_{H^{3/2+(\varepsilon_--\alpha)/2}}\|\theta_0\|_{W^{2,q_+}},
\end{equation}
\begin{equation}\label{est:initial_f51}
  \begin{aligned}
\|F^{9,1}(0)\|_{W^{1-1/q_+, q_+}(\Sigma)}\lesssim\|\pa_tF^9(0)\|_{W^{1-1/q_+, q_+}(\Sigma)}+\|\pa_t\eta(0)\|_{H^{3/2+(\varepsilon_--\alpha)/2}}\|\theta_0\|_{W^{2,q_+}}.
  \end{aligned}
\end{equation}

\end{lemma}
\begin{proof}
  We estimate $\mathfrak{G}^8(\theta_0)$ term by term. By $\|\mathcal{A}_0\|_{L^\infty(\Om)}\lesssim1$, we have
  \begin{equation}\label{est:initial_g11}
\begin{aligned}
  &\|\pa_tJ(0)K(0)\Delta_{\mathcal{A}_0}\theta_0\|_{L^{q_+}}\\
  &\lesssim\|\pa_tJ(0)K(0)\|_{L^\infty(\Om)}(\|\nabla^2\theta_0\|_{L^{q_+}}+\|\nabla\mathcal{A}_0\|_{L^2(\Om)}\|\nabla \theta_0\|_{L^{2/(1-\varepsilon_+)}})\\
  &\lesssim \|\pa_t\eta(0)\|_{H^{3/2+(\varepsilon_--\alpha)/2}}(1+\|\eta_0\|_{W^{3-1/q_+,q_+}})\|\theta_0\|_{W^{2,q_+}}.
\end{aligned}
  \end{equation}
  Similarly, we have
  \begin{equation}\label{est:initial_g13}
\begin{aligned}
  &\|\dive_{\mathcal{A}_0}\nabla_{\pa_t\mathcal{A}(0)}\theta_0 + \dive_{\pa_t\mathcal{A}(0)}\nabla_{\mathcal{A}_0}\theta_0\|_{L^{q_+}}\\
  &\lesssim \|\pa_t\mathcal{A}(0)\|_{L^\infty(\Om)}\|\nabla^2\theta_0\|_{L^{q_+}}+\|\nabla\pa_t\mathcal{A}(0)\|_{L^2(\Om)}\|\nabla \theta_0\|_{L^{2/(1-\varepsilon_+)}}\\
  &\quad+\|\pa_t\mathcal{A}(0)\|_{L^\infty(\Om)}\|\nabla^2 \theta_0\|_{L^{q_+}}+\|\pa_t\mathcal{A}(0)\|_{L^\infty(\Om)}\|\nabla\mathcal{A}_0\|_{L^2(\Om)}\|\nabla \theta_0\|_{L^{2/(1-\varepsilon_+)}}\\
  &\lesssim\|\pa_t\eta(0)\|_{H^{3/2+(\varepsilon_--\alpha)/2}}(1+\|\eta_0\|_{W^{3-1/q_+,q_+}})\|\theta_0\|_{W^{2,q_+}}.
\end{aligned}
  \end{equation}
  The combination of \eqref{est:initial_g11}--\eqref{est:initial_g13} gives the bound of \eqref{est:initial_f11}.

  We turn to the bounds for $\mathfrak{G}^9(\theta_0)$. We first employ Sobolev embedding theory and usual trace theory to deduce that
  \begin{equation}\label{est:initial_g41}
  \begin{aligned}
\|\nabla_{\pa_t\mathcal{A}(0)}\theta_0 \cdot\mathcal{N}(0)\|_{W^{1-1/q_+, q_+}(\Sigma)}&\lesssim \|\nabla_{\pa_t\mathcal{A}(0)}\theta_0\|_{W^{1, q_+}(\Om)}\\
&\lesssim \|\nabla_{\pa_t\mathcal{A}(0)}\theta_0\|_{L^{q_+}}+\|\pa_t\mathcal{A}(0)\nabla\theta_0)\|_{L^{q_+}}\\
&\quad+\|\pa_t\mathcal{A}(0)\nabla^2 \theta_0)\|_{L^{q_+}}\\
&\lesssim \|\pa_t\eta(0)\|_{H^{3/2+(\varepsilon_--\alpha)/2}}\|\theta_0\|_{W^{2,q_+}}.
  \end{aligned}
  \end{equation}
  The similar argument allows us to deduce
  \begin{equation}\label{est:initial_g42}
  \begin{aligned}
\|\nabla_{\mathcal{A}_0}\theta_0\pa_t\mathcal{N}(0)\|_{W^{1-1/q_+, q_+}(\Sigma)}
&\lesssim \|\nabla_{\mathcal{A}_0}\theta_0\|_{W^{1, q_+}(\Om)}(\|\pa_t\mathcal{N}(0)\|_{L^\infty} +\|\pa_t\mathcal{N}(0)\|_{W^{1, q_+}(\Sigma)})\\
&\lesssim \|\nabla \theta_0\|_{W^{1, q_+}}+\||\nabla\mathcal{A}_0|\nabla \theta_0\|_{L^{q_+}})\|\pa_t\eta(0)\|_{H^{3/2+(\varepsilon_--\alpha)/2}}\\
&\lesssim \|\pa_t\eta(0)\|_{H^{3/2+(\varepsilon_--\alpha)/2}}\|\theta_0\|_{W^{2,q_+}},
  \end{aligned}
  \end{equation}
  and
  \begin{equation}\label{est:initial_g43}
\|\theta_0\pa_t|\mathcal{N}(0)|\|_{W^{1-1/q_+, q_+}(\Sigma)}\lesssim \|\pa_t\eta(0)\|_{H^{3/2+(\varepsilon_--\alpha)/2}}\|\theta_0\|_{W^{2,q_+}}.
  \end{equation}

  The estimates \eqref{est:initial_g41}--\eqref{est:initial_g43} along with the assumption $\pa_tF^9(0)\in W^{1-1/q_+, q_+}(\Sigma)$ imply \eqref{est:initial_f51}.

\end{proof}

Under the assumption on $\pa_t^2\theta(0)$ and Lemma \ref{lem:initial_force1}, we wish to construct $\pa_t\theta(0)$. However, from Lemma \ref{lem:initial_force1}, we know that the forcing terms depend on $\theta_0$, whichh is yet unknown. So we need an iteration as follows.
\begin{align}\label{eq:iteration_initial1}
  \left\{
  \begin{aligned}
    & - k\Delta_{\mathcal{A}_0} \pa_t\theta^{n+1} (0) = \pa_t^2\theta(0) + F^{8,1}(\theta_0^n) \quad & \text{in}& \ \Om,\\
    &k\nabla_{\mathcal{A}_0}\pa_t\theta^{n+1} (0) \cdot\mathcal{N}(0) + \pa_t\theta^{n+1} (0)|\mathcal{N}(0)| = F^{9,1}(\theta_0^n) \quad & \text{on}& \ \Sigma,\\
    &\pa_t\theta^{n+1} (0) =0\quad & \text{on}& \ \Sigma_s,
  \end{aligned}
  \right.
\end{align}
and
\begin{align}\label{eq:iteration_initial2}
  \left\{
  \begin{aligned}
    & - k\Delta_{\mathcal{A}_0} \theta^{n+1} (0) = \pa_t\theta^n(0) + F^8(0) \quad & \text{in}& \ \Om,\\
    &k\nabla_{\mathcal{A}_0}\theta^{n+1} (0) \cdot\mathcal{N}(0) + \theta^{n+1} (0)|\mathcal{N}(0)| = F^9(0) \quad & \text{on}& \ \Sigma,\\
    &\theta^{n+1} (0) =0\quad & \text{on}& \ \Sigma_s.
  \end{aligned}
  \right.
\end{align}
These two coupled iterations can present the construction of both $\pa_t\theta(0)$ and $\theta_0$, which obey the elliptic estimates:
\begin{lemma}\label{lem:initial_dt}
   \begin{align}\label{est:initial_h}
   \begin{aligned}
&\|\theta_0\|_{W^{2, q_+}}^2+\|\pa_t\theta(0)\|_{W^{2, q_+}}^2 + \|\pa_t\theta(0)\|_{H^{1+\varepsilon_-/2}}^2\\
&\lesssim\|\pa_t^2\theta(0)\|_{L^2}^2+
 \|F^8(0)\|_{L^{q_+}}^2+\|F^9(0)\|_{W^{1-1/q_+,q_+}}^2+\\
 &\quad+\|\pa_tF^8(0)\|_{L^{q_+}}^2+\|\pa_tF^9(0)\|_{W^{1-1/q_+,q_+}}^2.
 \end{aligned}
  \end{align}
\end{lemma}
\begin{proof}
The elliptic theory to \eqref{eq:iteration_initial1} allows us to construct $\pa_t\theta^{n+1}(0)$ such that
  \begin{equation}\label{est:initial_3}
  \begin{aligned}
\|\pa_t\theta^{n+1}(0)\|_{W^{2, q_+}}^2 + \|\pa_t\theta^{n+1}(0)\|_{H^{1+\varepsilon_-/2}}^2
&\lesssim\|\pa_t\eta(0)\|_{H^{3/2+(\varepsilon_--\alpha)/2}}^2\|\theta_0^n\|_{W^{2,q_+}}^2+ \|\pa_t^2\theta(0)\|_{H^0}^2\\
&+\|\pa_tF^8(0)\|_{L^{q_+}}^2+\|\pa_tF^9(0)\|_{W^{1-1/q_+, q_+}}^2.
\end{aligned}
  \end{equation}
and to \eqref{eq:iteration_initial2} allows us to construct $\theta^{n+1}_n$ such that
  \begin{equation}\label{est:initial_4}
  \begin{aligned}
\|\theta^{n+1}(0)\|_{W^{2, q_+}}^2
\lesssim \|\pa_t\theta^n(0)\|_{H^0}^2
+\|F^8(0)\|_{L^{q_+}}^2+\|F^9(0)\|_{W^{1-1/q_+, q_+}}^2.
\end{aligned}
  \end{equation}
  by a standard contraction argument, we can show that $(\pa_t\theta^n(0), \theta_0^n) \in H^1(\Om) \times H^1(\Om)$ is a Cauchy sequence. Hence, by passing the limit $n\to \infty$,
$(\pa_t\theta^n(0), \theta_0^n) \times \left( W^{2, q_+}(\Om) \cap  H^{1+\varepsilon_-/2} \right) \times W^{2, q_+}(\Om)$ obeys the estimate \eqref{est:initial_h}.
\end{proof}

\subsection{Contraction and Initial Data for Nonlinear PDEs}\label{sec:initial_nonlinear}

In this section, we give the construction of the initial data for $\varepsilon$ regularized nonlinear system \eqref{eq:geometric} and \eqref{eq:geo_bc}. In this section, all the $(u, p, \theta, \eta)$  and $(v, q, \vartheta, \xi)$ should be considered with superscript $\varepsilon$. But for simplicity, we ignore $\varepsilon$.

The main result in this section is to show Theorem \ref{thm:initial} by the fixed point theory. We first assume that $(u_0, p_0, \theta_0, \eta_0)$, $(D_tu(0), \pa_tp(0), \pa_t\theta(0), \pa_t\eta(0))$, and $(D_t^2u(0), \pa_t^2\theta(0), \pa_t^2\eta(0))$ are given so that
$\mathcal{E}^\varepsilon (u,p,\theta,\eta)(0) + \|\pa_t\eta(0)\|_{W^{2-1/q_+, q_+}}^2\le \delta$ for $\delta$ sufficiently small.  Then we denote the unknowns $(v_0, q_0, \vartheta_0, \xi_0)$ in \eqref{eq:modified_linear} and \eqref{eq:linear_heat} with $j=0$ at $t=0$ instead of $(u_0, p_0, \theta_0, \xi_0)$, and the unknowns $(D_tv(0), \pa_tq(0), \pa_t\vartheta(0), \pa_t\xi(0))$ in \eqref{eq:modified_linear} and \eqref{eq:linear_heat} with $j=1$ at $t=0$ instead of $(D_tu(0), \pa_tp(0), \pa_t\theta(0), \pa_t\eta(0))$, and obtain the boundedness $\mathcal{E}^\varepsilon (v,q, \vartheta, \xi)(0) + \|\pa_t^2\xi(0)\|_{W^{2-1/q_+, q_+}}^2\le \delta$ for nonlinear initial data, subjected to the assumptions that $D_t^2v(0) = D_t^2u(0)$, $\pa_t^2\vartheta(0) = \pa_t^2\theta(0)$ and $\pa_t^2\xi(0) = \pa_t^2 \eta(0)$. Finally we use the contraction to show the existence of fixed points.

In order to use Lemma \ref{lem:initial_dt}, we give the estimate of forcing terms in \eqref{est:initial_h}.
\begin{lemma}\label{est:initial_force3}
It holds that
  \begin{equation}\label{est:initial_force1}
  \begin{aligned}
\|F^8(0)\|_{L^{q_+}}
 \lesssim(\|\pa_t\eta(0)\|_{H^{3/2+(\varepsilon_--\alpha)/2}}+\|u_0\|_{W^{2,q_+}}^2)\|\theta_0\|_{W^{2,q_+}},
\end{aligned}
  \end{equation}
  \begin{equation}\label{est:initial_force2}
  \begin{aligned}
 \|\pa_tF^8(0)\|_{L^{q_+}}
 &\lesssim(\|\pa_t^2\eta(0)\|_{H^1} + \|u_0\|_{H^1}\|\pa_t\eta(0)\|_{H^{3/2+(\varepsilon_--\alpha)/2}} + \|D_tu(0)\|_{H^1}) \|\theta_0\|_{W^{2,q_+}}\\
 &\quad + \|u_0\|_{H^1}\|\pa_t\theta(0)\|_{H^{1+\varepsilon_-/2}} .
\end{aligned}
  \end{equation}
\end{lemma}
\begin{proof}
  From the definition of $F^8(0)$, we directly use the H\"older inequality to show that
  \begin{equation}\label{est:initial_f111}
  \begin{aligned}
\|F^8(0)\|_{L^{q_+}}&\lesssim \|\pa_t\bar{\eta}(0)\pa_2\theta_0\|_{L^{q_+}}+\|u_0\cdot\nabla_{\mathcal{A}_0}\theta_0\|_{L^{q_+}}\\
&\lesssim (\|\pa_t\eta(0)\|_{H^{3/2+(\varepsilon_-\alpha)/2}}+\|u_0\|_{W^{2,q_+}})\|\theta_0\|_{W^{2,q_+}},
\end{aligned}
  \end{equation}
  where in the last inequality, we have used the fact $\|\mathcal{A}_0\|\lesssim 1$, and the Sobolev embedding theory and usual trace theory.

 We estimate $\pa_tF^8(0)$ term by term. First, by the H\"older inequality, and Sobolev embedding theory and usual trace theory, we have that
 \begin{equation}\label{est:initial_dtf111}
\begin{aligned}
 \|\pa_t^2\bar{\eta}(0)K(0)\pa_2\theta_0\|_{L^{q_+}}\lesssim \|\pa_t^2\bar{\eta}(0)\|_{L^\infty}\|\pa_2\theta_0\|_{L^{q_+}}\lesssim \|\pa_t^2\eta(0)\|_{H^1}\|\theta_0\|_{W^{2,q_+}},
\end{aligned}
 \end{equation}
 where we have used $\|K(0)\|_{L^\infty}\lesssim 1$. Similarly,
 \begin{equation}\label{est:initial_dtf112}
\begin{aligned}
 \|\pa_t\bar{\eta}(0)\pa_tK(0)\pa_2\theta_0+\pa_t\bar{\eta}(0)K(0)\pa_2\pa_t\theta(0)\|_{L^{q_+}}\lesssim \|\pa_t\eta(0)\|_{H^1}(\|\theta_0\|_{W^{2,q_+}}+\|\pa_t\theta(0)\|_{H^{1+\varepsilon_-/2}}).
\end{aligned}
 \end{equation}

Then the lower order terms are bounded by
\begin{equation}\label{est:initial_dtf114}
  \begin{aligned}
&\|D_tu(0)\nabla_{\mathcal{A}_0}\theta_0\|_{L^{q_+}}+\|u_0\nabla_{\pa_t\mathcal{A}(0)}\theta_0\|_{L^{q_+}} + \|u_0\nabla_{\mathcal{A}_0}\pa_t\theta(0)\|_{L^{q_+}}+\|(R(0)u_0)\nabla_{\mathcal{A}_0}\theta_0\|_{L^{q_+}}\\
&\lesssim (\|D_tu(0)\|_{H^1}+\|u_0\|_{H^1}\|\pa_t\eta(0)\|_{H^{3/2+(\varepsilon_--\alpha)/2}})\|\theta_0\|_{W^{2,q_+}}+\|u_0\|_{H^1}\|\pa_t\theta(0)\|_{H^{1+\varepsilon_-/2}}.
  \end{aligned}
\end{equation}

So all the above estimates from \eqref{est:initial_dtf111}--\eqref{est:initial_dtf114} show the estimate \eqref{est:initial_force2}.
\end{proof}

In order to prove the contraction, we refer to velocities as $v^i, u^i$, pressures as $q^i, p^i$, temperatures as $\vartheta^i, \theta^i$ and surface functions as $\eta^i, \xi^i$, $i=1, 2$. We suppose that $(u_0^j, p_0^j, \theta_0^j, \eta_0^j, D_tu^j(0), \pa_tp^j(0), \pa_t\eta^j(0))$ for $j=1, 2$,
as well as$
(D_t^2v^1(0), \pa_t^2\theta^1(0), \pa_t^2\eta^1(0))=(D_t^2v^2(0), \pa_t^2\theta^2(0), \pa_t^2\eta^2(0))$
 are two initial data for the nonlinear $\varepsilon-$ regularized system \eqref{eq:modified_linear} and \eqref{eq:linear_heat} with $\xi$ replaced by $\eta$. Then it is concerned with the following systems with $j=0, 1$:
 \begin{equation}\label{eq:pde_hdt1}
\left\{
\begin{aligned}
  &\pa_t^{j+1}\vartheta^i(0) - k \Delta_{\mathcal{A}^i_0} \pa_t^j\vartheta^i(0)=\mathfrak{f}^{8,i}_j(0)\quad&\text{in}&\quad\Om,\\
  & k \mathcal{N}^i(0) \cdot\nabla_{\mathcal{A}_0} \pa_t^j\vartheta^i(0) + \vartheta |\mathcal{N}^i(0)| =\mathfrak{f}^{9,i}_j(0)\quad&\text{on}&\quad\Sigma,\\
  &\pa_t^j\vartheta^i(0) = 0\quad&\text{on}&\quad\Sigma_s,
\end{aligned}
\right.
\end{equation}
where $ \mathfrak{f}^{k,i}_j(0) = \pa_t \pa_tF^{k}(\theta_0^i,\eta_0^i,\pa_t\eta^i(0))+ \mathfrak{G}^k(\theta_0^i)$, $k = 8, 9$, $i = 1, 2$, and
\begin{equation}\label{eq:pde_dt1}
  \left\{
  \begin{aligned}
  &(D_t^{j+1}v^i)(0)+\nabla_{\mathcal{A}^i_0}(\pa_t^jq^i)(0)-\mu\Delta_{\mathcal{A}^i_0}(D_t^jv^i(0)+2^jR^i(0)D_t^jv^i(0) \\ &\quad\quad\quad\quad\quad\quad\quad\quad\quad\quad\quad\quad\quad\quad\quad\quad\quad\quad\quad\quad- g\pa_t^j\theta^i(0)e_2=\mathfrak{f}^{1,i}_{j}(0)\quad&\text{in}&\quad\Om,\\
  &\dive_{\mathcal{A}^i_0}(D_t^jv^i)(0)=0\quad&\text{in}&\quad\Om,\\
  &S_{\mathcal{A}^i_0}(\pa_t^jq^i(0),D_t^jv^i(0))\mathcal{N}^i(0)=\bigg[g(\pa_t^j\xi^i(0)+\varepsilon\pa_t^{j+1}\xi^i(0))\\
  &\quad\quad\quad\quad-\sigma_1\pa_1\left(\frac{\pa_1\pa_t^j\xi^i(0)+\varepsilon\pa_1\pa_t^{j+1}\xi^i(0)}{(1+|\pa_1\zeta_0|^2)^{3/2}}\right)-\pa_1 \mathfrak{f}^{3,i}_j(0)\bigg]\mathcal{N}(0)+\mathfrak{f}^{4,i}_j(0)\quad&\text{on}&\quad\Sigma,\\
  &D_t^jv^i(0)\cdot\mathcal{N}^i(0)=\pa_t^{j+1}\xi^i(0)\quad&\text{on}&\quad \Sigma,\\
  &(S_{\mathcal{A}^i_0}(\pa_t^jq^i(0),D_t^jv^i(0))\nu-\beta D_t^jv^i(0))\cdot\tau=\mathfrak{f}^{5,i}_j(0)\quad&\text{on}&\quad\Sigma_s,\\
  &D_t^jv^i(0)\cdot\nu=0\quad&\text{on}&\quad\Sigma_s,\\
  &\mp\sigma_1\frac{\pa_1\pa_t^j\xi^i(0)+\varepsilon\pa_1\pa_t^{j+1}\xi^i(0)}{(1+|\pa_1\zeta_0|^2)^{3/2}}(\pm\ell)=\kappa(D_t^jv^i(0)\cdot\mathcal{N}^i(0))(\pm\ell)\\
  &\quad\quad\quad\quad\quad\quad\quad\quad\quad\quad\quad\quad\quad\quad\quad\quad\pm \mathfrak{f}^{3,i}_j(\pm\ell,0)-\mathfrak{f}^{7,i}_j(\pm\ell,0),
  \end{aligned}
  \right.
\end{equation}
where $\mathfrak{f}^{k,i}_0$ for $i=1, 2$, $k=1, 3, 4, 5, 9$ can be given in \cite{GTWZ2023}.  We define $\mathcal{A}^i$, $\mathcal{N}^i$, etc, in terms of $\eta^i$, $i=1, 2$. Set
\[
\begin{aligned}
  \widetilde{\mathfrak{E}}_0
  &=\|D_tu^i(0)\|_{H^{1+\varepsilon_-/2}}^2+\|\pa_tp^i(0)\|_{H^{\varepsilon_-/2}}^2+\|\pa_t\theta^i(0)\|_{H^{1+\varepsilon_-/2}}^2+\|\pa_t\xi^i(0)\|_{W^{3-1/q_+, q_+}}^2\\
  &\quad+\|u_0^i\|_{W^{2, q_+}}^2+\|p_0^i\|_{W^{1, q_+}}^2+\|\theta_0^i\|_{W^{2, q_+}}^2+\|\xi_0^i\|_{W^{3-1/q_+, q_+}}^2.
\end{aligned}
\]

We use Lemma \ref{lem:initial_dt} and \ref{est:initial_force3} to get that
\begin{lemma}
  The solutions $\pa_t^j\vartheta(0)$ solving \eqref{eq:pde_hdt1}  are bounded by
  \begin{equation}
  \|\vartheta_0^i\|_{W^{2, q_+}}^2+\|\pa_t\vartheta^i(0)\|_{H^{1+\varepsilon_-/2}}^2
  \lesssim \|\pa_t^2\theta(0)\|_{L^2}^2+\widetilde{\mathfrak{E}}_0.
  \end{equation}
\end{lemma}
Therefore, by assuming $\|D_t^2u(0)\|_{L^2}^2+\|\pa_t^2\theta(0)\|_{L^2}^2+\|\pa_t^2\eta(0)\|_{W^{3-1/q_+, q_+}}^2\lesssim \delta_0$, we have $
  \widetilde{\mathfrak{E}}_0\lesssim \delta
$
for a universal constant $\delta$ sufficiently small, $i=1, 2$.

For simplicity,  we might abuse the notation to denote
\[
  \begin{aligned}
  &u_0=u_0^1-u_0^2,\ p_0=p_0^1-p_0^2,\ \theta_0 = \theta_0^1-\theta_0^2,\ \xi_0=\xi_0^1-\xi_0^2,\ v_0=v_0^1-v_0^2,\\ &q_0=q_0^1-q_0^2,\ \vartheta_0 = \vartheta_0^1 - \vartheta_0^2,\ \eta_0=\eta_0^1-\eta_0^2,
  \end{aligned}
\]
\[
\begin{aligned}
  D_tu(0)=D_tu^1(0)-D_tu^2(0),\ \pa_tp(0)=\pa_tp^1(0)-\pa_tp^2(0),\ \pa_t\theta(0) = \pa_t\theta^1(0) - \pa_t\theta^2(0), \\ \pa_t\xi(0)=\pa_t\xi^1(0)-\pa_t\xi^2(0),\ D_tv(0)=D_tv^1(0)-D_tv^2(0),\ \pa_tq=\pa_tq^1(0)-\pa_tq^2(0),\\
   \pa_t\vartheta(0) = \pa_t\vartheta^1(0) - \pa_t\vartheta^2(0),\ \pa_t\eta(0)=\pa_t\eta^1(0)-\pa_t\eta^2(0).
\end{aligned}
\]
\begin{lemma}\label{lem:contract_1}
  $(D_tv(0), \pa_tq(0), \pa_t\vartheta(0), \pa_t\xi(0))$ obey the estimate
  \begin{equation}\label{est:initial_contract_4}
\begin{aligned}
  &\|D_tv(0)\|_{H^1}^2+\|\pa_tq(0)\|_{H^0}^2 + \|\pa_t\vartheta(0)\|_{H^1}^2+\|\pa_t\xi(0)\|_{H^{3/2}}^2\\
  &\le C\delta(\|D_tu(0)\|_{H^1}^2+\|u_0\|_{W^{2,q_+}}^2+\|\pa_t\theta(0)\|_{H^1}^2+\|\theta_0\|_{W^{2,q_+}}^2+\|\eta_0\|_{W^{3-1/q_+,q_+}}^2\\
&\quad+\|\pa_t\eta(0)\|_{H^{3/2}}^2(\|v_0\|_{W^{2,q_+}}^2 +\|\vartheta_0\|_{W^{2,q_+}}^2 )+\|q_0\|_{W^{1,q_+}}^2+\|\xi_0\|_{W^{3-1/q_+,q_+}}^2)
\end{aligned}
  \end{equation}
  and $(v_0, q_0, \theta_0, \xi_0)$  obey the estimate
\begin{equation}
  \begin{aligned}
&\|v_0\|_{W^{2,q_+}}^2+\|q_0\|_{W^{1,q_+}}^2+\|\vartheta_0\|_{W^{2,q_+}}^2+\|\xi_0\|_{W^{3-1/q_+,q_+}}^2\\
&\le C\left [ \|D_tv(0)\|_{L^2}^2 + \|\pa_t\vartheta(0)\|_{L^2}^2 +\delta(\|\eta_0\|_{W^{3-1/q_+,q_+}}^2+\|\pa_t\eta(0)\|_{H^{3/2}}^2+\|u_0\|_{W^{2,q_+}}^2+\|\theta_0\|_{W^{2,q_+}}^2) \right],
  \end{aligned}
\end{equation}
  where $C$ is a universal constant to be changed from line to line.
\end{lemma}
\begin{proof}
We first consider the temperature.
By $\pa_t^2\vartheta^1(0)=\pa_t^2\vartheta^2(0)$, $\pa_t\vartheta(0)$ satisfies the equations
\begin{equation}\label{eq:diff_initial1}
  \left\{
  \begin{aligned}
&-k\Delta_{\mathcal{A}_0^1}\pa_t\vartheta(0) =\mathfrak{g}^8_1\quad&\text{in}&\ \Om,\\
&k\nabla_{\mathcal{A}_0^1}\pa_t\vartheta(0)\cdot\mathcal{N}^1(0) + \pa_t\vartheta(0)|\mathcal{N}^1(0)|=\mathfrak{g}^9_1&\text{on}&\ \Sigma,\\
  &\pa_t\vartheta(0)=0&\text{on}&\ \Sigma_s,
  \end{aligned}
  \right.
\end{equation}
where the forcing terms $\mathfrak{g}^k_1$, $k=8, 9$, are presented in Appendix \ref{sec:initial_forcing}.

We multiply the first equation of \eqref{eq:diff_initial1} by $J^1(0)\pa_t\vartheta(0)$, and then integrate over $\Om$ and integrate by parts to deduce that
\[
  \begin{aligned}
&\frac k2\int_\Om|\nabla_{\mathcal{A}^1_0}\pa_t\vartheta(0)|^2J^1(0)+\int_{\Sigma}|\pa_t\vartheta(0)|^2|\mathcal{N}^1(0)|\\
&=\int_\Om\mathfrak{g}^8_1\cdot \pa_t\vartheta(0)J^1(0)+\int_\Sigma\mathfrak{g}^9_1\pa_t\vartheta(0).
  \end{aligned}
\]
 By the smallness of $\delta_0$ and integration by parts, we employ the same arguments as in the proof of Lemma \ref{lem:initial_dt} to obtain the bounds
 \[
  \begin{aligned}
&\left|\int_\Om\mathfrak{g}^8_1\cdot \pa_t\vartheta(0)J^1(0)+\int_\Sigma\mathfrak{g}^9_1\pa_t\vartheta(0)\right|\\
&\lesssim \delta^{1/2}(\|\pa_t\eta(0)\|_{H^{3/2-\alpha}}+\|\eta_0\|_{W^{3-1/q_+,q_+}}+\|D_tu(0)\|_{H^1}+\|u_0\|_{W^{2,q_+}}\\
&\quad +\|\theta_0\|_{W^{2,q_+}}+\|\vartheta_0\|_{W^{2,q_+}})\|\pa_t\vartheta(0)\|_{H^1}.
  \end{aligned}
\]
By the  Cauchy inequality, we have the bounds
\[
  \begin{aligned}
\|\pa_t\vartheta(0)\|_{H^1}^2&\le C\delta(\|D_tu(0)\|_{H^1}^2+\|u_0\|_{W^{2,q_+}}^2+\|\theta_0\|_{W^{2,q_+}}^2+\|\eta_0\|_{W^{3-1/q_+,q_+}}^2+\|\pa_t\eta(0)\|_{H^{3/2-\alpha}}^2)\\
&\quad+C\delta\|\vartheta_0\|_{W^{2,q_+}}^2.
  \end{aligned}
\]

$\vartheta_0$ satisfies the equations
\begin{equation}\label{eq:diff_initial2}
  \left\{
  \begin{aligned}
&-k\Delta_{\mathcal{A}_0^1}\vartheta_0 =\mathfrak{g}^8\quad&\text{in}&\ \Om,\\
&k\nabla_{\mathcal{A}_0^1}\vartheta_0\cdot\mathcal{N}^1(0) + \vartheta_0|\mathcal{N}^1(0)|=\mathfrak{g}^9&\text{on}&\ \Sigma,\\
  &\vartheta_0=0&\text{on}&\ \Sigma_s.
  \end{aligned}
  \right.
\end{equation}
We directly use the elliptic estimate to get that
\[
  \begin{aligned}
    \|\vartheta_0\|_{W^{2,q_+}}^2\lesssim \|\pa_t\vartheta(0)\|_{L^2}^2+\delta(\|\eta_0\|_{W^{3-1/q_+,q_+}}^2+\|\pa_t\eta(0)\|_{H^{3/2-\alpha}}^2+\|u_0\|_{W^{2,q_+}}^2+\|\theta_0\|_{W^{2,q_+}}^2 ).
  \end{aligned}
\]

Then we use the similar argument as in \cite{GTWZ2023} to estimate the velocity $(v_0, D_tv(0))$, pressure $q_0$ and surface functions $(\xi_0, \pa_t\xi(0))$ to complete the proof.

\end{proof}

The Lemma \ref{lem:contract_1} gives the initial data for $\varepsilon$- nonlinear system \eqref{eq:geometric}.
\begin{proof}[Proof of Theorem \ref{thm:initial}.]
  The results in Lemma \ref{lem:contract_1} imply the estimate
  \[
  \begin{aligned}
&\|v_0\|_{W^{2,q_+}}^2+\|q_0\|_{W^{1,q_+}}^2+\|\vartheta_0\|_{W^{2,q_+}}^2+\|\xi_0\|_{W^{3-1/q_+,q_+}}^2+\|D_tv(0)\|_{H^1}^2 +\|\pa_t\vartheta(0)\|_{H^1}^2 \\
&\quad+\|\pa_tq(0)\|_{H^0}^2+\|\pa_t\xi(0)\|_{H^{3/2-\alpha}}^2\\
  &\le \frac12(\|D_tu(0)\|_{H^1}^2+\|\pa_t\theta(0)\|_{H^1}^2+\|u_0\|_{W^{2,q_+}}^2+\|\theta_0\|_{W^{2,q_+}}^2+\|\eta_0\|_{W^{3-1/q_+,q_+}}^2+\|\pa_t\eta(0)\|_{H^{3/2-\alpha}}^2)
\end{aligned}
  \]
  for $\delta$ sufficiently small.

  Then by the Banach's fixed point theory, given $D_t^2u(0)$, $\pa_t^2\theta(0)$ and $\pa_t^2\eta(0)$ satisfying the conditions in Theorem \ref{thm:initial}, there exists a unique $(u_0, p_0, \theta_0, \eta_0, D_tu(0), \pa_tp(0), \pa_t\theta(0), \pa_t\eta(0))$ obeying the full initial conditions.
\end{proof}

\subsection{Initial Forcing Terms for Contraction}\label{sec:initial_forcing}

We give the forcing terms in the system of \eqref{eq:diff_initial1} and \eqref{eq:diff_initial2}.
\begin{equation}
  \begin{aligned}
\mathfrak{g}^8&= k\dive_{\mathcal{A}_0^1- \mathcal{A}_0^2}\nabla_{\mathcal{A}_0^1}\vartheta_0^2 + k\dive_{\mathcal{A}_0^2}\nabla_{\mathcal{A}_0^1-\mathcal{A}_0^2}\vartheta_0^2 + F^8(\eta^1_0, u^1_0, \theta^1_0)- F^8(\eta^2_0, u^2_0, \theta^2_0),\\
  \mathfrak{g}^9&=-  k [\mathcal{N}^1(0)-\mathcal{N}^2(0)] \cdot \nabla_{\mathcal{A}_0^1}  \vartheta_0^2 - k\mathcal{N}^2(0) \cdot \nabla_{\mathcal{A}_0^1-\mathcal{A}_0^2}\vartheta_0^2- \vartheta_0^2\cdot[|\mathcal{N}^1(0)|- |\mathcal{N}^2(0)|],\\
  \mathfrak{g}^8_1&= k \dive_{\pa_t\mathcal{A}^1(0)}\nabla_{\mathcal{A}^1_0}(\vartheta_0^1 -\vartheta_0^2) + k \dive_{\mathcal{A}^1_0}\nabla_{\pa_t\mathcal{A}^1(0)}(\vartheta_0^1 -\vartheta_0^2) + k\dive_{\pa_t\mathcal{A}^1(0)- \pa_t\mathcal{A}^2(0)}\nabla_{\mathcal{A}_0^1}\vartheta_0^2 \\
  &\quad + k\dive_{\mathcal{A}_0^1- \mathcal{A}_0^2}\nabla_{\pa_t\mathcal{A}^1(0)}\vartheta_0^2 + k\dive_{\mathcal{A}_0^1- \mathcal{A}_0^2}\nabla_{\mathcal{A}_0^1}\pa_t\vartheta^2(0) + k\dive_{\pa_t\mathcal{A}^2(0)}\nabla_{\mathcal{A}_0^1-\mathcal{A}_0^2}\vartheta_0^2 \\ &\quad +k\dive_{\mathcal{A}_0^2}\nabla_{\pa_t\mathcal{A}^1(0)-\pa_t\mathcal{A}^2(0)}\vartheta_0^2 +k\dive_{\mathcal{A}_0^2}\nabla_{\mathcal{A}_0^1-\mathcal{A}_0^2}\pa_t\vartheta^2(0) + \pa_t[F^8(\eta^1, u^1, \theta^1)- F^8(\eta^2, u^2, \theta^2)](0),\\
\mathfrak{g}^9_1&= -k\nabla_{\pa_t\mathcal{A}^1(0)}\vartheta_0\cdot\mathcal{N}^1(0) - k\nabla_{\mathcal{A}_0^1}\vartheta_0\cdot\pa_t\mathcal{N}^1(0) - \vartheta_0\pa_t|\mathcal{N}^1|(0)-k \pa_t[\mathcal{N}^1-\mathcal{N}^2](0) \cdot \nabla_{\mathcal{A}_0^1}  \vartheta_0^2 \\
&\quad -  k [\mathcal{N}^1(0)-\mathcal{N}^2(0)] \cdot \nabla_{\pa_t\mathcal{A}^1(0)}  \vartheta_0^2 -  k [\mathcal{N}^1(0)-\mathcal{N}^2(0)] \cdot \nabla_{\mathcal{A}_0^1}  \pa_t\vartheta^2(0)- k\pa_t\mathcal{N}^2(0) \cdot \nabla_{\mathcal{A}_0^1-\mathcal{A}_0^2}\vartheta_0^2\\
&\quad - k\mathcal{N}^2(0) \cdot \nabla_{\pa_t\mathcal{A}^1(0)-\pa_t\mathcal{A}^2(0)}\vartheta_0^2 - k\mathcal{N}^2(0) \cdot \nabla_{\mathcal{A}_0^1-\mathcal{A}_0^2}\pa_t\vartheta^2(0) - \pa_t\vartheta^2(0)\cdot[|\mathcal{N}^1(0)|- |\mathcal{N}^2(0)|]\\
&\quad - \vartheta^2_0\cdot \pa_t[|\mathcal{N}^1|- |\mathcal{N}^2|](0).
  \end{aligned}
\end{equation}


\end{appendix}


\end{document}